\documentclass[11pt]{article}

\usepackage{hyperref}
\usepackage{authblk}
\usepackage{fullpage}  % for narrow margins
\usepackage{float}
\usepackage{pb-diagram}  		% for commutative diagrams
\usepackage{esvect}
\usepackage{url}

\usepackage{algorithm, algorithmicx,algpseudocode, listings} %format of the algorithm
\usepackage{multirow} %//multirow for format of table
\usepackage{hhline}  % for \hhline in tables
\usepackage{amsmath}
\usepackage{xcolor}

\usepackage{graphicx}
\usepackage{amscd}
\usepackage{amssymb,mathrsfs}
\input{epsf.sty}
\usepackage{amsthm,amscd}
\usepackage{color}
\usepackage{latexsym}
\usepackage{epic}
\usepackage{appendix}
\usepackage{enumerate}
\usepackage{longtable}
\usepackage{lscape}
\usepackage{extarrows}
\usepackage{epstopdf}
\usepackage{verbatim}

\usepackage{listings}

\allowdisplaybreaks
\newcommand\numberthis{\addtocounter{equation}{1}\tag{\theequation}}

\newlength\myindent

\newcommand{\whcomm}[2]{{}{#2}}
\newcommand{\whfirrev}[2]{{\sf\color{purple} #1}{#2}}

\newcommand{\whsecrev}[2]{{}{#2}}

\newcommand{\whthirev}[2]{{}{#2}}

\newcommand{\kwcomm}[1]{{}}
\newcommand{\kw}[1]{{}}
\newcommand{\kwnote}[2]{{\sf\color{red} #1}{#2}}

\newcommand{\comm}[1]{}

\newcommand{\pacomm}[1]{{}}

\newcommand{\delete}[1]{}

\renewcommand{\theequation}{\arabic{equation}}

\newcommand{\Rmnum}[1]{\expandafter\@slowromancap\romannumeral #1@}
\newcommand{\inner}[3][]{{\left\langle #2,#3 \right\rangle_{#1}}}

\newtheorem{definition}{Definition}[section]
\newtheorem{theorem}{Theorem}[section]
\newtheorem{lemma}{Lemma}[section]

\newtheorem{assumption}{Assumption}[section]
\newtheorem{remark}{Remark}[section]

\newtheorem{proposition}{Proposition}[section]
\numberwithin{equation}{section}

\DeclareMathOperator{\T}{\mathrm{T}}
\DeclareMathOperator{\Hess}{\mathrm{Hess}}
\DeclareMathOperator{\grad}{\mathrm{grad}}
\DeclareMathOperator{\Exp}{\mathrm{Exp}}
\DeclareMathOperator{\id}{\mathrm{id}}

\DeclareMathOperator{\N}{\mathrm{N}}
\DeclareMathOperator{\E}{\mathrm{F}}
\DeclareMathOperator{\D}{\mathrm{D}}
\DeclareMathOperator{\dist}{\mathrm{dist}}
\DeclareMathOperator{\trace}{\mathrm{trace}}

\DeclareMathOperator{\sign}{\mathrm{sign}}

\DeclareMathOperator{\diag}{\mathrm{diag}}
\DeclareMathOperator{\St}{\mathrm{St}}

\DeclareMathOperator{\vvec}{\mathrm{vec}}

\begin{document}

\title{Riemannian Proximal Gradient Methods (extended version)\footnotetext{Authors  are listed alphabetically, and correspondence may be addressed to \texttt{wen.huang@xmu.edu.cn} (WH) and \texttt{kewei@fudan.edu.cn} (KW). \\WH was partially  supported by the Fundamental Research Funds for the Central Universities (NO. 20720190060). KW was partially  supported by the NSFC Grant 11801088 and the Shanghai Sailing Program 18YF1401600.}
}

\author[1]{Wen Huang}
\author[2]{Ke Wei}

\affil[1]{ School of Mathematical Sciences, Xiamen University, Xiamen, China.\vspace{.15cm}}
\affil[2]{School of Data Science, Fudan University, Shanghai, China.}

\maketitle

%\whfirrev{TODO: 
%\begin{itemize}
%\item modify title? abstract, introduction, and conclusion
%\item If possible, show that the Lasso problem on the unit sphere satisfies the Riemannian KL property with $\theta = 1/2$
%\item If possible, show that the SPCA problem satisfies the Riemannian KL property.
%\end{itemize}
%}{}

\begin{abstract}
In the Euclidean setting the proximal  gradient method and its accelerated variants are a class of efficient algorithms for  optimization problems with decomposable objective.
\whfirrev{}{In this paper, we develop a Riemannian proximal gradient method (RPG) and its accelerated variant (ARPG)  for similar problems but constrained on a manifold. The global convergence of RPG \whsecrev{}{is} established under mild assumptions, and the  $O(1/k)$ is also derived for RPG based on the notion of retraction convexity. If  assuming the objective function obeys the Rimannian Kurdyka-\L ojasiewicz (KL) property, it is further shown that the sequence generated by RPG converges to a single stationary point.
 As in the Euclidean setting, local convergence rate  can be established if the objective function satisfies the Riemannian KL property with an exponent. 
 Moreover, we \whsecrev{}{show} that the restriction of a semialgebraic function onto  the Stiefel manifold  satisfies the Riemannian KL property, which covers for example the well-known sparse PCA problem. }
%Empirical performance comparisons show that the proposed RPG and ARPG are competitive with existing ones.
\whsecrev{}{Numerical experiments on  random and synthetic data are conducted to test  the performance of the  proposed RPG and ARPG.}
\end{abstract}

\section{Introduction}

We consider the problem of minimizing a sum of two functions on a Riemannian manifold,
\begin{equation} \label{RPG:prob1}
\min_{x \in \mathcal{M}} F(x) = f(x) + g(x),
\end{equation}
where $\mathcal{M}$ is a finite dimensional Riemannian manifold, $f$ is  differentiable, and $g$ is continuous but could be nonsmooth. This problem arises from a wide range of applications, such as sparse principal component analysis~\cite{JoTrUd2003a,GHT2015}, sparse blind deconvolution~\cite{ZLKCPW2017}, and unsupervised feature selection~\cite{Tang2012Unsupervised}. % see Ma's paper for references

In the case when the manifold constraint is dropped (i.e., $\mathcal{M}$ is a Euclidean space), the \whcomm{}{nonsmooth} optimization problem~\eqref{RPG:prob1} have been extensively investigated and many algorithms have developed and analysed, see e.g.,~\cite{Darzentas1983Problem,Nesterov83,LL2015,GL2016,Beck2017} and references therein. Among them are a family of simple yet effective methods known as  %Nestrov's paper, book, Li and Lin's paper, and the citations on page 2 of Li and lin's paper.
proximal gradient  method and its accelerated variants. Starting from an initial guess $x_0$, the proximal gradient method updates the estimate of a minimizer via
\begin{align}
&\left\{
\begin{array}{ll}
	d_k = \arg\min_{p \in \mathbb{R}^{n}} \inner[2]{\nabla f(x_k)}{p} + \frac{L}{2} \|p\|^2_{2} + g(x_k + p), & \hbox{(Proximal mapping\footnotemark)}  \\
	x_{k+1} = x_k + d_k, & \hbox{(Update iterates)} \label{RPG:EPG}
\end{array}
\right.
\end{align}
\whcomm{}{where $\inner[2]{u}{v} = u^T v$ and $\|u\|_2^2 = \inner[2]{u}{u}$.}
The intuition behind this method is to simplify the objective function in each iteration  by replacing the differentiable term $f$ with its first order approximation around the current estimate.
\footnotetext{The commonly-used update expression is $x_{k+1}=\arg\min_x\langle\nabla f(x_k),x-x_k\rangle_2+\frac{L}{2}\|x-x_k\|_2^2+g(x)$. We reformulate it equivalently for the convenience of the Riemannian formulation given later.}
%Since the proximal mapping can be solved cheaply or even has a closed solution in many applications,
In many practical settings, the proximal mapping either has a closed-form solution or can be solved efficiently. Thus,
the algorithm has low per iteration cost and    is applicable for large-scale problems. Furthermore, under the assumptions that $f$ is convex, Lipschitz-continuously differentiable with Lipschitz constant $L$, $g$ is convex, and $F$ is coercive, the proximal gradient method converges  on the order of $O(1/k)$~\cite{Beck2009,Beck2017}.
Note that the convergence rate of the proximal gradient method is not optimal and algorithms achieving the optimal $O(1/k^2)$ \cite{Darzentas1983Problem,Nesterov83} convergence rate can be developed based on certain acceleration schemes.
%However, this convergence rate is not fully satisfactory since it can be further improved to the optimal rate $O(1/k^2)$ with an accelerated scheme under the same assumptions\footnote{The optimal convergence rate is shown in~\cite{Darzentas1983Problem,Nesterov83}.}.
One of the representative  accelerated proximal gradient methods is the fast iterative shrinkage-thresholding algorithm (FISTA, \cite{Beck2009}):
\begin{align}
&\quad\; \hbox{Initial iterate: $x_0$ and let $y_0 = x_0$, $t_0 = 1$}, \nonumber \\
&\left\{
\begin{array}{ll}
	d_k = \arg\min_{p \in \mathbb{R}^{n}} \inner[2]{\nabla f(y_k)}{p} + \frac{L}{2} \|p\|_F^2 + g(y_k + p), \\
	x_{k+1} = y_k + d_k, \\
	t_{k+1} = \frac{1 + \sqrt{4 t_k^2 + 1}}{2}, \\
	y_{k+1} = x_{k+1} + \frac{t_k - 1}{t_{k+1}} (x_{k+1} - x_k).
\end{array}
\right. \label{RPG:EAPG}
\end{align}
FISTA uses the Nesterov momentum technique to generate an auxiliary sequence $\{y_k\}$ and has been proven to converge on the order of $O(1/k^2)$  \cite{Beck2009}.

With the presence of the  manifold constraint, the \whcomm{}{nonsmooth} optimization problem~\eqref{RPG:prob1} becomes more challenging, and  only a few optimization methods have been proposed and analyzed.
\whcomm{}{When the cost function is assumed to be Lipschitz continuous, existing methods are mostly based on the notion of $\epsilon$-subgradient. The $\epsilon$-subgradient refers to the technique of using the gradients at nearby points  to estimate the subgradient at a given point, so no subgradient is computed explicitly. Specifically,} in~\cite{GH2015a} and~\cite{GH2015b}, Grohs and Hosseini come up two $\epsilon$-subgradient-based optimization methods using the line search  and trust region strategies, respectively. It is proved that any limit point of the sequence from the algorithms is a critical point. In~\cite{HUANG2013}, Huang generalizes a gradient sampling method to the Riemannian setting. This method is very  efficient for small-scale problems, but lacks convergence analysis. In~\cite{HU2017}, Hosseini and Uschmajew fill this gap and present a Riemannian gradient sampling method with convergence analysis. In~\cite{HHY2018}, Hosseini et al. propose a new Riemannian line search  method which combines the $\epsilon$-subgradient and quasi-Newton ideas.
\whcomm{}{When the cost function is further assumed to be convex, several algorithms with convergent rate analysis have been proposed.} For example, in~\cite{ZS2016}, Zhang and Sra analyze a subgradient-based Riemannian method and show that the cost function decreases to the optimal value on the order of $O(1/\sqrt{k})$. Note that the subgradient is explicitly needed in this method, which differs from the $\epsilon$-subgradient-based methods. In~\cite{FO2002}, Ferreira and Oliveira propose a Riemannian proximal point method and the $O(1/k)$ convergence rate of the method for the Hadamard manifold is established by Bento et al. in~\cite{BFM2017}. However, the Riemannian proximal point method relies on the existence of an efficient algorithm for its subproblem in~\cite[(24)]{FO2002}, and no such algorithms or instances exist as far as we know. \kwcomm{need to be careful if reviewers argue this is also the case for the algorithms in this paper} \whcomm{[I think current version is fine. We have numerical experiments that show efficient implementations for the sparse PCA applications.]}{}
\whcomm{}{In addition, when $g=0$ and $F = f$ is Lipschitz-continuously differentiable, an accelerated first order method for convex functions on Riemannian manifolds has been analyzed in~\cite{LSCCJ2017} which shows that the optimal convergence rate $O(1/k^2)$ can be achieved.}

\kwcomm{Different objective functions for these literature work, should make it clear. Maybe literature review can be done according to different objective functions.}\whcomm{[WH: Done]}{}

%In the Euclidean setting, it is known that the $\epsilon$-subgradient-based methods are slower than the proximal gradient based methods.  %Therefore, it is expected that Riemannian proximal gradient based methods that exploit the structure of the cost function are more efficient.
Note that the aforementioned algorithms have not fully exploited the split structure of the cost function in \eqref{RPG:prob1}. In contrast,
 Chen et. al~\cite{CMSZ2019} recently present a Riemannian proximal gradient method which is suitable for the case when $\mathcal{M}$ is a submanifold of a Euclidean space. The algorithm is exactly parallel to \eqref{RPG:EPG}, and its global convergence has been established. The authors show  that the norm of the search direction computed from its Riemannian proximal mapping goes to zero. Moreover, if there exists a point such that the search direction from this point vanishes, then this point must be a critical point. Numerical experiments show that the proposed method is more efficient than existing methods based on the conventional constrained optimization framework such as SOC~\cite{LO2014} and PAMAL~\cite{CHY2016}.  Later on, Huang and Wei~\cite{HuaWei2019} show that any limit point of the sequence generated by the Riemannian proximal gradient method in~\cite{CMSZ2019} is indeed a critical point. Furthermore, they propose a Riemannian version of FISTA  with safeguard which exhibits the accelerated behavior over the Riemannian proximal gradient method. Nevertheless, no convergence rate analysis is presented there.

\whcomm{Our contribution:}{
The main contributions of this paper are summarized as follows. A Riemannian proximal gradient method (RPG) and its accelerated variant (ARPG)  are proposed and studied. These methods are based on a different Riemannian proximal mapping, compared to those in \cite{CMSZ2019,HuaWei2019}, which allows them to work for generic manifolds. % as well as to be amenable to the convergence rate analysis.
It is proved that any accumulation point of RPG is a critical point under mild assumptions. \whfirrev{}{Based on a notion of retraction convexity on  Riemannian manifolds, we show that RPG has a $O(1/k)$ convergence rate. Furthermore, it is shown that the sequence generated by RPG converges to a single stationary point if the objective function satisfies the Riemannian KL property and local convergence rate can be given if the KL exponent is known. In particular, we have proved that the restriction of a semialgebraic function onto  the Stiefel manifold  satisfies the Riemannian KL property, and this result applies to the well-known sparse PCA problem. }%Though the Riemannian KL property has been used to analyze an abstract subgradient method for manifold optimization in \cite{Hosseini2017}, 
%To the best of our knowledge, this is the first work which gives a verifiable condition for the Riemannian KL property of a nonsmooth function on the  Stiefel manifold which covers   the practical sparse PCA problem.}
%while the Riemannian FISTA method has a $O(1/k^2)$ convergence rate. To the best of our knowledge, these are the first class of Riemannian proximal gradient methods that possess a convergence rate analysis. 
In addition, a practical Riemannian proximal gradient method, which shares the features of the Riemannian proximal gradient method (global convergence under mild conditions) and the Riemannian FISTA method (fast convergence \whfirrev{}{empirically}%under stronger conditions
), is derived. %The optimization problems from sparse principal component analysis are used to demonstrate the performance of the Riemannian proximal gradient methods.
We then examine the performance of the proposed methods through two different optimization problems for sparse principle component analysis.
}

\whfirrev{}{
%Note that there are limited results about the Riemannian KL property for manifold optimization algorithms.
%[Theorem~11]
The Riemannian KL property is overall similar to the Euclidean KL, but with related notations defined on a manifold, see Definition~\ref{RPG:def:RKL}. To the best of our knowledge, 
the Riemannian KL property was first defined by Kurdyka in~\cite{KMP2000} for analytic manifolds and analytical functions. In~\cite{Lageman2007}, it was extended for analytic manifolds and differentiable $\mathcal{C}$-functions in an analytic-geometric category. 
In~\cite{BDLS2007}, a verifiable condition for a Riemannian KL property is given when the function on a manifold is differentiable, the manifold is an embedded submanifold of $\mathbb{R}^n$, and the Riemannian metric is endowed from the Euclidean metric. 
%In~\cite{SU2015}, the Riemannian KL property was used to analyze a differentiable function on a closure of the manifold of fixed rank matrices.
In~\cite{BBNOS2019}, the Riemannian KL property was used to analyze a Riemannian steepest descent method for computing a Riemannian center of mass on Hadamard manifold. The results about Riemnanian KL property for nonsmooth optimization on manifolds are still limited.
In~\cite{BCO2011}, a Riemannian generalization of KL property for nonsmooth functions is given and a Riemannian proximal point method is analyzed using the Riemannian KL property. In~\cite{Hosseini2017}, the Riemannian KL property is used to analyze an abstract subgradient method for manifold optimization. %However, no verifiable condition and instances are given therein.
}

%\whcomm{}{
%It is worth noting that the difficulty in the convergence rate analysis for the Riemannian proximal gradient methods is the lack of linearity on a generic manifold. In a Euclidean space, it is known that the vector addition/substraction holds in the sense that $\vv{xz} - \vv{xy}$ equals $\vv{yz}$, where $\vv{ab}$ denotes the vector from $a$ to $b$. However, such property does not hold on a manifold in general. %This phenomenon is illustrated in Figure~\whcomm{TODO}{}.
%In this paper we mitigate this issue by not requiring the equality holds exactly but only approximately, see Assumptions~\ref{RPG:as6} and~\ref{RPG:as16} for details. \kwnote{I will suggest removing this paragraph}{}
%}

\whcomm{}{
This paper is organized as follows.
Notation and preliminaries on manifolds are given in Section~\ref{RPG:sect:NotationPreliminaries}. 
The Riemannian proximal gradient method together with its convergence analyses, are presented in Section~\ref{RPG:sect:RPG}. The accelerated Riemannian proximal gradient method and a practical variant  is described in Section~\ref{sec:ARPG}.
Numerical experiments are reported in Section~\ref{RPG:sect:NumExp}. This paper is concluded with potential future directions in Section~\ref{RPG:sect:Con}.
}

%\whfirrev{[I worked from here:]}{}

\section{Notation and Preliminaries on Manifolds} \label{RPG:sect:NotationPreliminaries}

The Riemannian concepts of this paper follow from the standard literature, e.g.,~\cite{Boo1986,AMS2008} and the  related notation  follows from~\cite{AMS2008}. A Riemannian manifold $\mathcal{M}$ is a manifold endowed with a Riemannian metric $(\eta_x, \xi_x) \mapsto \inner[x]{\eta_x}{ \xi_x} \in \mathbb{R}$, where $\eta_x$ and $\xi_x$ are tangent vectors in the tangent space of $\mathcal{M}$ at $x$. The induced norm in the tangent space at $x$ is denoted by $\|\cdot\|_x$ \whfirrev{}{or $\|\cdot\|$ when the subscript is clear from the context.}
The tangent space of the manifold $\mathcal{M}$ at $x$ is denoted by $\T_x \mathcal{M}$, and the tangent bundle, which is the set of all tangent vectors, is denoted by $\T \mathcal{M}$. A vector field is a function from the manifold to its tangent bundle, i.e., $\eta:\mathcal{M} \rightarrow \T \mathcal{M}: x\mapsto \eta_x$. \whfirrev{}{An open ball on a tangent space is denoted by $\mathcal{B}(\eta_x, r) = \{\xi_x \in \T_{x} \mathcal{M} \mid \|\xi_x - \eta_x\|_x < r\}$. % and the its closure is a closed ball denoted by $\overline{\mathcal{B}}(\eta_x, r) = \{\xi_x \in \T_{x} \mathcal{M} \mid \|\xi_x - \eta_x\|_x \leq r\}$. 
An open ball on the manifold is denoted by $\mathbb{B}(x, r) = \{ y \in \mathcal{M} \mid \dist(y, x) < r \}$, where $\dist(x, y)$ denotes the distance between $x$ and $y$ on $\mathcal{M}$. %and its closure is a closed ball denoted by $\overline{\mathbb{B}}(x, r) = \{ y \in \mathcal{M} \mid \dist(y, x) \leq r \}$.
}

A retraction is a \whfirrev{}{smooth ($C^\infty$)} mapping from the tangent bundle to the manifold such that (i)~$R(0_x) = x$ for all $x \in \mathcal{M}$, where $0_x$ denotes the origin of $\T_x \mathcal{M}$, and (ii) $\frac{d}{d t} R(t \eta_x) \vert_{t = 0} = \eta_x$ for all $\eta_x \in \T_x \mathcal{M}$. The domain of $R$ does not need to be the entire tangent bundle. However, it is usually the case in practice, and in this paper we assume  $R$  is always well-defined. Moreover, $R_x$ denotes the restriction of $R$ to $\T_x \mathcal{M}$, \whcomm{}{see Figure~\ref{fig:RandVT} for an illustration of  $R_x$}. \whfirrev{}{For any $x \in \mathcal{M}$, there always exists a neighborhood of $0_x$ such that the mapping $R_x$ is a diffeomorphism in the neighborhood.} An important retraction is the exponential mapping, denoted by $\mathrm{Exp}$, satisfying $\mathrm{Exp}_x(\eta_x) = \gamma(1)$, where $\gamma(0) = x$, $\gamma'(0) = \eta_x$, and $\gamma$ is the geodesic passing through $x$. \whcomm{}{In a Euclidean space, the most common retraction is the exponential mapping given by addition $\Exp_x(\eta_x) = x + \eta_x$.} 
%\whfirrev{}{If the manifold is an embedded submanifold of $\mathbb{R}^n$, then the retraction $R_x(\eta_x)$ is an first order approximation of $x + \eta_x$, i.e.,
%\begin{equation} \label{RPG:e72}
%\|R_x(\eta_x) - x - \eta_x\|_2 = O(\|\eta_x\|_2^2),
%\end{equation}
%where $O(t)$ denotes a function of $t$ satisfying $\limsup_{t \rightarrow 0} O(t) / t < \infty$.
%}

A vector transport $\mathcal{T}: \T \mathcal{M} \oplus \T \mathcal{M} \rightarrow \T \mathcal{M}: (\eta_x, \xi_x) \mapsto \mathcal{T}_{\eta_x} \xi_x$ associated with a retraction $R$ is a \whfirrev{}{smooth ($C^\infty$)} mapping such that, for all $(x, \eta_x)$ in the domain of $R$ and all $\xi_x \in \T_x \mathcal{M}$, it holds that (i) $\mathcal{T}_{\eta_x} \xi_x \in \T_{R(\eta_x)} \mathcal{M}$ and (ii) $\mathcal{T}_{\eta_x}$ is a linear map, \whcomm{}{see Figure~\ref{fig:RandVT} for an illustration of a vector transport $\mathcal{T}_{\eta_x}$}. If $R_x^{-1}(y)$ is well-defined for $x, y \in \mathcal{M}$, then the vector transport $\mathcal{T}_{R_x^{-1}(y)}$ is also denoted by $\mathcal{T}_{x \rightarrow y}$.
An isometric vector transport $\mathcal{T}_{\mathrm{S}}$ additionally satisfies $\inner[R_x(\eta_x)]{\mathcal{T}_{\mathrm{S}_{\eta_x}} \xi_x}{\mathcal{T}_{\mathrm{S}_{\eta_x}} \zeta_x} = \inner[x]{\xi_x}{\zeta_x}$, for any $\eta_x, \xi_x, \zeta_x \in \T_x \mathcal{M}$.
\whcomm{}{An important vector transport is the parallel translation, denoted  $\mathcal{P}$. The basic idea behind the parallel translation is to move a tangent vector along a given curve on a manifold ``parallelly''. We refer to~\cite{AMS2008} for its rigorous definition. Note that parallel translation is an isometric vector transport.}
The vector transport by differential retraction $\mathcal{T}_R$ is defined by $\mathcal{T}_{R_{\eta_x}} \xi_x = \frac{d}{d t} R_{x}(\eta_x + t \xi_x) \vert_{t = 0}$. The adjoint operator of a vector transport $\mathcal{T}$, denoted by $\mathcal{T}^\sharp$, is a vector transport satisfying $\inner[y]{\xi_y}{\mathcal{T}_{\eta_x} \zeta_x} = \inner[x]{\mathcal{T}_{\eta_x}^\sharp \xi_y}{\zeta_x}$ for all $\eta_x, \zeta_x \in \T_x \mathcal{M}$ and $\xi_y \in \T_y \mathcal{M}$, where $y = R_x(\eta_x)$. \whcomm{}{The inverse operator of a vector transport, denoted $\mathcal{T}^{-1}$, is a vector transport satisfying $\mathcal{T}_{\eta_x}^{-1} \mathcal{T}_{\eta_x} = \id$ for all $\eta_x \in \T_x \mathcal{M}$, where $\id$ is the identity operator. In the Euclidean setting, a vector transport $\mathcal{T}_{\eta_x}$ for any $\eta_x \in \T_x \mathcal{M}$ can be represented by a matrix  (the commonly-used vector transport is the identity matrix). Then the adjoint and inverse operators of a vector transport are given by the transpose and inverse of the corresponding matrix, respectively. \kw{not clear to me what last sentence means, may just remove it to avoid confusion}} \whcomm{WH10[Done, symmetry to transpose]}{}

\begin{figure}
\centering
\scalebox{0.7}{
  \setlength{\unitlength}{1bp}%
  \begin{picture}(219.70, 178.15)(0,0)
  \put(0,0){\includegraphics{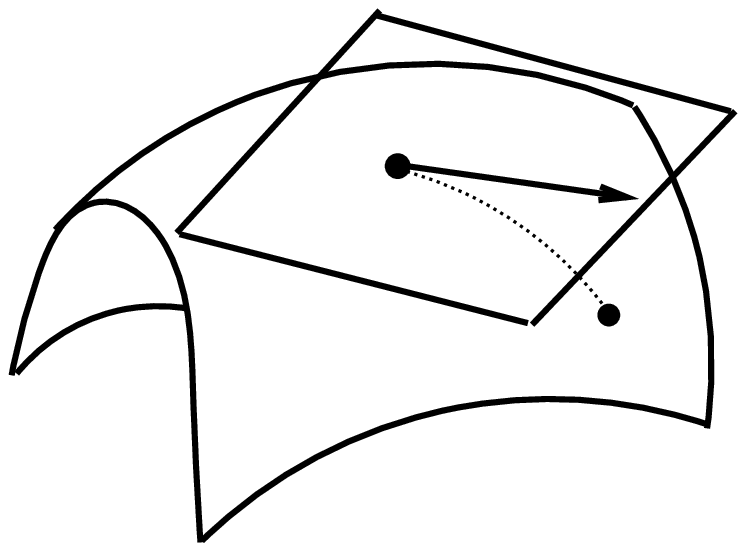}}
  \put(74.18,33.70){\fontsize{14.23}{17.07}\selectfont $\mathcal{M}$}
  \put(102.98,103.21){\fontsize{14.23}{17.07}\selectfont $x$}
  \put(169.61,112.53){\fontsize{14.23}{17.07}\selectfont $\eta_x$}
  \put(118.55,161.37){\fontsize{14.23}{17.07}\selectfont $\mathrm{T}_x \mathcal{M}$}
  \put(163.50,56.51){\fontsize{14.23}{17.07}\selectfont $R_x(\eta_x)$}
  \end{picture}%
  }
 \hspace{3em}
\scalebox{.7}{
  \begin{picture}(220.04, 178.15)(0,0)
  \put(0,0){\includegraphics{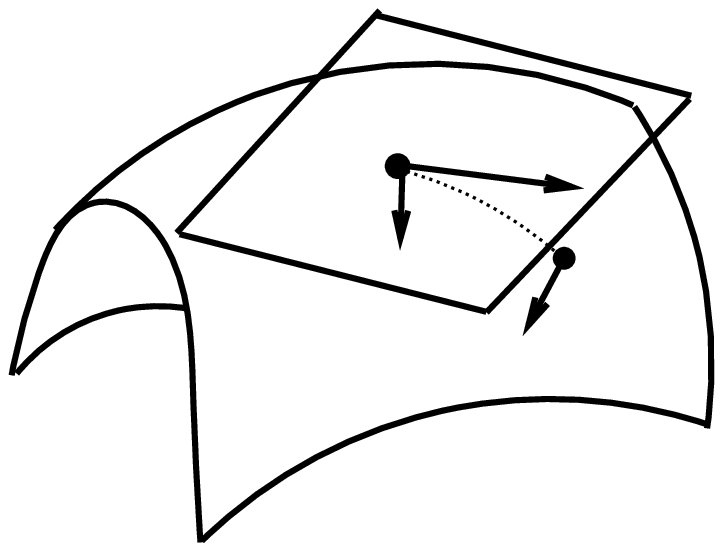}}
  \put(74.18,33.70){\fontsize{14.23}{17.07}\selectfont $\mathcal{M}$}
  \put(102.98,103.21){\fontsize{14.23}{17.07}\selectfont $x$}
  \put(167.61,113.53){\fontsize{14.23}{17.07}\selectfont $\eta_x$}
  \put(118.55,161.37){\fontsize{14.23}{17.07}\selectfont $\mathrm{T}_x \mathcal{M}$}
  \put(164.78,73.51){\fontsize{14.23}{17.07}\selectfont $R_x(\eta_x)$}
  \put(122.14,88.72){\fontsize{14.23}{17.07}\selectfont $\xi_x$}
  \put(126.29,54.27){\fontsize{14.23}{17.07}\selectfont $\mathcal{T}_{\eta_x} \xi_x$}
  \end{picture}%
}
  \caption{(Left) Retraction; (Right) Vector transport \kw{a) left: $\eta\rightarrow\eta_x$; b) different line width and fontsize in the two plots} \whcomm{WH10[Done]}{}}
  \label{fig:RandVT}
\end{figure}

\whfirrev{}{
The Riemannian gradient of a function $h:\mathcal{M} \rightarrow \mathbb{R}$%is denoted by $\grad h(x)$, and
, denote  $\grad h(x)$, is the unique tangent vector satisfying:
\begin{equation*}
\D h(x) [\eta_x] = \inner[x]{\eta_x}{\grad h(x)}, \forall \eta_x \in \T_x \mathcal{M},
\end{equation*}
where $\D h(x) [\eta_x]$ denotes the directional derivative along the direction $\eta_x$.
The Riemannian Hessian of $h$ at $x$, denoted by $\Hess h(x)$, is a linear operator on $\T_x \mathcal{M}$ satisfying
\[
\Hess h(x) [\eta_x] = \overline{\nabla}_{ \eta_x } \grad h(x), \qquad \forall \eta_x \in \T_x \mathcal{M},
\]
where $\Hess h(x) [\eta_x]$ denotes the action of $\Hess h(x)$ on a tangent vector $\eta_x \in \T_x \mathcal{M}$, and $\overline{\nabla}$ denotes the Riemannian affine connection. Roughly speaking, an affine connection generalizes the concept of a directional derivative of a vector field and we refer to ~\cite[Section~5.3]{AMS2008} for its rigorous definition. 
}

%The Riemannian Hessian of $h$ at $x$ is denoted by $\Hess h(x)$. 
%The action of $\Hess h(x)$ on a tangent vector $\eta_x \in \T_x\mathcal{M}$ is denoted by $\Hess h(x) [\eta_x]$.
If $h$ is  Lipschitz continuous but not differentiable, then the Riemannian version of generalized subdifferential defined in~\cite{HHY2018} is used. Specifically, since $\hat{h}_x = h \circ R_x$ is a Lipschitz continuous function defined on a Hilbert space $\T_x \mathcal{M}$, the Clarke generalized directional derivative at $\eta_x \in \T_x \mathcal{M}$, denoted by $\hat{h}_x^\circ(\eta_x; v)$, is defined by $\hat{h}_x^\circ(\eta_x; v) = \lim_{\xi_x \rightarrow \eta_x} \sup_{t \downarrow 0} \frac{\hat{h}_x(\xi_x + t v) - \hat{h}_x(\xi_x)}{t}$, where $v \in \T_x \mathcal{M}$. The generalized subdifferential of $\hat{h}_x$ at $\eta_x$, denoted  $\partial \hat{h}_x(\eta_x)$, is defined by $\partial \hat{h}_x(\eta_x) = \{\eta_x \in \T_x \mathcal{M} \mid \inner[x]{\eta_x}{v} \leq \hat{h}_x^\circ(\eta_x; v) \hbox{ for all } v \in \T_x \mathcal{M}\}$. The Riemannian version of the Clarke generalized direction derivative of $h$ at $x$ in the direction $\eta_x \in \T_x \mathcal{M}$, denoted  $h^\circ (x; \eta_x)$, is defined by $h^\circ (x; \eta_x) = \hat{h}_x^\circ (0_x; \eta_x)$. The generalized subdifferential of $h$ at $x$, denoted  $\partial h(x)$, is defined as $\partial h(x) = \partial \hat{h}_x(0_x)$. Any tangent vector $\xi_x \in \partial h(x)$ is called a \whfirrev{}{Riemannian} subgradient of $h$ at $x$.

\whfirrev{}{
A vector field $\eta$ is called Lipschitz continuous if there exist \whsecrev{}{a positive injectivity radius $i(\mathcal{M})$ and} a positive constant $L_v$ such that for all $x, y \in \mathcal{M}$ with $\dist(x, y) < i(\mathcal{M})$, it holds that
\begin{equation} \label{RPG:VFLipCon}
\| \mathcal{P}_{\gamma}^{0 \leftarrow 1} \eta_y - \eta_x \|_x \leq L_v \dist(y, x),
\end{equation}
where $\gamma$ is a geodesic with $\gamma(0) = x$ and $\gamma(1) = y$, \whsecrev{}{the injectivity radius $i(\mathcal{M})$ is defined by} $i(\mathcal{M}) = \inf_{x \in \mathcal{M}} i_x$ and $i_x = \sup \{\epsilon > 0 \mid \Exp_x\vert_{\mathbb{B}(x, \epsilon)} \hbox{ is a diffeomorphism} \}$. \whsecrev{}{Note that for any compact manifold, the injectivity radius is positive~\cite[Lemma~6.16]{Lee2018} .} A vector field $\eta$ is called locally Lipschitz continuous if for any compact subset $\bar\Omega$ of $\mathcal{M}$, there exists a positive constant $L_v$ such that for all $x, y \in \bar\Omega$ with $\dist(x, y) < i({\bar\Omega})$, inequality~\eqref{RPG:VFLipCon} holds. A function on $\mathcal{M}$ is called (locally) Lipschitz continuous differentiable if the vector field of its gradient is (locally) Lipschitz continuous.
}

\whfirrev{}{
Let $\tilde{\Omega}$ be a subset of $\mathcal{M}$. If there exists a positive constant $\varrho$  such that, for all $y \in \tilde{\Omega}, \tilde{\Omega} \subset R_y(\mathcal{B}(0_y, \varrho))$ and $R_y$ is a diffeomorphism on $\mathcal{B}(0_y, \varrho)$, \kwcomm{not very clear what it means here}\whcomm{[It is a definition of a set. For any two points in this set, the inverse of the retraction is well-defined.]}{} then we call \whfirrev{}{$\tilde{\Omega}$} a totally retractive set with respect to $\varrho$. 
The existence of $\tilde{\Omega}$ can be shown along the lines of~\cite[Theorem~3.7]{dC92}, i.e., given any $x \in \mathcal{M}$, there exists a neighborhood of $x$ which is a totally retractive set.
}

In a Euclidean space, the Euclidean metric is denoted by $\inner[\E]{\eta_x}{ \xi_x}$, where $\inner[\E]{\eta_x}{ \xi_x}$  is  equal to the summation of the entry-wise products of $\eta_x$ and $\xi_x$, such as $\eta_x^T \xi_x$ for vectors and $\trace(\eta_x^T \xi_x)$ for matrices. The induced Euclidean norm is denoted by $\|\cdot\|_{\mathrm{F}}$. \whcomm{}{For any matrix $M$, the spectral norm is denoted by $\|M\|_2$. For any vector $v \in \mathbb{R}^n$, the $p$-norm, denoted $\|v\|_p$, is equal to $\left(\sum_{i = 1}^n |v_i|^p \right)^{\frac{1}{p}}$.}
\whsecrev{}{In this paper, $\mathbb{R}^n$ does not only refer to a vector space, but also can refer to a matrix space or a tensor space.}

\kwcomm{maybe more concise to use notation without subscript to denote the Euclidean case} \whcomm{[I preferred to omit the subscript for the commonly-used norm in this paper, which is the induced norm. But I always include subscript in this version.]}{}

%In this notes, we only consider Euclidean metric for simplicity. The Riemannian gradient with respect to the Euclidean metric is denoted by $\grad f(x)$, which is equal to $P_{\T_x \mathcal{M}} \nabla f(x)$.
%%%%%%%%%%%%%%%
%%%%%%%%%%%%%%%
%%%%%%%%%%%%%%%
\section{A Riemannian Proximal Gradient Method} \label{RPG:sect:RPG}

The Riemannian proximal gradient method proposed in this paper is stated in Algorithm~\ref{RPG:a1}.  %Due to the lack of linearity of a generic manifold, it is not easy to define a computationally efficient Riemannian proximal mapping. Here, we consider two generalizations~\eqref{RPG:subproblem1} and~\eqref{RPG:subproblem2} in which the former one has an extra constraint and the latter one does not require to find a global minimizer. More discussions about the Riemannian proximal mapping are given in Section~\ref{RPG:sect:subproblem}. %The addition $x_k + \eta$ is replaced by $R_{x_k}(\eta)$. %Moreover, a radius parameter $\rho$ is used to control domain of the subproblem. The motivation for $\rho$ is discussed later in Section~\ref{RPG:sect:subproblem}. Note that if $\rho = \infty$, the manifold is the Euclidean space, and the retraction is the exponential mapping, then Algorithm~\ref{RPG:a1} reduces to the proximal gradient method in~\eqref{RPG:EPG}.
In each iteration, the algorithm first computes a search direction by solving a proximal subproblem on the tangent space at the current estimate and then a new estimate is obtained through the application of the retraction. Steps~\ref{RPG:a1:st1} and~\ref{RPG:a1:st2} are a generalization of the proximal mapping and the iterate update formula in~\eqref{RPG:EPG}, repectively.
The discussion on solving the Riemannian proximal mapping \eqref{RPG:subproblem2} will be deferred to Section~\ref{RPG:sect:subproblem}, after the presentation of the convergence analysis. \kwcomm{not sure should point out the difference with Ma's method here or in Section~\ref{RPG:sect:subproblem}} \whcomm{[I prefer to discuss Ma's method later in Section 3.3, because I prefer to make each section independent, i.e., this section just focuses on our method.]}{}

\begin{algorithm}[H]
\caption{Riemannian Proximal Gradient Method (RPG)}
\label{RPG:a1}
\begin{algorithmic}[1]
\Require Initial iterate $x_0$; a positive constant $\tilde{L} \whfirrev{}{> L}$; % and $\rho$; %\whcomm{[TODO make it to be an assumption]}{a bounded below function $f$}; %, $\sigma \in (0, 1)$, $\nu \in (0, 1)$.
%\State $t_{-1} = 0$, $t_0 = 1$, $\xi_{0} = 0$, $y_{-1} = x_{-1} = x_0$;
\For {$k = 0, \ldots$}
%\State Set $\tilde{\xi}_{k - 1} = \frac{t_{k - 1} - 1}{t_{k}} \xi_{k - 1}$ and $y_{k} = R_{x_k}\left( \tilde\xi_{k-1}\right)$
\State Let $\ell_{x_k}(\eta) = \inner[x_k]{\grad f(x_k)}{\eta} + \frac{\tilde{L}}{2} \|\eta\|_{x_k}^2 + g(R_{x_k}(\eta))$;
\State \label{RPG:a1:st1}
Find $\eta_{x_k}^*{\in\T_{x_k}\mathcal{M}}$ such that \kwcomm{(3.1) not right way to express a stationary point} \whcomm{[Done]}{}
\begin{align}
\eta_{x_k}^* \hbox{ is a \whfirrev{}{stationary point} of } \ell_{x_k}(\eta) ~{\hbox{ on }\T_{x_k}\mathcal{M}} \hbox{ and }  \ell_{x_k}(0) \geq \ell_{x_k}(\eta_{x_k}^*); \label{RPG:subproblem2}
\end{align}
%
%or finding any local minimizer
%\begin{align}
%\eta_{x_k}^* =& \arg\min_{\eta \in \mathcal{B}(0_x, \rho)} \ell_{x_k}(\eta);
%\end{align}
\State \label{RPG:a1:st2} $x_{k+1} = R_{x_k}(\eta_{x_k}^*)$;
\EndFor
\end{algorithmic}
\end{algorithm}

%\whcomm{[Notes: If it holds that $|F(x + \eta_x) - F(R_x(\eta_x))| < \epsilon \|\eta_x\|^2$ for a sufficiently small $\epsilon$, then $\eta_x$ in Algorithm~\ref{RPG:a1} can be replaced by
%\begin{equation*}
%\eta_{x_k} = \arg\min_{\eta \in \T_{x_k} \mathcal{M}} \inner[]{\grad f(x_k)}{\eta} + \frac{\tilde{L}}{2} \|\eta\|^2 + g(x_k + \eta)
%\end{equation*}
%and the local convergence analysis follows similarily. But the condition is too strong anyway...
%]}{}

%\whfirrev{[Lipschitz con $\rightarrow$ locally Lipschitz con?]}{}

\subsection{Global Convergence Analysis}

In the Euclidean setting, the global convergence of the proximal gradient method is established under the assumptions that $f$ is $L$-smooth and $F$ is coercive,
\whcomm{}{
where a continuously differentiable function $f:\mathbb{R}^n \rightarrow \mathbb{R}$ is called $L$-smooth if
$$
f(x) \leq f(y) + \inner[\E]{x - y}{\nabla f(y)} + \frac{L}{2} \|x - y\|_{\E}^2 \hbox{ for all $x, y \in \mathbb{R}^n$,}
$$
and $F$ is called coercive if $F(x) \rightarrow \infty$ as $\|x\| \rightarrow \infty$, see the definitions in e.g.,~\cite{Beck2017}.
}
Similar assumptions will be made for the Riemannian setting, where the coercive property is replaced by compactness of the sublevel set.
\begin{assumption} \label{RPG:as10}
The function $F$ is bounded from below and the sublevel set $\Omega_{x_0} = \{x \in \mathcal{M} \mid F(x) \leq F(x_0)\}$ is compact.
\end{assumption}

In Definition~\ref{RPG:def:Lsmooth}, we generalize the $L$-smoothness to the Riemannian setting and define a notion of $L$-retraction-smooth.
\kwcomm{history about retraction-smooth} \whcomm{[Done]}{}
\begin{definition} \label{RPG:def:Lsmooth}
A function $h:\mathcal{M} \rightarrow \mathbb{R}$ is called $L$-retraction-smooth with respect to a retraction $R$ in $\mathcal{N} \subseteq \mathcal{M}$ if for any $x \in \mathcal{N}$ and any $\mathcal{S}_x \subseteq \T_x \mathcal{M}$ such that $R_x(\mathcal{S}_x)\subseteq \mathcal{N}$, we have that %$q_x = h \circ R_x$ satisfies
\begin{equation} \label{RPG:e65}
\whfirrev{}{
h (R_x(\eta)) \leq h(x) + \inner[x]{\grad h(x)}{\eta} + \frac{L}{2} \|\eta\|_x^2,\quad \forall \eta \in \mathcal{S}_x.
}
%q_x(\eta) \leq q_x(\xi) + \inner[x]{\grad q_x(\xi)}{\eta - \xi} + \frac{L}{2} \|\eta - \xi\|_x^2\;\; \forall \eta, \xi \in \mathcal{S}_x.
\end{equation}
\end{definition}
%\kwcomm{make sure this remark is right} \whcomm{[It is right.]}{}
%Noticing that $\grad q_x(0)=\grad h(x)$, it follows immediately from Definition \ref{RPG:def:Lsmooth} that
%\begin{align} \label{RPG:e65}
%h(R_x(\eta))\leq h(x)+\inner[x]{\grad h(x)}{\eta}+\frac{L}{2} \|\eta\|_x^2.
%\end{align}
\whcomm{[WH10]}{A stronger version of~\eqref{RPG:e65}, which assumes 
\begin{equation} \label{RPG:e92}
\left|h(R_x(\eta)) - h(x)-\inner[x]{\grad h(x)}{\eta} \right| \leq \frac{L}{2} \|\eta\|_x^2,
\end{equation}
has been used in~\cite[Assumption~2.6]{BAC2018}.}
\whcomm{}{In addition, if we choose the retraction to be the exponential mapping, \whcomm{WH10[Definition~\ref{RPG:def:Lsmooth}]}{inequality~\eqref{RPG:e65}} also implies  $h$ is  geodesically $L$-smooth~\cite{ZS2016,LSCCJ2017}, that is,
\begin{equation} \label{RPG:e48}
h(y)\leq h(x)+\inner[x]{\grad h(x)}{\Exp_x^{-1}(y)}+\frac{L}{2} \|\Exp_x^{-1}(y)\|_x^2.
\end{equation}
%Note that the generalization in Definition~\ref{RPG:def:Lsmooth} implies~\eqref{RPG:e48} if the retraction is the exponential mapping.
}

\begin{assumption} \label{RPG:as3}
The function $f$ is $L$-retraction-smooth with respect to the retraction $R$ in the sublevel set $\Omega_{x_0}$.
\end{assumption}

\whfirrev{}{
It has been shown in~\cite[Lemma~2.7]{BAC2018} that if $\mathcal{M}$ is a compact Riemannian submanifold of a Euclidean space $\mathbb{R}^n$, the retraction $R$ is globally defined, $f:\mathbb{R}^n \rightarrow \mathbb{R}$ is $L$-smooth in the convex hull of $\mathcal{M}$, then the inequality~\eqref{RPG:e92} holds with $\mathcal{N} = \mathcal{M}$. By following the lines in the proofs of~\cite[Lemma~2.7]{BAC2018}, one can show that the conclusion still holds if the metric of $\mathcal{M}$ is not endowed from the Euclidean space.\footnote{\whfirrev{}{Such result can be obtained by noting (i) $\D f(x)[\eta] = \inner[\E]{P_{\T_x \mathcal{M}} \nabla f(x)}{\eta} = \inner[x]{\grad f(x)}{\eta}$ for~\cite[(B.2)]{BAC2018}, and (ii) there exists a constant $\alpha > 0$ such that $\|\eta\|_{\E} \leq \alpha \|\eta\|_x$ for all $x \in \mathcal{M}$ by smoothness of the Riemannian metric and compactness of $\mathcal{M}$. }}
}

Lemma~\ref{RPG:le3} shows that RPG is a descent algorithm. \whsecrev{}{It  is  worth noting  that  the key difference between Lemma~\ref{RPG:le3} and  the descent property established in~\cite[Lemma~5.2]{CMSZ2019} for the Riemannian proximal gradient method therein  is that Lemma~\ref{RPG:le3} does not require convexity of $g$ whereas~\cite[Lemma~5.2]{CMSZ2019} does. }
\begin{lemma} \label{RPG:le3}
Suppose Assumption~\ref{RPG:as3} holds. Then the sequence $\{x_k\}$ generated by Algorithm~\ref{RPG:a1} satisfies
\begin{equation} \label{RPG:e9}
F(x_k) - F(x_{k+1}) \geq \beta \|\eta_{x_k}^*\|_{x_k}^2,
\end{equation}
where $\beta = (\tilde{L} - L) / 2$.
\end{lemma}
\begin{proof}
By the definition of $\eta_{x_k}^*$ and the $L$-retraction-smooth of $f$, we have
\begin{align}
F(x_k) &= f(x_k) + g(x_k) \geq f(x_k) + \inner[x_k]{\grad f(x_k)}{\eta_{x_k}^*} + \frac{\tilde{L}}{2} \|\eta_{x_k}^*\|_{x_k}^2 + g(R_{x_k}(\eta_{x_k}^*)) \nonumber \\
&\geq \frac{\tilde{L} - L}{2} \|\eta_{x_k}^*\|_{x_k}^2 + f(R_{x_k}(\eta_{x_k}^*)) + g(R_{x_k}(\eta_{x_k}^*)) = F(x_{k+1}) + \frac{\tilde{L} - L}{2} \|\eta_{x_k}^*\|_{x_k}^2, \nonumber
\end{align}
which completes the proof.
\end{proof}

Lemma~\ref{RPG:le10} will be used for the global convergence analysis in Theorem~\ref{RPG:globaltheo}.
\kwcomm{a) $\kappa$ shouldn't rely on $x$; b) add more details in the derivation so that non-experts can follow easily.}

\begin{lemma} \label{RPG:le10}
Let $\xi$ be a continuous vector field. Then  $\lim_{\eta_x \rightarrow 0} \|\xi_y - \mathcal{T}_{\eta_x}^{- \sharp} \xi_x\|_y = 0$, where $y=R_x(\eta_x)$. %\whfirrev{Can be removed }{Furthermore, if $\xi$ is a locally Lipschitz continuous vector field, then for any compact set $\bar\Omega \subset \mathcal{M}$, there exist positive constant $\delta$ and $L_c$ such that $\|\xi_y - \mathcal{T}_{\eta_x}^{- \sharp} \xi_x\|_y \leq L_c \|\eta_x\|_x$ for any $x$ and $y$ satisfying $\dist(x, y) < \delta$.}
\end{lemma}
\begin{proof}
\whfirrev{}{
Define function $h:\T \mathcal{M} \rightarrow \T \mathcal{M}: \eta_x \mapsto \xi_y - \mathcal{T}_{\eta_x}^{- \sharp} \xi_x$. By definition of vector transport and its adjoint operator, we have that $\mathcal{T}^{-\sharp} \in C^1$ and $\mathcal{T}_{0_x}^{-\sharp} = \mathrm{id}$, and therefore $\mathcal{T}^{-\sharp}$ is a vector transport, where $\mathrm{id}$ denotes the identity operator. It follows that $h(0_x) = 0$ and $h$ is a continuous function. Therefore, $\lim_{\eta_x \rightarrow 0} \|\xi_y - \mathcal{T}_{\eta_x}^{- \sharp} \xi_x\|_y = 0$.}

\end{proof}

We are now in the position to give a global convergence analysis of Algorithm~\ref{RPG:a1}.\kwcomm{a) scale problem we once mentioned? b) assumptions seem weaker than in Ma's result}

\begin{theorem} \label{RPG:globaltheo}
If $\eta_{x_k}^* = 0$, then $x_k$ is a stationary point.
Suppose Assumptions~\ref{RPG:as10} and~\ref{RPG:as3} hold.
Then the sequence $\{x_k\}$ has at least one accumulation point.
Let $x_*$ be any accumulation point of the sequence $\{x_k\}$.
Then $x_*$ is a stationary point.
Furthermore, Algorithm~\ref{RPG:a1} returns $x_k$ satisfying $\|\eta_{x_k}^*\|_{x_k} \leq \epsilon$ in at most $(F(x_0) - F(x_*)) / (\beta \epsilon^2)$ iterations.
\end{theorem}
\begin{proof}
If $\eta_{x_k}^* = 0$, then we have $0 \in \partial F(x_k)$, which is the first-order necessary condition for the optimality of \eqref{RPG:prob1}.  %If $\eta_{x_k}^* = 0$, then we have $0 \in \partial F(x_k)$\whcomm{[TODO, here, a generalization of subdifferential to the Riemannian setting is needed. See the end of page 599 of~\cite{HHY2018}.]}{}, which is the first-order necessary condition.
By Assumption~\ref{RPG:as10} and Lemma~\ref{RPG:le3}, the sequence $\{x_k\}$ stays in the compact set $\Omega_{x_0}$, which implies the existence of an accumulation point.

In order to prove that any accumulation point is a stationary point, we will resort to ~\cite[Theorem 2.2(c)]{HHY2018} which states that if $\{z_i\} \subset \mathcal{M}$, $\xi_i \in \partial F(z_i)$, $z_i \rightarrow z_*$, \whfirrev{}{$F(z_i)\rightarrow F(z_*)$}, and $\xi_i \rightarrow \xi_*$ as $i \rightarrow \infty$, then $\xi_* \in \partial F(z_*)$.

By Lemma~\ref{RPG:le3}, we have that $F(x_0) - F(\whfirrev{}{\tilde{x}}) \geq \beta \sum_{i = 0}^\infty \|\eta_{x_k}^*\|_{x_k}^2$, where $\whfirrev{}{\tilde{x}}$ denotes a global minimizer of $F$. Therefore, 
\begin{equation} \label{RPG:e82}
\lim_{k \rightarrow \infty} \|\eta_{x_k}^*\|_{x_k} = 0.
\end{equation}
Let $\{x_{k_j}\}$ be a subsequence satisfying 
\begin{equation} \label{RPG:e91}
x_{k_j} \rightarrow x_*
\end{equation}
as $j \rightarrow \infty$.
\whfirrev{}{
Choose $\delta>0$ sufficiently small such that $\mathbb{B}(x_*, \delta )$ is a totally retractive set, or equivalently, there exists a positive constant $\varrho$ such that, for all $y \in \mathbb{B}(x_*, \delta )$, $\mathbb{B}(x_*, \delta ) \subset R_y(\mathcal{B}(0_y, \varrho) )$ and $R_y$ is a diffeomorphism in $\mathcal{B}(0_y, \varrho)$. By~\eqref{RPG:e91}, there exists $J_1 > 0$ such that $x_{k_j} \in \mathbb{B}(x_*, \delta )$ for all $j > J_1$. By~\eqref{RPG:e82}, there exists $J_2 > 0$ such that $\|\eta_{x_{k_j}}^*\|_{x_{k_j}} < \varrho$ for all $j > J_2$.

%There exists $K > 0$ such that for all $k > K$, it holds that $\|\eta_{x_k}^*\|_{x_k} < \varrho$.
}
\whfirrev{}{Since $R_{x_{k_j}}$ is smooth, % in $\mathbb{B}(x_{k_j}, i_R(\Omega_{x_0}))$, 
the limit \eqref{RPG:e82} yields \whsecrev{}{
$\lim_{j \rightarrow \infty} R_{x_{k_j}} (\eta_{x_{k_j}}^*) = \lim_{j \rightarrow \infty} R_{x_{k_j}} (0_{x_{k_j}}) = \lim_{j \rightarrow \infty} x_{k_j}$. Further using~\eqref{RPG:e91} and noting $x_{k_j + 1} = R_{x_{k_j}} (\eta_{x_{k_j}}^*)$, we have $\lim_{j\rightarrow \infty} x_{k_j + 1} = x_*$.
}}

%$R_{x_{k_j}} ( \eta_{x_{k_j}}^* ) \rightarrow R_{x_{k_j}} ( 0_{x_{k_j}} )$, which implies $x_{k_j + 1} \rightarrow x_{k_j}$. Therefore, it holds that $x_{k_j + 1} \rightarrow x_*$.}

%By Lemma~\ref{RPG:le16} and~\eqref{RPG:e82}, it holds that $\dist(x_{k_j}, x_{k_j + 1}) \leq \kappa \|\eta_{x_{k_j}}^*\|_{x_{k_j}} \rightarrow 0$. Therefore, we have $x_{k_j + 1} \rightarrow x_*$. \kwcomm{not the right one to use} \whcomm{[Should be right now.]}{}
By the definition of $\eta_{x_k}^*$ in~\eqref{RPG:subproblem2}, there exists $\zeta_{x_{k+1}} \in \partial g(x_{k+1})$ such that
\begin{equation} \label{RPG:e27}
	\grad f(x_k) + \tilde{L} \eta_{x_k}^* + \mathcal{T}_{R_{\eta_{x_k}^*}}^\sharp \zeta_{x_{k+1}} = 0.
\end{equation}
\whfirrev{}{
Since for any $j > \max(J_1, J_2)$, $R_{x_{k_j}}$ is a diffeomprohism in $\mathbb{B}(x_{k_j}, \varrho)$ and $\eta_{x_{k_j}}^* \in \mathbb{B}(x_{k_j}, \varrho )$, the vector transport by differentiated retraction $\mathcal{T}_{R_{ \eta_{x_{k_j}} }}$ is invertible. It follows that $\mathcal{T}_{R_{ \eta_{x_{k_j}} }}^{\sharp}$ is invertible.}
Therefore, we have
\begin{equation} \label{RPG:e26}
	\grad f(R_{x_{k_j}}(\eta_{x_{k_j}}^*)) - \mathcal{T}_{R_{\eta_{x_{k_j}}^*}}^{- \sharp} (\grad f(x_{k_j}) + \tilde{L}  \eta_{x_{k_j}}^*) = \grad f(x_{k_j + 1}) + \zeta_{x_{k_j+1}} \in \partial F(x_{k_j + 1}).
\end{equation}
Combining~\eqref{RPG:e82} with Lemma~\ref{RPG:le10} yields
\begin{align}
&\|\grad f(R_{x_{k_j}}(\eta_{x_{k_j}}^*)) - \mathcal{T}_{R_{\eta_{x_{k_j}}^*}}^{- \sharp} (\grad f(x_{k_j}) + \tilde{L} \eta_{x_{k_j}}^*)\|_{R_{x_{k_j}}(\eta_{x_{k_j}}^*)} \nonumber \\
\leq& \|\grad f(R_{x_{k_j}}(\eta_{x_{k_j}}^*)) - \mathcal{T}_{R_{\eta_{x_{k_j}}^*}}^{- \sharp} \grad f(x_{k_j})\|_{R_{x_{k_j}}(\eta_{x_{k_j}}^*)} + \|\mathcal{T}_{R_{\eta_{x_{k_j}}^*}}^{- \sharp} \tilde{L} \eta_{x_{k_j}}^*\|_{R_{x_{k_j}}(\eta_{x_{k_j}}^*)} \rightarrow 0, \hbox{ as $j \rightarrow \infty$}. \nonumber %\\
%\leq& \whfirrev{}{b_0 \|\eta_{x_{k_j}}^*\|_{x_{k_j}}}  \rightarrow 0, \hbox{ as $j \rightarrow \infty$}, \label{RPG:e70}
\end{align}
%where $b_0$ is a positive constant.
%The application of Lemmas~\ref{RPG:le3} and~\ref{RPG:le10} yields  that \kwcomm{not very straightforward}
%$$\grad f(R_{x_k}(\eta_{x_k}^*)) - \mathcal{T}_{R_{\eta_{x_k}^*}}^{- \sharp} (\grad f(x_k) + \tilde{L} \eta_{x_k}^*) \rightarrow 0.$$
\kwnote{}{Moreover, since $F$ is continuous, $F(x_{k_j+1})\rightarrow F(x_*)$.} It follows from~\cite[Theorem 2.2(c)]{HHY2018} that $0 \in \partial F(x_*)$, so $x_*$ is a stationary point.

\whfirrev{}{
Lastly, we show that Algorithm~\ref{RPG:a1} returns $x_k$ satisfying $\|\eta_{x_k}^*\|_{x_k} \leq \epsilon$ in at most $(F(x_0) - F(x_*)) / (\beta \epsilon^2)$ iterations. If it was not true, then it would hold that $\|\eta_{x_k}^*\|_{x_k} > \epsilon$ for all $k = 0, 1, \ldots, K-1$, where $K$ is the the smallest integer larger than or equal to $(F(x_0) - F(x_*)) / (\beta \epsilon^2)$. It follows that $F(x_0) - F(x_*) \geq F(x_0) - F(x_K) > K \beta \epsilon^2 \geq (F(x_0) - F(x_*)) / (\beta \epsilon^2) * (\beta \epsilon^2) = (F(x_0) - F(x_*))$, which is a contradiction.
}

%Suppose Algorithm~\ref{RPG:a1} does not find $x_k$ satisfying $\|\eta_{x_k}^*\|_{x_k} \leq \epsilon$ in $K$ iterations, i.e., $\|\eta_{x_k}^*\|_{x_k} > \epsilon$ for all $k = 0, 1, \ldots, K-1$. It follows that $F(x_0) - F(x_*) \geq F(x_0) - F(x_k) \geq \beta K \epsilon^2$. Therefore, after at most $K > (F(x_0) - F(x_*)) / (\beta \epsilon^2)$ iterations, Algorithm~\ref{RPG:a1} finds $x_k$ satisfying $\|\eta_{x_k}^*\|_{x_k} \leq \epsilon$.

\end{proof}

\subsection{Convergence Rate Analysis Using Retraction Convexity} \label{RPG:sect:ConRatConv} %$O(1/k)$ 

%The function $g$ is assumed to be convex in order to have $O(1/k)$ convergence rate in Euclidean space.
It is well-known that in the Euclidean setting the proximal gradient method \eqref{RPG:EPG} has $O(1/k)$ convergence rate for convex problems~\cite{Beck2009}.
 %In the Riemannian setting, we make a similar assumption and a concept of convexity on a manifold is given in Definition~\ref{RPG:def:Rconvexity}.
 In order to establish the convergence rate of Algorithm~\ref{RPG:a1} in the Riemannian setting, we use the following concept of convexity on a manifold.

%In this section, we show that under certain assumptions, the function values $f(x_k)$ converges to $f(x_*)$ on the order of $O(1/k)$. %asymptotically. Therefore, without loss of generality, we make a blanket assumption that $\Omega_{x_0}$ is small as needed.

%\begin{assumption} \label{RPG:as7}
%There exists a compact set $\Omega \subseteq \mathcal{M}$ such that the sequence $\{x_k\}$ generated by Algorithm~\ref{RPG:a1} stays in $\Omega$.
%\end{assumption}

\kwcomm{history about retraction-convex?} \whcomm{[It has been discussed after Lemma~\ref{RPG:le13}.]}{}
\begin{definition} \label{RPG:def:Rconvexity}
A function $h:\mathcal{M} \rightarrow \mathbb{R}$ is called retraction-convex with respect to a retraction $R$ in  $\mathcal{N} \subseteq \mathcal{M}$ if for any $x \in \mathcal{N}$ and any $\mathcal{S}_x \subseteq \T_x \mathcal{M}$ such that $R_x(\mathcal{S}_x)\subseteq \mathcal{N}$, there exists a tangent vector $\zeta \in \T_x \mathcal{M}$ such that $q_x = h \circ R_x$ satisfies
\begin{equation} \label{RPG:def:conv}
q_x(\eta) \geq q_x(\xi) + \inner[x]{\zeta}{\eta - \xi}\;\; \forall \eta, \xi \in \mathcal{S}_x.
\end{equation}
Note that $\zeta = \grad q_x(\xi)$ if $h$ is differentiable; otherwise, $\zeta$ is any \whfirrev{}{Riemannian} subgradient of $q_x$ at $\xi$.
%Furthermore, the function $h$ is $\mu$-strongly-retraction-convex if
%$$
%q_x(\eta) \geq q_x(\xi) + \inner[]{\grad q_x(\xi)}{\eta - \xi} + \frac{\mu}{2} \|\eta - \xi\|^2\;\; \forall \eta, \xi \in \mathcal{S}_x.
%$$
\end{definition}
\whcomm{}{
In a Euclidean space,  any local minimizer of a convex function over a convex set is a global minimizer. In the Riemannian setting, a notion of retraction-convex set is not well-defined in general. %\footnote{If the retraction is the exponential mapping, geodesically convex set has been defined in ~\cite{ZS2016}.} 
To avoid such technical difficulties, we can assume that minimizers only appear in the interior of the constrained set. Then it is not difficult to show that any local minimizer is a global minimizer. \whfirrev{}{The details are omitted due to the similarity with the Euclidean case.}
}

Convexity of functions on  Riemannian manifolds has already been investigated in the literature  based on geodesic, see for example \cite{FO2002,ZS2016}.  A function $h:\mathcal{M} \rightarrow \mathbb{R}$ is called geodesic convex, if for any $x, y \in \mathcal{M}$, there exists a tangent vector $\eta_x \in \T_x \mathcal{M}$ such that $f(y) \geq f(x) + \inner[x]{\eta_x}{\mathrm{Exp}_x^{-1}(y)}$. It can be verified that if a function is retraction-convex with respect to the exponential mapping, then it is indeed geodesic-convex. \whcomm{}{This can be easily seen by setting $\xi = 0$ and choosing the retraction to be the exponential mapping in~\eqref{RPG:def:conv}.} \kwcomm{cite}\whcomm{[Added a sentence before.]}{}  \whcomm{}{In~\cite{HGA2014}, a retraction-convexity is defined  for $C^2$ functions on manifolds which can be viewed as  a Riemannian generalization of the geodesic convexity for $C^2$ functions.}
\kwcomm{not sure whether this comment added here is appropriate or not} \whcomm{[I am fine with this comment.]}{}
%, but not vice versa.
% I think the following lemma is true. But I do not know how to prove it yet.
%\begin{lemma} \label{RPG:le11}
%	Let the manifold $\mathcal{M}$ be an embedded submanifold of $\mathbb{R}^n$ endowed with the Euclidean metric. Let $\mathcal{S}$ be an open subset of $\mathbb{R}^n$ and $\tilde{h}: \mathcal{S} \subset \mathbb{R}^n \rightarrow \mathbb{R}$ be a strongly convex function in the Euclidean setting, i.e., there exists a vector $u \in \mathbb{R}^n$ and a positive constant $\mu$ such that $\tilde{h}(y) \geq \tilde{h}(x) + \inner[2]{u}{y - x} + \frac{\mu}{2} \|y - x\|_2^2$ for any $x, y \in \mathcal{S}$.
%	Then for any $x \in \mathcal{S}$, there exists a neighborhood of $x$, denoted by $\mathcal{N}_x$, such that the function $h: \mathcal{M} \rightarrow \mathbb{R}: x \mapsto \tilde{h}(x)$, obtained by restricting $\tilde{h}$ to $\mathcal{M}$, is retraction-convex with respect to any second order retraction in $\mathcal{N}_x$.
%\end{lemma}
The following lemma presents two sufficient conditions for a function to be  locally retraction-convex.
\begin{lemma} \label{RPG:le12}
\whfirrev{}{Given $x \in \mathcal{M}$ and a twice continuously differentiable function $h:\mathcal{M} \rightarrow \mathbb{R}$,  if one of the following conditions holds:
\begin{itemize}
% The function $h:\mathcal{M} \rightarrow \mathbb{R}$ is ;
\item $\Hess h$ is positive definite at $x$, and the retraction is second order;
\item The manifold $\mathcal{M}$ is an embedded submanifold of $\mathbb{R}^n$ endowed with the Euclidean metric; $\mathcal{W}$ is an open subset of $\mathbb{R}^n$; $x \in \mathcal{W}$; $h: \mathcal{W} \subset \mathbb{R}^n \rightarrow \mathbb{R}$ is a $\mu$-strongly convex function in the Euclidean setting for a sufficient large $\mu$; the retraction is second order; %, i.e., there exists a vector $u \in \mathbb{R}^n$ and a positive constant $\mu$ such that $\tilde{h}(y) \geq \tilde{h}(x) + \inner[2]{u}{y - x} + \frac{\mu}{2} \|y - x\|_2^2$ for any $x, y \in \mathcal{W}$.
\end{itemize}
then there exists a neighborhood of $x$, denoted by $\mathcal{N}_x$, such that the function $h:\mathcal{M} \rightarrow \mathbb{R}$ is retraction-convex in $\mathcal{N}_x$.
}

%Let the manifold $\mathcal{M}$ be an embedded submanifold of $\mathbb{R}^n$ endowed with the Euclidean metric. Let $\mathcal{W}$ be an open subset of $\mathbb{R}^n$ and $\tilde{h}: \mathcal{W} \subset \mathbb{R}^n \rightarrow \mathbb{R}$ be a $\mu$-strongly convex function in the Euclidean setting, i.e., there exists a vector $u \in \mathbb{R}^n$ and a positive constant $\mu$ such that $\tilde{h}(y) \geq \tilde{h}(x) + \inner[2]{u}{y - x} + \frac{\mu}{2} \|y - x\|_2^2$ for any $x, y \in \mathcal{W}$. Suppose $\tilde{h}$ is twice continuously differentiable. Then if $\mu$ is sufficient large, then for any $x \in \mathcal{W} \cap \mathcal{M}$, there exists a neighborhood of $x$, denoted by $\mathcal{N}_x$ such that the function $h: \mathcal{M} \rightarrow \mathbb{R}: x \mapsto \tilde{h}(x)$, obtained by restricting $\tilde{h}$ to $\mathcal{M}$, is retraction-convex with respect to any second order retraction in $\mathcal{N}_x$.
\end{lemma}
{\em Proof. }
First note that an equivalent condition of $\mu$-strongly convexity for a twice continuously differentiable function is that the smallest eigenvalue of its Hessian is greater than $\mu$.
\begin{itemize}
\item Since the retraction is second order, it follows from~\cite[Proposition~5.5.6]{AMS2008} that $\Hess h(x) = \Hess (h \circ R_x) (0_x)$. Therefore, $\Hess (h\circ R)$ is positive definite at $0_x$. Since $h$ and $R$ are twice continuously differentiable, \whcomm{[fixed the notation]}{$\Hess (h\circ R)$} is continuous in $\T \mathcal{M}$.\kwcomm{function of what? or should it be $\Hess (h\circ R_x) (\eta)$} \whcomm{[it is a function of a tangent vector in the tangent bundle. In other words, $\Hess (h \circ R)$ is a function from the tangent bundle to a linear operator on a tangent space.]}{} Therefore, there exists a neighborhood of $0_x$, denoted by $\mathcal{S}_x \subset \T \mathcal{M}$, such that $\Hess (h \circ R)(\eta)$ is positive definite for any $\eta \in \mathcal{S}_x$ \kwcomm{should it be $\Hess (h \circ R_x)(\eta)$}. This implies that $h$ is retraction-convex in a sufficient small neighborhood of $x$.
\item The Riemannian Hessian of $h$ at any point $x \in \mathcal{M}$ is (see~\cite{AMT2013})
	\begin{equation*}
		\Hess h(x) [\eta_x] = P_{\T_x \mathcal{M}} \nabla^2 h(x) \eta_x +  P_{\T_x \mathcal{M}} (\D_{\eta_x}  P) \nabla h(x),
	\end{equation*}
	where $\D_{\eta_x} P = \lim_{t \rightarrow 0} \frac{P_{\T_{\gamma(t)} \mathcal{M}} - P_{\T_{\gamma(0)} \mathcal{M}}}{t}$, $\gamma$ is a curve on $\mathcal{M}$ such that $\gamma(0) = x$ and $\gamma'(0) = \eta_x$. Let $\vartheta = \sup_{\eta_x \in \T_x \mathcal{M}} \frac{\|P_{\T_x \mathcal{M}} (\D_{\eta_x}  P) \nabla h(x)\|_{\E}}{\|\eta_{x}\|_{\E}}$. It holds that \whfirrev{}{for all $\eta_x \in \T_x \mathcal{M}$},
	\begin{equation} \label{RPG:e29}
	\inner[\E]{\eta_x}{\Hess h(x) [\eta_x]} = \inner[\E]{\eta_x}{\nabla^2 h(x) \eta_x} + \inner[\E]{\eta_x}{P_{\T_x \mathcal{M}} (\D_{\eta_x}  P) \nabla h(x)} \geq (\mu - \vartheta) \|\eta_x\|_{\E}^2,
	\end{equation}
	where the inequality is from the $\mu$-strongly convexity of $h$. As a result, $\lambda_\mathrm{min}(\Hess h(x)) \geq \mu - \vartheta$, where $\lambda_{\mathrm{min}}(M)$ denotes the smallest eigenvalue of the linear operator $M$. It follows from~\cite[Proposition~5.5.6]{AMS2008} that $\inner[\E]{\eta_x}{\Hess h\circ R_x(0_x) [\eta_x]} = \inner[\E]{\eta_x}{\Hess h(x) [\eta_x]}$ for any second order retraction. Therefore, we have $\lambda_\mathrm{min}(\Hess h\circ R_x(0_x)) \geq \mu - \vartheta$. If \whfirrev{}{$\mu > \vartheta$},
%	\begin{equation} \label{RPG:e35}
%		\mu > \vartheta = \sup_{x \in \mathcal{W} \cap \mathcal{M}} \frac{\|P_{\T_x \mathcal{M}} (\D_{\eta_x}  P) \nabla h(x)\|_{\E}}{\|\eta_x\|_{\E}},
%	\end{equation}
    then  $\Hess h \circ R$ is positive definite at $0_x$. It follows that $h$ is retraction-convex in a sufficiently small neighborhood of $x$.\hfil\qed
	%Suppose $\mu \geq 2 \vartheta$. Define a function $q:\T \mathcal{M} \rightarrow \mathbb{R}: x \rightarrow \lambda_\mathrm{min}(\Hess h\circ R_x(0_x))$. It follows from twice differentiable continuity of $R$ and $h$ that $q$ is continuous. Therefore, there exists a neighborhood of $x$, denoted by $\mathcal{N}_x$, such that $\Hess f \circ R_y(\xi_y)$ is positive definite for any $y \in \mathcal{N}_x$ and $\xi_y$ satisfying $R_y(\xi_y) \in \mathcal{N}_x$. This implies $h \circ R_y(\xi_y)$ is convex for any $y \in \mathcal{N}_x$ and $\xi_y$ satisfying $R_y(\xi_y) \in \mathcal{N}_x$.
\end{itemize}

The convergence rate analysis of the Riemannian proximal gradient methods relies on the following two assumptions.
\begin{assumption} \label{RPG:as1}
There exists an open set $\Omega \supseteq \Omega_{x_0}$ such that the function $f$ is $L$-retraction-smooth and retraction-convex with respect to the retraction $R$ in $\Omega$. The function $g$ is retraction-convex with respect to the retraction $R$ in $\Omega$. %The radius $\rho$ satisfies $\rho \geq \sup_{x, y \in \Omega_{x_0}} \|R_x^{-1}(y)\|$. %, where $x_*$ is a minimizer of $F$.
\end{assumption}

%In Assumption~\ref{RPG:as1}, the radius $\rho$ is assumed to be sufficient large in order to let the subproblem~\eqref{RPG:subproblem1} to find a minimizer in the interior of the constrained set. \kwcomm{somehow conflict}
%\whcomm{[If $\tilde{f}: \mathbb{R}^n \rightarrow \mathbb{R}$ is $C^2$ and $\tilde\mu$-strongly convex in a neighborhood of a point $\bar{x} \in \mathcal{M}$ and the retraction $R$ is a second order retraction, then the function by restricting $\tilde{f}$ on $\mathcal{M}$, $f: \mathcal{M} \rightarrow \mathbb{R}: x \mapsto \tilde{f}(x)$, is $\mu$-strongly-retraction-convex in a neighborhood of $\bar{x}$.
%
%But we do not have such nice result for non-strongly convex function. Imposing conditions on the retraction $R$ should solve this problem.]}{}
%
%\whcomm{[In the Euclidean setting, it does not need to assume strongly convexity. Is it possible to further release this condition for Riemannian problems?]}{}
%

 \kwcomm{why righthand is $\eta_x$ not $\xi_x$ which seems unsymmetric} \whcomm{[The right hand side can be $\kappa \min(\|\eta_x\|^2, \|\xi_x\|^2) \|\zeta\|_y^2$. A footnote is added.]}{}
\begin{assumption} \label{RPG:as6}
For any $x, y, z \in \Omega$, there exists a constant \whfirrev{}{$\kappa_{\Omega}$} such that \footnote{The right hand side of~\eqref{RPG:e52} can be $\whfirrev{}{\kappa_{\Omega}} \min(\|\eta_x\|_x^2, \|\xi_x\|_x^2) \|\zeta_y\|_y^2$. We use the the form in~\eqref{RPG:e52} for simplicity.}
\begin{align} \label{RPG:e52}
\left| \| \xi_x - \eta_x \|_x^2 - \| \zeta_y \|_y^2 \right| \leq& \whfirrev{}{\kappa_{\Omega}} \|\eta_x\|_x^2,
%\left| \| \xi_x - \eta_x \|^2 - \| \zeta_y \|^2 \right| \leq& \kappa \min( \|\eta_x\|^2, \|\xi_x\|^2) \|\zeta_y\|^2
\end{align}
where $\eta_x = R_x^{-1}(y)$, $\xi_x = R_x^{-1}(z)$, $\zeta_y = R_y^{-1}(z)$, \whfirrev{}{$\kappa_{\Omega}$} is a constant, and $\Omega$ is defined in Assumption~\ref{RPG:as1}.
\end{assumption}

 Assumption~\ref{RPG:as6}  imposes an additional restriction on the retraction $R$. In the Euclidean setting, this assumption naturally holds since $\xi_x - \eta_x = (z - x) - (y - x) = (z - y) = \zeta_y$. In the Riemannian setting, we find this assumption reasonable in the sense that \whsecrev{}{
 it has been proven in~\cite[Lemma~2.4]{DDDR2018} that~\eqref{RPG:e52} holds for any compact set $\overline{\Omega}$ when the retraction is the exponential mapping, where $\overline{\Omega}$ denotes the closure of $\Omega$.
 }

The following lemma is central to the later convergence rate analysis and it is a Riemannian version of~\cite[Lemma~2.3]{Beck2009}.

\kwcomm{something about the connection between stationary point in the algorithm and local min in the lemma should be discussed under the rectraction-convex property} \whcomm{[Done]}{} \whcomm{[modify the definition of $\eta_x^*$.]}{}

\begin{lemma} \label{RPG:le1}
\whcomm{}{Let $\eta_x^*$ be a \whfirrev{}{stationary point} of $\ell_x(\eta) = \inner[x]{\grad f(x)}{\eta} + \frac{\tilde{L}}{2} \|\eta\|_x^2 + g\left(R_{x}(\eta)\right)$ such that $\ell_x(0) \geq \ell_x(\eta_x^*)$.} Suppose Assumption~\ref{RPG:as1} holds, and $x$ and $z = R_x(\eta_x^*)$ are in $\Omega$. Then \whfirrev{}{for any $\xi_x \in \T_x \mathcal{M}$ such that $y:= R_x(\xi_x) \in \Omega$}, we have
\begin{equation*}
F(z) \leq F(y) + \frac{\tilde{L}}{2} \left( \|\xi_x\|_x^2 - \|\xi_x - \eta_x^*\|_x^2 \right).
\end{equation*}
%where $\xi_x = R_x^{-1}(y)$.
\end{lemma}
\begin{proof}
By definition of $\eta_x^*$, we have
\begin{equation} \label{RPG:e1}
0 = \grad f(x) + \tilde{L} \eta_x^* + \mathcal{T}_{R_{\eta_x^*}}^\sharp \zeta_z,
\end{equation}
where $\zeta_z \in \partial g(z)\subset\T_z\mathcal{M}$. Since $g$ is retraction-convex and $y, z$ are in $\Omega$, we have
\begin{equation} \label{RPG:e2}
g(y) - g(z) = g(R_x(\xi_x)) - g(R_x(\eta_x^*)) \geq \inner[x]{\mathcal{T}_{R_{\eta_x^*}}^\sharp \zeta_z}{ \left( \xi_x - \eta_x^* \right) }.
\end{equation}
Combining~\eqref{RPG:e2} with~\eqref{RPG:e1} yields
\begin{equation} \label{RPG:e3}
g(y) - g(z) \geq  \inner[x]{\grad f(x) + \tilde{L} \eta_x^*}{ \left( \eta_x^* - \xi_x \right) }.
\end{equation}
It follows that
\begin{align*}
&F(z) = F(R_x(\eta_x^*)) = f(R_x(\eta_x^*)) + g(R_x(\eta_x^*)) = f(z) + g(z) \\
\leq& g(y) + \inner[x]{ \grad f(x) + \tilde{L} \eta_x^* }{ \left( \xi_x - \eta_x^* \right) } + f(z)\;\; \hbox{ (using~\eqref{RPG:e3})} \\
\leq& g(y) + \inner[x]{ \grad f(x) + \tilde{L} \eta_x^* }{ \left( \xi_x - \eta_x^* \right) } + f(x) + \inner[x]{\grad f(x)}{\eta_x^*} + \frac{\tilde{L}}{2} \|\eta_x^*\|_x^2\;\; \hbox{ ($f$ is $L$-smooth)} \\
=& g(y) + f(x) + \inner[x]{\grad f(x)}{\xi_x} + \inner[x]{\tilde{L} \eta_x^*}{\xi_x - \eta_x^*} + \frac{\tilde{L}}{2} \|\eta_x^*\|_x^2 {\kwcomm{should~ it ~be~ equality?} \whcomm{[yes]}{}}\\
\leq& g(y) + f(y) + \inner[x]{\tilde{L} \eta_x^*}{\xi_x - \eta_x^*} + \frac{\tilde{L}}{2} \|\eta_x^*\|_x^2 \;\;\;\; \hbox{ ($f$ is retraction-convex)} \\
=& F(y) + \frac{\tilde{L}}{2}\left( \inner[x]{\eta_x^*}{2 \xi_x - \eta_x^*} \right)\\
 =& F(y) + \frac{\tilde{L}}{2} \left( \|\xi_x\|_x^2 - \|\xi_x - \eta_x^*\|_x^2 \right),
\end{align*}
which concludes the proof.
\end{proof}

Theorem~\ref{RPG:th1} shows that Algorithm~\ref{RPG:a1} converges on the order of $O(1/k)$. Note that in the Euclidean setting, the second term on the right side of~\eqref{RPG:e30} vanishes since $\whfirrev{}{\kappa_{\Omega}}= 0$. \whsecrev{}{The convergence rate given in Theorem~\ref{RPG:th1} is different from the iteration complexity stated in Theorem~\ref{RPG:globaltheo} in the sense that the iteration complexity in Theorem~\ref{RPG:globaltheo}   measures the rate of $\|\eta_{x_k}^*\|_{x_k}$ going to zero, not that of  $F(x_k)$ going to $F(x_*)$. }
\begin{theorem} \label{RPG:th1}
Suppose Assumptions~\ref{RPG:as10}, ~\ref{RPG:as1} and~\ref{RPG:as6} hold. \whfirrev{}{Let $x_*$ be any accumulation point of the sequence $\{x_k\}$.} Then the sequence $\{x_k\}$ generated by Algorithm~\ref{RPG:a1} satisfies
\begin{equation} \label{RPG:e30}
F(x_k) - F(x_*) \leq \frac{1}{k} \left( \frac{\tilde{L}}{2} \|R_{x_0}^{-1} (x_*)\|_{x_0}^2 + \frac{\tilde{L} \whfirrev{}{\kappa_{\Omega}}}{2 \beta} (F(x_0) - F(x_*))\right),
\end{equation}
where \whfirrev{}{$\kappa_{\Omega}$} is defined in Assumption~\ref{RPG:as6} and $\beta$ is defined in~\eqref{RPG:e9}.
\end{theorem}

\begin{proof}
Lemma~\ref{RPG:le1} with $x = x_k$ and $y = x_*$ gives
\begin{equation*}
F(x_{k+1}) - F(x_*) \leq \frac{\tilde{L}}{2} \left( \|R_{x_k}^{-1} (x_*)\|_{x_k}^2 - \|R_{x_k}^{-1} (x_*) - \eta_{x_k}^*\|_{x_k}^2 \right).
\end{equation*}
Furthermore, Assumption~\ref{RPG:as6} with $x = x_k$, $y = x_{k+1}$, $z = x_*$ gives
\begin{equation*}
\left|\|R_{x_k}^{-1} (x_*) - \eta_{x_k}^*\|_{x_k}^2 - \|R_{x_{k+1}}^{-1} (x_*)\|_{x_{k + 1}}^2\right| \leq \whfirrev{}{\kappa_{\Omega}} \|\eta_{x_k}^*\|_{x_k}^2.
\end{equation*}
Consequently,
\begin{align}
F(x_{k+1}) - F(x_*) &\leq \frac{\tilde{L}}{2} \left( \|R_{x_k}^{-1} (x_*)\|_{x_k}^2 - \|R_{x_k}^{-1} (x_*) - \eta_{x_k}^*\|_{x_k}^2 \right) \nonumber \\
&\leq \frac{\tilde{L}}{2} \left( \|R_{x_k}^{-1} (x_*)\|_{x_k}^2 - \|R_{x_{k+1}}^{-1} (x_*)\|_{x_{k+1}}^2 \right) + \frac{\tilde{L}}{2} \whfirrev{}{\kappa_{\Omega}} \|\eta_{x_k}^*\|_{x_k}^2. \label{RPG:e8}
\end{align}
Combining~\eqref{RPG:e9} and~\eqref{RPG:e8} yields
\begin{equation} \label{RPG:e11}
F(x_{k+1}) - F(x_*) \leq \frac{\tilde{L}}{2} \left( \|R_{x_k}^{-1} (x_*)\|_{x_k}^2 - \|R_{x_{k+1}}^{-1} (x_*)\|_{x_{k+1}}^2 \right) + \frac{\tilde{L} \whfirrev{}{\kappa_{\Omega}}}{2 \beta} (F(x_{k}) - F(x_{k+1})).
\end{equation}
Thus, after summing~\eqref{RPG:e11} over $k$ from 0 to $s-1$ and dividing the result by $s$, we obtain
%\begin{equation} \label{RPG:e12}
%\left(\frac{1}{s} \sum_{k = 0}^{s-1} F(x_{k+1}) - F(x_*)\right) \leq \frac{\tilde{L}}{2s} \left( \|R_{x_0}^{-1} (x_*)\|_{x_0}^2 - \|R_{x_{s}}^{-1} (x_*)\|_{x_s}^2 \right) + \frac{\tilde{L} \whfirrev{}{\kappa_{\Omega}}}{2 \beta s} (F(x_{0}) - F(x_{*})).
%\end{equation}
\begin{equation} \label{RPG:e12}
\left(\frac{1}{k} \sum_{s = 0}^{k-1} F(x_{s+1}) - F(x_*)\right) \leq \frac{\tilde{L}}{2k} \left( \|R_{x_0}^{-1} (x_*)\|_{x_0}^2 - \|R_{x_{k}}^{-1} (x_*)\|_{x_k}^2 \right) + \frac{\tilde{L} \whfirrev{}{\kappa_{\Omega}}}{2 \beta k} (F(x_{0}) - F(x_{*})).
\end{equation}
%Finally, the result~\eqref{RPG:e30} follows from~\eqref{RPG:e12} and~\eqref{RPG:e9}.
Since~\eqref{RPG:e9} implies $F(x_k)-F(x_*)$ is decreasing, \eqref{RPG:e30} follows immediately from~\eqref{RPG:e12}.
%inequalities~\eqref{RPG:e12} and~\eqref{RPG:e9} yields
%\begin{equation*}
%F(x_s) - F(x_*) \leq \frac{1}{s} \left( \frac{\tilde{L}}{2} \|R_{x_0}^{-1} (x_*)\|^2 + \frac{\tilde{L} \kappa}{2 \beta} (F(x_0) - F(x_*))\right).
%\end{equation*}
\end{proof}

%\whfirrev{[Section~\ref{sec:CRA} is new.]}{}

\subsection{\whfirrev{}{Local Convergence Rate Analysis Using Riemannian Kurdyka-\L ojasiewicz Property}} \label{sec:CRA}

%In the Eucliean setting, the Kurdyka-Lojasiewicz (KL) property has been used to analyze the local convergence rate of proximal gradient methods, see e.g.,~\cite{LL2015}. %It has been shown that the proximal gradient method archives linear convergence rate under certain conditions. 
%In this section, the local convergence rate analysis of the Riemannian proximal gradient method is established based on a Riemannian Kurdyka-Lojasiewicz (KL) property. %In addition, two sufficient conditions are given such that one can verify whether a function has the Riemannian KL property.
The KL property has been widely used for the convergence analysis of various convex and nonconvex algorithms in the Euclidean case, see e.g., \cite{ABRS2010,ABS2013,BST2014,LL2015}. In this section we will study the convergence of RPG base on the Riemannian Kurdyka-\L ojasiewicz (KL) property, introduced in~{\cite{KMP2000}} for the analytic setting and in~\cite{BCO2011} for the nonsmooth setting .
%The local convergence rate analysis in this section is established on the Riemannian Kurdyka-\L ojasiewicz (KL) property  
%The Kurdyka-Lojasiewicz property was defined for a function on the Euclidean space, see e.g.,~\cite{ABRS2010}, and has been used to analyze the local convergence rate of Euclidean proximal gradient methods, see e.g.,~\cite{LL2015}. To emphasize its difference to the Riemannian KL property, we refer it to the Euclidean KL property. For completeness, the definitions of Euclidean KL property and Riemannian KL property are given in Definitions~\ref{RPG:def:EKL} and~\ref{RPG:def:RKL} respectively.

%\begin{definition}  \label{RPG:def:EKL}
%A local Lipschitz function $f: \mathbb{R}^n \rightarrow \mathbb{R}$ is said to have the Euclidean Kurdyka-Lojasiewicz property at $x \in \mathbb{R}^n$ if and only if there exists $\varepsilon \in (0, \infty]$, a neighborhood $U \subset \mathbb{R}^n$ of $x$, and a continuous concave function $\varsigma: [0, \varepsilon] \rightarrow [0, \infty)$ such that
%\begin{itemize}
%\item $\varsigma(0) = 0$,
%\item $\varsigma$ is $C^1$ on $(0, \varepsilon)$,
%\item $\varsigma' > 0$ on $(0, \eta)$,
%\item For every $y \in U$ with $f(x) < f(y) < f(x) + \varepsilon$, the Kurdyka-Lojasiewicz inequality holds
%\[
%\varsigma'(f(y) - f(x)) \dist(0, \partial f(y)) \geq 1,
%\]
%where $\dist(0, \partial f(y)) = \inf\{ \|v\|_F : v \in \partial f(y) \}$ and $\partial$ denotes the Euclidean generalized subdifferential. The function $\varsigma$ is called the desingularising function.
%\end{itemize}
%\end{definition}

\begin{definition} \label{RPG:def:RKL}
A continuous function $f: \mathcal{M} \rightarrow \mathbb{R}$ is said to have the Riemannian KL property at $x \in \mathcal{M}$ if and only if there exists $\varepsilon \in (0, \infty]$, a neighborhood $U \subset \mathcal{M}$ of $x$, and a continuous concave function $\varsigma: [0, \varepsilon] \rightarrow [0, \infty)$ such that
\begin{itemize}
\item $\varsigma(0) = 0$,
\item $\varsigma$ is $C^1$ on $(0, \varepsilon)$,
\item $\varsigma' > 0$ on $(0, \whsecrev{}{\varepsilon})$,
\item For every $y \in U$ with $f(x) < f(y) < f(x) + \varepsilon$, we have 
\[
\varsigma'(f(y) - f(x)) \dist(0, \partial f(y)) \geq 1,
\]
where $\dist(0, \partial f(y)) = \inf\{ \|v\|_y : v \in \partial f(y) \}$ and $\partial$ denotes the Riemannian generalized subdifferential. The function $\varsigma$ is called the desingularising function.
\end{itemize}
\end{definition}

 Note that the definition of the Riemannian KL property is overall similar to the KL property in the Euclidean setting, except that related notions including $U$, $\partial f(y)$ and $\dist(0, \partial f(y))$ are all defined on a manifold. 
Theorem~\ref{RPG:th3} provides an approach to verify if a function on a manifold satisfies the Riemannian KL property based on a chart of the manifold and the Euclidean KL property. This theorem is a slight generalization of \cite[Theorem~4.3]{BCO2011}, where we directly impose the condition that $F \circ \phi^{-1}$ satisfies the Euclidean KL property rather than first require that $F$ is a continuous $\mathcal{C}$-function on a manifold. It enables us to %avoid the discussion of the abstract function property on a manifold  when establishing the Riemannian KL property of a semialgebraic function on the Stiefel manifold in Section~\ref{sec:KLStiefel}, so that we can prove the result only based on the basic semialgebraic properties of Euclidean functions. 
establish the Riemannian KL property of a semialgebraic function on the Stiefel manifold (see Section~\ref{sec:KLStiefel}) without first resorting to the discussion of the abstract manifold property of the function, but only based on the basic semialgebraic properties of the Euclidean function.
Here we include the short proof of Theorem~\ref{RPG:th3} for the presentation to be self-contained. 

\begin{theorem}~\label{RPG:th3}
Given $x \in \mathcal{M}$, let $(\phi, \mathcal{U})$ denote a chart of $\mathcal{M}$ covering $x$, i.e., $x \in \mathcal{U}$. We assume that $F \circ \phi^{-1}: \mathbb{R}^{d} \rightarrow \mathbb{R}$ satisfies the Euclidean KL property at $\phi(x)$ with the desingularising function $\tilde{\varsigma}_x$, then $F$ satisfies the Riemannian KL property at $x$ with the desingularising function $\tilde{\varsigma}_x / C_x$, where $C_x$ is a constant.
\end{theorem}
\begin{proof}
Let $\tilde{F}$ denote $F \circ \phi^{-1}$. Since $\tilde{F}$ satisfies the Euclidean KL property at $\tilde{x} := \phi(x)$, there exists $\varepsilon > 0$, a neighborhood $\tilde{\mathcal{U}}$ of $\tilde{x}$ and a desingularising function $\tilde{\varsigma}:[0, \varepsilon] \rightarrow [0, \infty)$ such that for every $\tilde{y} \in \tilde{\mathcal{U}} \cap \phi(\mathcal{U}) \cap \{\tilde{z} \mid \tilde{F}(\tilde{x}) < \tilde{F}(\tilde{z}) < \tilde{F}(x) + \varepsilon \}$
\begin{equation} \label{RPG:e69}
\tilde{\varsigma}'(\tilde{F}(\tilde{y}) - \tilde{F}(\tilde{x})) \dist(0, \partial \tilde{F}(\tilde{y})) \geq 1.
\end{equation}
Let $y$ denote $\phi^{-1}(\tilde{y})$. Since $\phi$ is a bijection in $\mathcal{U}$, we have that $y$ can be any point in $\phi^{-1}(\tilde{\mathcal{U}} \cap \phi(\mathcal{U}) ) \cap \{z \mid F(x) < F(z) < F(x) + \varepsilon \}$.
Note that $\partial F(y) = [\D \phi(y)]^{\sharp} [\partial \tilde{F} (\tilde{y})]$ by~\cite[Proposition~2.5]{HP2011}, where $\sharp$ denotes the adjoint operator. Inequality~\eqref{RPG:e69} becomes
\[
\tilde{\varsigma}'(F(y) - F(x)) \dist(0, [\D \phi(y)]^{-\sharp} \partial F(y)) \geq 1.
\]
Since $\phi$ is a diffeomorphism, there exists a positive constant $c_0$ such that $\|[\D \phi(x)]^{-\sharp}\| \leq c_0$. Therefore, there exists a neighborhood $\mathcal{W}_x$ of $x$ such that $\|[\D \phi(z)]^{-\sharp}\| \leq 2 c_0$ for all $z \in \mathcal{W}_x$. Thus, for any $y \in \mathcal{W}_x \cap \phi^{-1}(\tilde{\mathcal{U}} \cap \phi(\mathcal{U}) ) \cap \{z \mid F(x) < F(z) < F(x) + \varepsilon \}$, it holds that 
\begin{align*}
\tilde{\varsigma}'(F(y) - F(x)) \dist(0, \partial F(y)) \geq & \frac{1}{2 c_0} %\tilde{\varsigma}'(F(y) - F(x)) \dist(0, [\D \phi(y)]^{\sharp} [\D \phi(y)]^{-\sharp} \partial F(y)) \\ \geq& \frac{c_0}{2} 
\tilde{\varsigma}'(F(y) - F(x)) \dist(0, [\D \phi(y)]^{-\sharp} \partial F(y)) \geq \frac{1}{2 c_0}.
\end{align*}
It follows that $F$ satisfies the Riemannian KL property at $x$ with desingularising function~$2 c_0 \tilde{\varsigma}$.
\end{proof}

A Riemannian generalization of the uniformized Euclidean KL property \cite[Lemma 6]{BST2014} is given in Lemma~\ref{RPG:le15}. It shows that if the Riemannian KL property holds for every single point in a compact set with the same function value, then there exists a single desingularising function such that the Riemannian KL property holds for all points in the compact set.
Note that this generalization also appears implicitly in the proof of \cite[Theorem~4.8]{Hosseini2017}.

\begin{lemma} \label{RPG:le15}
Let $\bar\Omega$ be a compact set in $\mathcal{M}$ and let $\sigma: \mathcal{M} \rightarrow (-\infty, \infty]$ be a continuous %proper and lower semicontinuous 
function. Assume that $\sigma$ is constant on $\bar\Omega$ and satisfies the Riemannian KL property at each point of $\bar\Omega$. Then, there exist $\varpi > 0$, $\varepsilon > 0$ and a continuous concave function $\varsigma: [0, \varepsilon] \rightarrow [0, \infty)$ such that for all $\bar{u}$ in $\bar\Omega$ and all $u$ in the following intersection:
\[
\{ u \in \mathcal{M} : \dist(u, \bar\Omega) < \varpi \} \cap \{u \in \mathcal{M}: \sigma(\bar{u}) < \sigma(u) < \sigma(\bar{u}) + \varepsilon \},
\]
one has
\[
\varsigma'(\sigma(u) - \sigma(\bar{u})) \dist(0, \partial \sigma(u)) \geq 1.
\]
\end{lemma}
\begin{proof}
Let $\sigma^*$ be the value of $\sigma$ over $\bar\Omega$. Let the compact set $\bar\Omega$ be covered by a finite number of open balls $\mathbb{B}(u_i, \varpi_i)$ (with $u_i \in \bar\Omega$ for $i = 1, \ldots, p$) on which the Riemannian KL property holds. For each $i = 1, \ldots, p$, we denote the corresponding desingularising function by $\varsigma_i:[0, \varepsilon_i) \rightarrow [0, \infty)$ with $\varepsilon_i > 0$. For each $u \in \mathbb{B}(u_i, \varpi_i) \cap \{u \mid \sigma^* < \sigma(u) < \sigma^* + \varepsilon_i\}$, we have
\[
\varsigma_i'(\sigma(u) - \sigma(u_i)) \dist(0, \partial \sigma (u)) = \varsigma_i'( \sigma(u) - \sigma^* ) \dist(0, \partial \sigma (u)) \geq 1.
\]
Choose $\varpi$ sufficiently small so that
\[
\mathcal{U}_{\varpi} := \left\{ x \in \mathcal{M} \mid \dist(x, \bar\Omega) \leq \varpi \right\} \subset \cup_{i = 1}^p \mathbb{B}(u_i, \varpi_i).
\]
Let $\varepsilon = \min(\varepsilon_i, i = 1, \ldots, p) > 0$ and 
\[
\varsigma(s) = \sum_{i = 1}^p \varsigma_i(s), \quad \forall s \in [0, \varepsilon).
\]
It follows that for all $u \in \mathcal{U}_{\varpi} \cap \{ u \mid \sigma^* < \sigma(u) < \sigma^* + \varepsilon \}$, it holds that
\[
\varsigma'(\sigma(u) - \sigma^*) \dist(0, \partial \sigma(u)) = \sum_{i = 1}^p \varsigma_i'(\sigma(u) - \sigma^*) \dist(0, \partial \sigma(u)) \geq 1,
\]
which completes the proof.
\end{proof}

%\whfirrev{[TODO: $\varsigma$ function can be chosen to be $\frac{C}{\theta} t^{\theta}$.]}{}

%n~\cite{Hosseini2017}, the Riemannian KL property has been used to analyze the convergence of a subgradient-oriented descent optimization algorithm for Lipschitz continuous functions on a Riemannian manifold. It has been shown that every accumulation point of the algorithm is a stationary point. Moreover, if the objective function has the Riemannian KL property, then the accumulation point is unique. Next we will show that this is also true for the Riemannian proximal method described in Algorithm~\ref{RPG:a1}.

{
Assumption~\ref{RPG:as19} will be used for the convergence analysis in this subsection. When the manifold $\mathcal{M}$ is the Euclidean space, such assumption has been made in e.g.,~\cite{LL2015}.
\begin{assumption} \label{RPG:as19}
$f:\mathcal{M} \rightarrow \mathbb{R}$ is locally Lipschitz continuously differentiable. %, and $g:\mathcal{M} \rightarrow \mathbb{R}$ is continuous.
\end{assumption}
}

{
In order to study the convergence analysis of Algorithm~\ref{RPG:a1} based on the Riemannian KL property, \whsecrev{}{we also need two results regarding to the retraction and vector  transport, given in Lemmas~~\ref{RPG:le16} and \ref{RPG:le19},   respectively. The proofs of these two  lemmas will be deferred to Appendix~\ref{app:proofs}.} Note that Lemma~\ref{RPG:le16} is a variant of~\cite[Proposition~7.4.5, Corollary~7.4.6]{AMS2008}, and the proof is partially the same as that of \cite[Proposition~7.4.5]{AMS2008}.   }
\begin{lemma} \label{RPG:le16}
{ \whsecrev{}{Let $\bar\Omega \subset \mathcal{M}$ be a compact set.} Then for any given $\delta_T > 0$, %Then there exists a positive constant $\delta_T$ such that for any $x \in \bar{\Omega}$, the set $\mathbb{B}(x, \delta_T)$ is totally restrictive. In addition, 
there exists a positive constant $\kappa$ such that $\dist(x, R_x(\eta_x)) \leq \kappa \|\eta_x\|_x$ for all $x \in \bar{\Omega}$ and for all $\eta_x \in \mathcal{B}(0_x, \delta_T)$.
%Then there exists a positive constant $\delta_T$ such that for any $x \in \bar\Omega$, $\eta_x \in \T_x \mathcal{M}$ satisfying , there exists a constant $\kappa$ such that $\dist(x, y) \leq \kappa \|\eta_x\|_x$, where $\eta_x = R_x^{-1}(y)$.
}
\end{lemma}

%Lemma~\ref{RPG:le19} will be used in Theorem~\ref{RPG:single} and Theorem~\ref{RPG:LocalRateKL} for the inequality~\eqref{RPG:e96}.

\begin{lemma} \label{RPG:le19}
{Let $\xi$ be a locally Lipschitz continuous vector field on $\mathcal{M}$. Given a constant $a$ and a compact set $\bar\Omega \subset \mathcal{M}$, there exist positive constants $\mu$ and $L_c$ such that $\|\xi_y - \mathcal{T}_{R_{\eta_x}}^{- \sharp} (\xi_x + a \eta_x)\|_y \leq L_c \|\eta_x\|_x$ for any $x$ and $\eta_x \in \T_x \mathcal{M}$ satisfying $\|\eta_x\|_x < \mu$, where $y = R_x(\eta_x)$ and $\mathcal{T}_R$ is the vector transport by differentiated the retraction $R$. %$y$ satisfying $\dist(x, y) < \delta$.
}
\end{lemma}

%\kwnote{Give the convergence result below}{}

Now, we are in position to  show the convergence of the iterates $\{x_k\}$ generated by Algorithm~1 to a single stationary point. The structure of the proof follows the one  for \cite[Theorem~1]{BST2014}.
\begin{theorem}\label{RPG:single}
Let $\{x_k\}$ denote the sequence generated by Algorithm~\ref{RPG:a1} and $\mathcal{S}$ denote the set of all accumulation points.
Suppose Assumptions~\ref{RPG:as10}, ~\ref{RPG:as3} and~\ref{RPG:as19} hold.  %, $F(x) \rightarrow \infty$ if $\|x\| \rightarrow \infty$
We further assume that $F = f + g$ satisfies the Riemannian KL property at every point in $\mathcal{S}$. Then, 
\begin{align*}
\sum_{k=0}^\infty \dist(x_k,x_{k+1})<\infty. \numberthis\label{eq:finite}
\end{align*}
It follows that $\mathcal{S}$ only contains a single point.
\end{theorem}
\begin{proof}
First note that the global convergence result in Theorem~\ref{RPG:globaltheo} implies that every point in $\mathcal{S}$
is a stationary point. Since  $\lim_{k \rightarrow \infty} \|\eta_{x_k}^*\|_{x_k} = 0$, there exists a $\delta_T>0$ such that $\|\eta_{x_k}^*\|_{x_k}\leq\delta_T$ for all $k$. {Thus, the application of Lemma~\ref{RPG:le16} implies that
\begin{align*}
\dist(x_{k},x_{k+1})=\dist(x_k,R_{x_k}(\eta_{x_k^*}))\leq\kappa \|\eta_{x_k}^*\|_{x_k}\rightarrow 0.\numberthis\label{eq:kwrev01a}
\end{align*}
Then by \cite[Remark 5]{BST2014}, we know that $\mathcal{S}$ is a compact set.
} 
Moreover, since $F(x_k)$ is nonincreasing and $F$ is continuous, $F$ has the same value at all the points in $\mathcal{S}$.  Therefore, by Lemma~\ref{RPG:le15}, there exists a single  desingularising function, denoted $\varsigma$, for the Riemannian KL property  of $F$ to hold at all the points in $\mathcal{S}$.

Let $x_*$ be a point in $\mathcal{S}$. Assume there exists $\bar k$ such that $x_{\bar k}=x_*$. Since $F(x_k)$ is non-increasing, it must hold $F(x_{\bar k})=F(x_{{\bar k}+1})$. By Lemma~\ref{RPG:le3}, we have $\eta_{x_{\bar k}}^*=0$,  $x_{\bar k}=x_{{\bar k}+1}$, \eqref{eq:finite} holds evidently. 

In the case when $F(x_k)>F(x_*)$ for all $k$, since $F(x_k)\rightarrow F(x_*)$, $\dist(x_k,\mathcal{S})\rightarrow 0$, by the Riemannian KL property of $F$ on $\mathcal{S}$, there exists an $l>0$  such that 
\begin{align*}
\varsigma'(F(x_k) - F(x_*)) \dist(0, \partial F(x_k)) \geq 1\quad\mbox{for all }k>l.%\numberthis\label{eq:kwrev01}
\end{align*}
It follows that 
\begin{align*}
\varsigma'(F(x_k) - F(x_*))\geq \dist(0, \partial F(x_k))^{-1}\quad\mbox{for all }k>l.\numberthis\label{eq:kwrev01}
\end{align*}
Since $\lim_{k \rightarrow \infty} \|\eta_{x_k}^*\|_{x_k} = 0$, there exists a constant $\hat{k}_0 > 0$ such that $\|\eta_{x_k}^*\|_{x_k} < \mu$ for all $k > \hat{k}_0$, where $\mu$ is defined in Lemma~\ref{RPG:le19}. Therefore, we have
%Since $\grad f$ is a locally Lipschitz continuous vector field, it follows from~\eqref{RPG:e82} and Lemma~\ref{RPG:le10} that
\begin{align} \label{RPG:e96}
\|\grad f(R_{x_{k}}(\eta_{x_{k}}^*)) - \mathcal{T}_{R_{\eta_{x_{k}}^*}}^{- \sharp} (\grad f(x_{k}) + \tilde{L} \eta_{x_{k}}^*)\|_{R_{x_{k}}(\eta_{x_{k}}^*)} \leq L_c \|\eta_{x_{k}}^*\|_{x_{k}}
\end{align}
for all $k \geq \hat{k}_0$. By~\eqref{RPG:e26}, it holds that
\begin{equation} \label{RPG:e97}
	\grad f(R_{x_{k}}(\eta_{x_{k}}^*)) - \mathcal{T}_{R_{\eta_{x_{k}}^*}}^{- \sharp} (\grad f(x_{k}) + \tilde{L}  \eta_{x_{k}}^*) = \grad f(x_{k + 1}) + \zeta_{x_{k+1}} \in \partial F(x_{k + 1}).
\end{equation}
Therefore,~\eqref{RPG:e96} and~\eqref{RPG:e97} yield
\begin{equation} \label{RPG:e66}
\dist(0, \partial F(x_k)) \leq L_c \|\eta_{x_{k-1}}^*\|_{x_{k-1}},
\end{equation}
for all $k > \hat{k}_0$. Inserting this into \eqref{eq:kwrev01} gives 
\begin{align*}\varsigma'(F(x_k) - F(x_*))\geq L_c^{-1}\|\eta_{x_{k-1}}^*\|_{x_{k-1}}^{-1}\quad\mbox{for all }k>\hat{l}:=\max(k_0,l).\numberthis\label{eq:kerev02}
\end{align*}
Moreover, the concavity of $\varsigma$ yields that 
\begin{align}
\varsigma(F(x_k)-F(x_*))-\varsigma(F(x_{k+1})-F(x_*))&\geq \varsigma'(F(x_k)-F(x_*))(F(x_k)-F(x_{k+1}))\\
&\geq L_c^{-1}\beta\frac{\|\eta_{x_{k}}^*\|_{x_{k}}^2}{\|\eta_{x_{k-1}}^*\|_{x_{k-1}}}\quad\mbox{for all }k>\hat{l}:=\max(k_0,l),
\label{RPG:e99}\end{align}
where the second inequality follows from Lemma~\ref{RPG:le3} and \eqref{eq:kerev02}. Finally, the same algebra manipulation as in the proof of  \cite[Theorem~1]{BST2014} yields that 
\begin{align*}
\sum_{k=0}^\infty\|\eta_{x_{k}}^*\|_{x_{k}}<\infty,
\end{align*}
and \eqref{eq:finite} follows immediately due to \eqref{eq:kwrev01a}.

\end{proof}
%%%%%%%%%%%%%%%%%%%%%%
Similar to the Euclidean case, if $F$ further satisfies the Riemannian KL property with the desingularising function being of the form\footnote{{When the desingularising function has the form $\varsigma(t) = \frac{C}{\theta} t^{\theta}$ for some $C > 0$, $\theta \in (0, 1]$, we say that $F$ satisfies the Riemannian KL property with an exponent $\theta$, as in the Euclidean case. }} $\varsigma(t) = \frac{C}{\theta} t^{\theta}$ for some $C > 0$, $\theta \in (0, 1]$, then the local convergence rate can be established. \whsecrev{}{The proof is overall similar to that  for the Euclidean case, and it is included in Appendix~\ref{appen:localKL} for completeness.}
\begin{theorem} \label{RPG:LocalRateKL}
Let $\{x_k\}$ denote the sequence generated by Algorithm~\ref{RPG:a1} and $\mathcal{S}$ denote the set of all accumulation points.
Suppose Assumptions~\ref{RPG:as10}, ~\ref{RPG:as3} and~\ref{RPG:as19} hold.  %, $F(x) \rightarrow \infty$ if $\|x\| \rightarrow \infty$
We further assume that $F = f + g$ satisfies the Riemannian KL property at every point in $\mathcal{S}$ with the desingularising function having the form of $\varsigma(t) = \frac{C}{\theta} t^{\theta}$ for some $C > 0$, $\theta \in (0, 1]$. The accumulation point, denoted $x_*$, is unique by Theorem~\ref{RPG:single}. Then
\begin{itemize}
\item If $\theta = 1$, then there exists $k_1$ such that $x_k = x_*$ for all $k > k_1$.
\item If $\theta \in [\frac{1}{2}, 1)$, then there exist constants $C_r > 0$ and $d \in (0, 1)$ such that for all $k$
\[
\dist(x_k, x_*) < C_r d^{k};
\]
\item If $\theta \in (0, \frac{1}{2})$, then there exists a positive constant $\tilde{C}_r$ such that for all $k$
\[
\dist(x_k, x_*) < \tilde{C}_r k^{\frac{-1}{1 - 2 \theta}}.
\]
\end{itemize}
\end{theorem}

%%%%%%%%%%%%%%
%\subsection{\whfirrev{}{A Case Study}}
%%%%%%%%%%%%%%

\subsection{\whfirrev{}{Restriction of  Semialgebraic Function onto  Stiefel Manifold satisfies  Riemannian KL }} \label{sec:KLStiefel}

It has been shown in e.g.,~\cite{ABRS2010,BST2014} that a semialgebraic function on $\mathbb{R}^n$ has the Euclidean KL property. A natural question is:  Given a submanifold $\mathcal{M}$ of $\mathbb{R}^n$ and a semialgebraic function $F$ on $\mathbb{R}^n$, does the function defined by restricting $F$ onto $\mathcal{M}$ have the Riemannian KL property? In this section, we will \whsecrev{}{give a sufficient condition under which the answer is positive and then verify the condition for $\mathcal{M}$ being the Stiefel manifold.} % give an affirmative answer when $\mathcal{M}$ is the Stiefel manifold. 
\whsecrev{}{We emphasize the Stiefel manifold here since it is used in our experiments, see Section~\ref{RPG:sect:NumExp}. }

The definitions of semialgebraic sets, mappings and functions are given in Definition~\ref{RPG:def:Semi}. More can be found in e.g.,~\cite{BCR1998}.
\begin{definition}[Semialgebraic sets, mappings and functions] \label{RPG:def:Semi}
\begin{enumerate}
\item A subset $\mathcal{S}$ of $\mathbb{R}^n$ is called semialgebraic if there exists a finite number of polynomial function $g_{i j}, h_{i j}: \mathbb{R}^n \rightarrow \mathbb{R}$ such that
\[
\mathcal{S} = \cup_{j = 1}^p \cap_{i = 1}^q \{ u \in \mathbb{R}^n \mid g_{i j}(u) = 0 \hbox{ and } h_{i j}(u) < 0\}.
\]
\item Let $\mathcal{A} \subseteq \mathbb{R}^m$ and $\mathcal{B} \subseteq \mathbb{R}^n$ be two semialgebraic sets. A mapping $: \mathcal{A} \rightarrow \mathcal{B}$ is semialgebraic if its graph is semialgebraic in $\mathbb{R}^{m + n}$. If $n = 1$, then the mapping is also called a semialgebraic function.
\end{enumerate}
\end{definition}
%Stability properties of semialgebraic sets and mappings are numerous. In Proposition~\ref{RPG:pr1}, we give the properties that will be used in this paper. 

The following properties about semialgebraic sets and mappings will be used later. 
Their proofs can be found in e.g.,~\cite{BCR1998,BDL2007,BDLS2007,ABRS2010,BST2014}.
\begin{proposition}[Properties of semialgebraic sets and mappings] \label{RPG:pr1}
\begin{enumerate}
\item \label{RPG:semi:p1} Generalized inverse of semialgebraic mappings are \whsecrev{}{semialgebraic};
\item \label{RPG:semi:p2} Composition of semialgebraic functions or mappings are semialgebraic;
\item \label{RPG:semi:p3} Continuous semialgebraic functions satisfy the Euclidean KL property with desingularising function in the form of $\varsigma(t) = \frac{C}{\theta} t^\theta$, where $\theta \in (0, 1]$ and $C > 0$.
\item \label{RPG:semi:p4} Let $\mathcal{S}$ be a semialgebraic set of $\mathbb{R}^{m + n}$ and $\pi: \mathbb{R}^{m + n} \rightarrow \mathbb{R}^m$ be the projection on the space of the first $m$ coordinates. Then $\pi(\mathcal{S})$ is a semialgebraic set of $\mathbb{R}^m$.
\end{enumerate}
\end{proposition}
%For simplicity, we only consider the Euclidean metric $\inner[X]{\eta_X}{\xi_X} = \trace(\eta_X^T \xi_X)$ on the Stiefel manifold. It is pointed out that the same approach can be used
%
%The Euclidean metric and the canonical metric are two popular Riemannian metric on $\St(p, n)$ and they are respectively defined by
%\[
%\hbox{Euclidean: } \inner[X]{\eta_X}{\xi_X}^{\mathrm{E}} = \trace(\eta_X^T \xi_X), \hbox{ Canonical: } \inner[X]{\eta_X}{\xi_X}^{\mathrm{C}} = \trace(\eta_X^T \left(I - \frac{1}{2} X X^T \right) \xi_X),
%\]
%where $\eta_X, \xi_X \in \T_X \St(p, n)$. Note that the Riemannian metrics are also legitimate metric for any $\eta_X, \xi_X$ in $\mathbb{R}^{n \times p}$. Therefore, the normal space $N_X \St(p, n)$, which is the orthogonal complement space of $\T_x \St(p, n)$, can be defined.

\whsecrev{}{
\begin{theorem} \label{RPG:th6}
Let $F: \mathbb{R}^{n} \rightarrow \mathbb{R}$ be a continuous semialgebraic function, $\mathcal{M}$ be an $d$-dimensional embedded submanifold of $\mathbb{R}^n$. For any $x \in \mathcal{M}$, let $(\mathcal{U}, \phi_x)$ be a chart covering $x$. If $\phi_x^{-1}: \phi_x(\mathcal{U}) \subset \mathbb{R}^d \rightarrow \mathbb{R}^n$ is a semialgebraic mapping, then the function by restricting $F$ onto $\mathcal{M}$, $\tilde{F}: \mathcal{M} \rightarrow \mathbb{R}$, has the Riemannian KL property at any point $X$ of $\mathcal{M}$ with desingularising function in the form of $\varsigma(t) = \frac{C_X}{\theta_X} t^{\theta_X}$, where $\theta_X \in (0, 1]$ and $C_X > 0$.
\end{theorem}

\begin{proof}
Since $F$ is a semialgebraic function and $\phi^{-1}$ is a semialgebraic mapping, it follows from ~\ref{RPG:semi:p2} in Proposition~\ref{RPG:pr1} that $F \circ \phi_X^{-1}$ is a semialgebraic function. 
Since the image of $\phi_X^{-1}$ is in $\St(p, n)$ and $\tilde{F}$ is the restriction of $F$ to $\St(p, n)$, we have $F \circ \phi_{X}^{-1} = \tilde{F} \circ \phi_X^{-1}$. It follows that $\tilde{F} \circ \phi_X^{-1}$ is a semialgebraic function, and thus satisfies the Euclidean KL property at any points in the domain $\phi_X(\mathcal{U} \cap \mathcal{M})$ with a desingularising function of the form $\frac{\tilde{C}}{\theta} t^{\theta}$ for certain $\tilde{C} > 0$ and $\theta \in (0, 1]$, see \ref{RPG:semi:p3} in Proposition~\ref{RPG:pr1}. %Suppose the desingularising function at $\phi_X(X)$ is $\tilde\varsigma(t) = \frac{\tilde{C}_X}{\theta_X} t^{\theta_X}$, where $\theta_X \in (0, 1]$ and $\tilde{C}_X > 0$. 
Together with Theorem~\ref{RPG:th3}, the proof is completed.
\end{proof}
}

\whsecrev{}{Next, we will prove that the assumptions in Theorem~\ref{RPG:th6} hold when the manifold is the Stiefel manifold. Specifically, for any $x$ in the Stiefel manifold, we will construct a chart of $x$ that satisfies the condition in Theorem~\ref{RPG:th6}.} The Stiefel manifold is $\St(p, n) = \{X \in \mathbb{R}^{n \times p} \mid X^T X = I_p \}$. The tangent space of $\St(p, n)$ at $X$ is $\T_X \St(p, n) = \{V \in \mathbb{R}^{n \times p} \mid X^T V + V^T X = 0\}$. %$\T_X \St(p, n) = \{V \in \mathbb{R}^{n \times p} \mid X \Omega + X_{\perp} K, \Omega = - \Omega^T, K \in \mathbb{R}^{(n - p) \times p} \}$, where $X_{\perp} \in \mathbb{R}^{n \times (n - p)}$ is an orthonormal matrix such that $X^T X_{\perp} = 0$. 
%The two commonly-used Riemannian metrics on $\St(p, n)$ are the Euclidean metric and the canonical metric, which are respectively defined by
%\[
%\hbox{Euclidean: } \inner[X]{\eta_X}{\xi_X}^{\mathrm{E}} = \trace(\eta_X^T \xi_X), \hbox{ Canonical: } \inner[X]{\eta_X}{\xi_X}^{\mathrm{C}} = \trace(\eta_X^T \left(I - \frac{1}{2} X X^T \right) \xi_X),
%\]
We consider the Riemannian metric  inherited from the Euclidean metric,
\[
\inner[X]{\eta_X}{\xi_X} = \trace(\eta_X^T \xi_X),
\]
where $\eta_X, \xi_X \in \T_X \St(p, n)$. %For simplicity, we focus on the Euclidean metric in this section and the differences of derivations to the canonical metric are highlighted in footnotes. 
The normal space, which is the orthonormal complement space of $\T_X \St(p, n)$, is given by $\N_x \St(p, n) = \{X S \mid S = S^T\}$. % \footnote{The normal space with respect to the canonical metric is the same as that to the Euclidean metric}.

%Next, we will construct a chart of $\St(p, n)$ when $\St(p, n)$ is viewed as a submanifold of $\mathbb{R}^{n \times p}$. 
The construction \whsecrev{}{of the chart for the Stiefel manifold} relies on the following result.% given in 
%Lemma~\ref{RPG:le17}. % gives a result used in Lemma~\ref{RPG:le18}.
%Lemma~\ref{RPG:le17} constructs a chart of $\mathbb{R}^{n \times p}$ on every point on $\St(p, n)$. Lemma~\ref{RPG:le18} further uses the submanifold property in~\cite[Proposition~3.3.2]{AMS2008} to construct a chart of $\St(p, n)$ when $\St(p, n)$ is viewed as a submanifold of $\mathbb{R}^{n \times p}$. Note that charts are independent of Riemannian metrics. Therefore, even though we use the Euclidean metric to construct a chart of $\St(p, n)$, the chart is still legal for another Riemannian metric on $\St(p, n)$.
\begin{lemma} \label{RPG:le17}
Let $X \in \St(p, n)$, $B_X \in \mathbb{R}^{np \times (np - \frac{1}{2} p (p+1))}$ be an orthonormal basis of $\T_X \St(p, n)$, and $H_X \in \mathbb{R}^{np \times \frac{1}{2} p (p+1))}$ be an orthonormal basis of $\N_x \St(p, n)$. Then there exists a positive constant $\delta_X$ such that the mapping
\begin{align*}
\tilde{\phi}_X&: \mathfrak{B}(X, \delta_X) \rightarrow \mathbb{R}^{np} \\
&: Y \mapsto \begin{bmatrix} B_X & H_X \end{bmatrix}^T \left( \frac{1}{2} \vvec( X (Y^T Y - I_p) ) + (I_p \otimes Y ) (I_p \oplus (X^T Y) )^{-1} \vvec(2 I_p) - \vvec(X) \right)
\end{align*}
is a diffeomorphism, where $\mathfrak{B}(X, \delta_X) = \{Y \in \mathbb{R}^{n \times p} \mid \|Y - X\|_{\E} < \delta_X \}$, $\otimes$ denotes the Kronecker product, $\oplus$ denotes the Kronecker sum, and $\vvec$ denotes the operation that stacks the columns of its matrix arguments into a single vector. % \footnote{If the canonical metric is used, then $\begin{bmatrix} B_X & H_X \end{bmatrix}^T$ need be replaced by $\begin{bmatrix} B_X & H_X \end{bmatrix}^{\flat}$, where $\flat$ denotes an operator such that $\eta_x^{\flat}: \T_x \mathcal{M} \rightarrow \mathbb{R}: \xi_x \mapsto \inner[x]{\eta_x}{\xi_x}$.}.
\end{lemma}
%This lemma implies that the pair $(\mathcal{U}_X, \tilde{\phi}_X )$ is a chart of $\mathbb{R}^{n \times p}$.
\begin{proof}
For any $V \in \mathbb{R}^{n \times p}$, we have
\begin{align*}
\D \tilde{\phi}_X(Y)[V] =& \begin{bmatrix} B_X & H_X \end{bmatrix}^T \Big( \frac{1}{2} \vvec( X (Y^T V + V^T Y) ) + (I_p \otimes V ) (I_p \oplus (X^T Y) )^{-1} \vvec(2 I_p) \\
&- (I_p \otimes Y ) (I_p \oplus (X^T Y) )^{-1} (I_p \oplus X^T V ) (I_p \oplus (X^T Y) )^{-1}\vvec(2 I_p) \Big).
\end{align*}
It follows that
\[
\D \tilde{\phi}_X(Y)[V] |_{Y = X} = \begin{bmatrix} B_X & H_X \end{bmatrix}^T \vvec(V),
\]
which implies $J_{\tilde{\phi}_X}(Y)$ at $Y = X$ is a surjective operator. Therefore, the determinant of $J_{\tilde{\phi}_X}(X)$ is nonzero. \whsecrev{}{The conclusion follows from the inverse function theorem.} % By the inverse function theorem, there exists a neighborhood of $X$, denoted by $\hat{\mathcal{U}}_X$, such that the mapping $\tilde{\phi}_X$ is invertible in the neighborhood. Note that $\tilde{\phi}_X$ is smooth in a neighborhood of $X$ and we denote this neighborhood by $\tilde{\mathcal{U}}_X$. Therefore, $\tilde{\phi}_X$ is a diffeomorphism in $\hat{\mathcal{U}}_X \cap \tilde{\mathcal{U}}_X$. The conclusion holds by choosing a sufficiently small $\delta_X$ such that $\mathfrak{B}(X, \delta_X) \subset \hat{\mathcal{U}}_X \cap \tilde{\mathcal{U}}_X$.
\end{proof}

\whsecrev{}{We are now ready to construct  the chart of $\St(p, n)$  that satisfies the condition in Theorem~\ref{RPG:th6}.}
\begin{lemma}[A chart of $\St(p, n)$] \label{RPG:le18}
The pair $(\mathcal{W}_X, \phi_X)$ is a chart of the embedded submanifold $\St(p, n)$, where $\mathcal{W}_X = \mathfrak{B}(X, \delta_X) \cap \St(p, n)$, $\phi_X = E^T \tilde{\phi}_X$, $\delta_X$ and $\tilde{\phi}_X$ are defined in Lemma~\ref{RPG:le17}, $E = \begin{bmatrix} e_1 & e_2 & \ldots & e_{np - \frac{1}{2} p (p + 1))} \end{bmatrix} \in \mathbb{R}^{np \times (np - \frac{1}{2} p (p + 1))}$ with $e_j$ being the $j$-th canonical basis of $\mathbb{R}^{np}$.
In addition, the inverse of $\phi_X$ is
\[
\phi_X^{-1}: \mathcal{E} \rightarrow \St(p, n): v \mapsto \tilde{\phi}_X^{-1} E v = R_X(B_X v),
\]
where 
$
\mathcal{E} = \left\{v \in \mathbb{R}^{np - \frac{1}{2} p (p + 1)} \mid 
\begin{bmatrix} v \\ 0 \end{bmatrix} \in \tilde{\phi}_X(\mathfrak{B}(X, \delta_X))
%\hbox{ the last $\frac{1}{2} p (p + 1)$ of $z$ are zeros } i.e., z_{i} = 0, i = np - \frac{1}{2} p(p + 1) + 1, \ldots, np - 1, np 
\right\}$,
and $R_X(\eta_X) = (X + \eta_X) (I_p + \eta_X^T \eta_X)^{-1/2}$ is the retraction by the polar decomposition~\cite[(4.7)]{AMS2008}. \whsecrev{}{Moreover, $\phi_X^{-1}$ is a semialgebraic mapping.}
%The collection of charts $\mathcal{A} = \{ (\mathcal{W}_X, \phi_X) \mid X \in \St(p, n) \}$ forms an atlas of $\St(p, n)$, where $\mathcal{W}_X = \mathcal{U}_X \cap \St(p, n)$, $\phi_X = E^T \tilde{\phi}_X$, $\mathcal{U}_X$ and $\tilde{\phi}_X$ are defined in Lemma~\ref{RPG:le17}, $E = \begin{bmatrix} e_1 & e_2 & \ldots & e_{np - \frac{1}{2} p (p + 1))} \end{bmatrix} \in \mathbb{R}^{np \times (np - \frac{1}{2} p (p + 1))}$ and $e_j$ denotes the $j$-th canonical basis of $\mathbb{R}^{np}$. %The maximal atlas generated by $\mathcal{A}$ is the differentiable structure that makes $\St(p, n)$ an embedded submanifold of $\mathbb{R}^{n \times p}$.
\end{lemma}
\begin{proof}
The proof relies on the submanifold property given in~\cite[Proposition~3.3.2]{AMS2008}. Note that $\tilde{\phi}_X$ is a chart of $\mathbb{R}^{n \times p}$, which is the embedding space of $\St(p, n)$. We only need to show that for any $Y$ in $\mathcal{U}_X \cap \St(p, n)$, the last $\frac{1}{2} p (p + 1)$ entries of $\tilde{\phi}_X(Y)$ are zeros \whthirev{}{, and for any $Z \notin \mathcal{U}_X \cap \St(p, n)$, the last $\frac{1}{2} p (p + 1)$ entries of $\tilde{\phi}_X(Y)$ are not all zeros}. To the end, for any $Y \in \mathcal{U}_X \cap \St(p, n)$, it holds that
\begin{align*}
\tilde{\phi}_X(Y) =& \begin{bmatrix} B_X & H_X \end{bmatrix}^T \left( (I_p \otimes Y ) (I_p \oplus (X^T Y) )^{-1} \vvec(2 I_p) - \vvec(X) \right) \\
=& \begin{bmatrix} B_X & H_X \end{bmatrix}^T \vvec \left( Y S - X \right),
\end{align*}
where $S$ is the solution of the Lyapunov equation $(X^T Y) S + S (Y^T X) = 2 I_p$. Since $X^T (Y S - X) + (Y S - X)^T X = 0$, we have $Y S - X \in \T_X \St(p, n)$. Therefore, it holds that $H_X^T \vvec(Y S - X) = 0$, which implies that the last $\frac{1}{2} p (p + 1)$ entries of $\tilde{\phi}_X(Y)$ are zeros. 
\whthirev{}{If $Z \notin \mathcal{U}_X \cap \St(p, n)$, then
\begin{align*}
\tilde{\phi}_X(Z) =& \begin{bmatrix} B_X & H_X \end{bmatrix}^T \left( \frac{1}{2} \mathrm{vec}(X (Z^T Z - I_p)) +  (I_p \otimes Y ) (I_p \oplus (X^T Y) )^{-1} \vvec(2 I_p) - \vvec(X) \right) \\
=& \begin{bmatrix} B_X & H_X \end{bmatrix}^T \left( \frac{1}{2} \mathrm{vec}(X (Z^T Z - I_p)) + \vvec \left( Y \tilde{S} - X \right) \right),
\end{align*}
where $\tilde{S}$ is the solution of the Lyapunov equation $(X^T Y) S + S (Y^T X) = 2 I_p$. Since $X (Z^T Z - I_p)$ is a non-zero matrix in $\N_X \St(p, n)$, and $Y \tilde{S} - X \in \T_X \St(p, n)$, we have that the last the last $\frac{1}{2} p (p + 1)$ entries of $\tilde{\phi}_X(Z)$ are not all zeros. }
Lastly, it is easy to verify that $\phi_X \circ R_X(B_x v) = v$ for $v \in \mathcal{W}_X$. Therefore, $\phi_X^{-1}(v) = R_X(B_X v)$.

Let $\tilde{\varphi}_X$ denote the function by restricting $\tilde{\phi}_X$ onto $\mathfrak{B}(X, \delta_X) \cap \St(p, n)$. %We first show $\tilde{\varphi}_X$ is a semialgebraic function. 
The graph $\mathcal{G}_X$ of $\tilde{\varphi}_X$ is given by
\begin{align*}
\mathcal{G}_X=\{&(Y,Z) \in \mathbb{R}^{n \times {2 p}} \mid \\
&  \begin{bmatrix} B_X & H_X \end{bmatrix}^T \left( \frac{1}{2} \vvec( X (Y^T Y - I_p) ) + (I_p \otimes Y ) (I_p \oplus (X^T Y) )^{-1} \vvec(2 I_p) - \vvec(X) \right)=\vvec(Z), \\
&\|Y - X\|_{\E}^2 - \delta_X^2 < 0, Y^T Y = I_p
\},
\end{align*}
which is equivalent to
\begin{align}
\mathcal{G}_X=\{&(Y, Z) \in \mathbb{R}^{n \times {2 p}} \mid \nonumber \\
& (I_p \oplus (X^T Y) ) (I_p \otimes Y^T )\left( \begin{bmatrix} B_X & H_X \end{bmatrix} \vvec(Z) - \frac{1}{2} \vvec( X (Y^T Y - I_p) ) + \vvec(X)  \right) = \vvec(2 I_p), \nonumber \\
& (I_{np} - I_p \otimes (Y Y^T) ) \left( \begin{bmatrix} B_X & H_X \end{bmatrix} \vvec(Z) - \frac{1}{2} \vvec( X (Y^T Y - I_p) ) + \vvec(X) \right) = 0, \nonumber \\
&\|Y - X\|_{\E}^2 - \delta_X^2 < 0, Y^T Y = I_p
\}. \label{RPG:e90}
\end{align}
Since all the constraints in~\eqref{RPG:e90} are given by polynomials, the set $\mathcal{G}_X$ is semialgebraic by Definition~\ref{RPG:def:Semi}. Define the projection $\pi_Z: \mathbb{R}^{n \times 2 p} \rightarrow \mathbb{R}^{n \times p}: (Z, Y) \mapsto Z$ and the projection $\pi_Y: \mathbb{R}^{n \times 2 p} \rightarrow \mathbb{R}^{n \times p}: (Z, Y) \mapsto Y$. It follows from~\ref{RPG:semi:p4} in Proposition~\ref{RPG:pr1} that the set $\mathcal{Y}_X = \pi_Y(\mathcal{G}_X)$ and $\mathcal{Z}_X = \pi_Z(\mathcal{G}_X)$ are semialgebraic sets. Therefore, by definition, $\tilde{\varphi}_X: \mathcal{Y}_X \rightarrow \mathcal{Z}_X$ is a semialgebraic mapping. By~\ref{RPG:semi:p1} in Proposition~\ref{RPG:pr1}, its inverse $\tilde{\varphi}_X^{-1}: \mathcal{Z}_X \rightarrow \mathcal{Y}_X$ is also a semialgebraic mapping. By the definition of $\phi_X^{-1}$ in Lemma~\ref{RPG:le18}, we have $\phi_X^{-1} = \tilde{\varphi}_X^{-1} E$. Therefore, {by \ref{RPG:semi:p2} in Proposition~\ref{RPG:pr1}}, $\phi_X^{-1}$ is a semialgebraic mapping .

%By Lemma~\ref{RPG:le17}, the set $\mathcal{U}_X$ is nonempty and therefore $\cup_{X \in \St(p, n)} (\mathcal{U}_X \cap \St(p, n)) = \St(p, n)$. In addition, the domain of $\phi_X \circ \phi_Y^{-1}$ is $\phi_Y( \mathcal{W}_Y \cap \mathcal{W}_X )$. Consider any $X$ and $Y$ such that $\phi_Y( \mathcal{W}_Y \cap \mathcal{W}_X ) \neq \emptyset$. It follows that 
%\[
%\phi_X \circ \phi_Y^{-1}: v \mapsto E^T \tilde{\phi}_X \circ ( E^T \tilde{\phi}_Y)^{-1} v = E^T \tilde{\phi}_X \circ \tilde{\phi}_Y^{-1} E v.
%\]
%Since $\tilde{\phi}_X$ and $\tilde{\phi}_Y$ are differmorphisms in $\mathcal{U}_X \cap \mathcal{U}_Y \supset \mathcal{W}_X \cap \mathcal{W}_Y$, $\tilde{\phi}_X \circ \tilde{\phi}_Y^{-1}$ is smooth. It follows that the function $\phi_X \circ \phi_Y^{-1}$ is smooth. Therefore, by the definition of an atlas, see e.g.,~\cite{AMS2008,Boo1986}, $\mathcal{A}$ is an atlas of $\St(p, n)$.
\end{proof}

\whsecrev{}{
%Corollary~\ref{RPG:th5} follows from Theorem~\ref{RPG:th6} and Lemma~\ref{RPG:le18}.
Combining Theorem~\ref{RPG:th6} and Lemma~\ref{RPG:le18}, we see that the restriction of  a semialgebraic function onto the Stiefel manifold satisfies the Riemannian KL property.
}
\begin{theorem} \label{RPG:th5}
If a continuous function $F: \mathbb{R}^{n \times p} \rightarrow \mathbb{R}$ is semialgebraic, then the function by restricting $F$ onto $\St(p, n)$, $\tilde{F}: \St(p, n) \rightarrow \mathbb{R}$, has the Riemannian KL property at any point $X$ of $\St(p, n)$ with desingularising function in the form of $\varsigma(t) = \frac{C_X}{\theta_X} t^{\theta_X}$, where $\theta_X \in (0, 1]$ and $C_X > 0$.
\end{theorem}

\begin{remark}\label{remark:KL}{
Note that the result in Theorem~\ref{RPG:th5} can be extended to product of Stiefel manifolds. In other words, if a continuous function $F: \mathbb{R}^{n_1 \times p_1} \times \mathbb{R}^{n_2 \times p_2} \times \ldots \times \mathbb{R}^{n_t \times p_t}$ is semialgebraic, then the function by restricting $F$ onto $ \St(p_1, n_1) \times \St(p_2, n_2) \times \ldots \times \St(p_t, n_t) $ has the Riemannian KL property at any point with a desingularising function in the form of $\varsigma(t) = \frac{C}{\theta} t^{\theta}$. The proofs follow the same spirit and therefore are omitted here.
}
\end{remark}

\whsecrev{}{
\begin{remark}\label{remark:localdef}
%In~\cite[Definition~3.6]{boumal2020intromanifolds}, the author uses local defining functions to define an embedded submanifold of a linear space, e.g., $\mathbb{R}^n$. If local defining functions are given by polynomials, then one can construct charts in~\cite[(8.1) and (8.2)]{boumal2020intromanifolds} that satisfy the conditions of Theorem~\ref{RPG:th6}. The proofs follow the same spirit in Lemma~\ref{RPG:le18}.
For embedded submanifolds of vector spaces (e.g., $\mathbb{R}$), they can be defined based on local defining functions, see for example ~\cite[Definition~3.6]{boumal2020intromanifolds}. If local defining functions are polynomials, by the same  arguments  used  in  the  proof  of  Lemma~\ref{RPG:le18}, one can show  that the  charts  constructed in~\cite[(8.1) and (8.2)]{boumal2020intromanifolds} that satisfy the conditions of Theorem~\ref{RPG:th6}. In this case, the restriction  of a  semialgebraic function onto the submanifolds satisfies the Riemannian KL property.
\end{remark}
}

\subsection{Solving the Riemannian Proximal Mapping \kwcomm{title may be improved} \whcomm{[Any suggestions? Why do you think this title is not sufficient?]}{}} \label{RPG:sect:subproblem}

As we have mentioned already, in~\cite{CMSZ2019} Chen et. al propose a Riemannian proximal gradient method based on a different proximal mapping,
\begin{equation} \label{RPG:subprobChen}
\eta_{x_k}^* = \arg\min_{\eta \in \T_{x_k} \mathcal{M}} \inner[\E]{\grad f(x_k)}{\eta} + \frac{\tilde{L}}{2} \|\eta\|_{\E}^2 + g(x_k + \eta),
\end{equation}
where the manifold $\mathcal{M}$ is assumed to be an embedded submanifold of a Euclidean space so that the addition $x_k + \eta$ is meaningful \kwcomm{may be move this to Section~3.1}\whcomm{[As the comment in Section 3, I prefer to keep the current structure.]}{}. If $g$ is a convex function in a Euclidean space, then~\eqref{RPG:subprobChen} is a constrained convex programming problem. In particular, when  $\mathcal{M}$ is the Stiefel manifold, a semismooth Newton method can be used to solve~\eqref{RPG:subprobChen} efficiently~\cite{CMSZ2019}.  %However, the step size $\alpha_k$ in the update $x_{k+1} = R_{x_k}(\alpha_k \eta_{x_k}^*)$ cannot be always one, which is unlike Algorithm~\ref{RPG:a1}. A line search algorithm, therefore, has to be used.\kwcomm{this may not be appropriate since the algorithm finally becomes line search}
%Nevertheless, solving the Riemannian proximal mapping subproblem~\eqref{RPG:subproblem2} is more difficult  since the nonlinear function $R_{x_k}(\eta)$ is involved. 
In this section we present an algorithm for solving \eqref{RPG:subproblem2}, which is an iterative descent method starting from $0_x$. % to find any local minimizer of  \eqref{RPG:subproblem2}.
For notational convenience, we first restate~\eqref{RPG:subproblem2} as \kwcomm{not appropriate expression for local min} \whcomm{[Done]}{}
\begin{align} \label{RPG:subprobRestate}
\hbox{find a \whfirrev{}{stationary point} } \eta_{x}^* \hbox{ of } \ell_{x}(\eta)\hbox{ on }\T_x \mathcal{M} \hbox{ such that }  \ell_{x}(0) \geq \ell_{x}(\eta_{x}^*),
%\eta_{u}^* = \arg\min_{\eta \in \T_u \mathcal{M}} \ell_u(\eta) =& \inner[u]{\grad f(u)}{\eta} + \frac{\tilde{L}}{2} \|\eta\|_u^2 + g(R_{u}(\eta)),
\end{align}
where $\ell_x(\eta) = \inner[x]{\grad f(x)}{\eta} + \frac{\tilde{L}}{2} \|\eta\|_x^2 + g(R_{x}(\eta))$.  %, $u \in \mathcal{M}$ and $\eta_{u}^*$ is any \kwcomm{local} minimizer  such that $\ell_u(\eta_u^*) \leq\ell_u(0_u)$.
The following assumption will be used in the derivation of the algorithm.
\begin{assumption} \label{RPG:as15}
	(i) The manifold $\mathcal{M}$ is an embedded submanifold of $\mathbb{R}^n$ or is a quotient manifold whose total space is an embedded submanifold of $\mathbb{R}^n$. %\footnote{Here, $\mathbb{R}^n$ does not only refer to a vector space, but also can refer to a matrix space or a tensor space.}
	(ii) The function $g$ is Lipschitz continuous with constant $L_g$ and convex in the classical setting.
	(iii) The function $g$ is bounded from below.
\end{assumption}

Suppose $\eta_k$ is the current estimate of $\eta_x^*$. Our goal is to find a descent direction. Towards this end, we note that
\begin{align*}
\ell_x(\eta_k + \tilde{\xi}_k) =& \inner[x]{\grad f(x)}{\eta_k + \tilde{\xi}_k} + \frac{\tilde{L}}{2} \|\eta_k + \tilde{\xi}_k\|_x^2 + g(R_{x}(\eta_k + \tilde{\xi}_k)) \\
=& \inner[x]{\grad f(x)}{\eta_k} + \frac{\tilde{L}}{2} \|\eta_k\|_x^2 + \inner[x]{\grad f(x) + \tilde{L} \eta_k}{\tilde\xi_k} +\frac{\tilde{L}}{2} \|\tilde\xi_k \|_x^2 +  g(R_{x}(\eta_k + \tilde\xi_k))
\end{align*}
for any $\tilde\xi_k \in \T_x \mathcal{M}$.
Let $y_k=R_x(\eta_k)$ and $\xi_k=\mathcal{T}_{R_{\eta_k}} \tilde{\xi}_k$. Since $R$ is \whfirrev{}{smooth} by definition, there holds $R_x(\eta_k + \tilde\xi_k) = y_k + \xi_k + O(\|\xi_k\|_x^2)$, where $y = x + O(z)$ means \whfirrev{}{$\limsup_{z \rightarrow 0} \|y - x\| / \|z\| < \infty$.} It follows that
\begin{align*}
\ell_x(\eta_k + \tilde{\xi}_k) =& \inner[x]{\grad f(x)}{\eta_k} + \frac{\tilde{L}}{2} \|\eta_k\|_x^2 \nonumber \\
&+ \inner[x]{\grad f(x) + \tilde{L} \eta_k}{\mathcal{T}_{R_{\eta_k}}^{-1} \xi_k} + \frac{\tilde{L}}{2} \| \mathcal{T}_{R_{\eta_k}}^{-1} \xi_k \|_x^2 + g(y_k + \xi_k + O(\|\xi_k\|_x^2)) \nonumber \\
=& \inner[x]{\grad f(x)}{\eta_k} + \frac{\tilde{L}}{2} \|\eta_k\|_x^2 \nonumber \\
&+ \inner[x]{\grad f(x) + \tilde{L} \eta_k}{\mathcal{T}_{R_{\eta_k}}^{-1} \xi_k} + \frac{\tilde{L}}{2} \| \mathcal{T}_{R_{\eta_k}}^{-1} \xi_k \|_x^2 + g(y_k + \xi_k) + O(\|\xi_k\|_x^2). \nonumber \\
%=& \inner[u]{\grad f(u)}{\eta_k} + \frac{\tilde{L}}{2} \|\eta_k\|_u^2 \nonumber \\
%&+ \inner[u]{\grad f(u) + \tilde{L} \eta_k}{\mathcal{T}_{R_{\eta_k}}^{-1} \xi_k} + \frac{\tilde{L}}{2} \| \mathcal{T}_{R_{\eta_k}}^{-1} \xi_k - P_{v \rightarrow u} \xi_k  \|_u^2 + \frac{\tilde{L}}{2} \| P_{v \rightarrow u} \xi_k \|_u^2 + g(v_k + \xi_k) + O(\|\xi_k\|_u^2), \nonumber \\
=& \inner[x]{\grad f(x)}{\eta_k} + \frac{\tilde{L}}{2} \|\eta_k\|_x^2 \nonumber \\
&+ \inner[x]{\grad f(x) + \tilde{L} \eta_k}{\mathcal{T}_{R_{\eta_k}}^{-1} \xi_k} + \frac{\tilde{L}}{2} \| \xi_k \|_{\E}^2 + g(y_k + \xi_k) + O(\|\xi_k\|_x^2),
\end{align*}
%\begin{equation} \label{RPG:e49}
%\|P_{x \rightarrow R_x(\eta_x)} - \mathcal{T}_{R_{\eta_x}}\| \leq b \|\eta_x\|,
%\end{equation}
where the second equation is from the Lipschitz continuity of $g$ and the last equation is from the equivalence between any two norms in a finite dimensional space \whcomm{}{and that both $\| \mathcal{T}_{R_{\eta_k}}^{-1} \xi_k \|_x^2$ and $\| \xi_k \|_{\E}^2$ are second order terms}.\kwcomm{how? more details} \whcomm{[Done]}{}Letting $\tilde{\ell}_{y_k}(\xi_k)$ denote
\begin{equation}\label{eq:ke00}
\inner[x]{\grad f(x) + \tilde{L} \eta_k}{\mathcal{T}_{R_{\eta_k}}^{-1} \xi_k} + \frac{\tilde{L}}{2} \| \xi_k \|_{\E}^2 + g(y_k + \xi_k),
\end{equation}
we may interpret  it as a  simple local model   \kwcomm{seems not quite right} of $\ell_x(\eta_k + \mathcal{T}_{R_{\eta_k}}^{-1} \xi_k )$. Therefore, in order to find a new estimate from $\eta_k$, we can first compute a search direction by minimizing \eqref{eq:ke00} on $\T_{y_k}\mathcal{M}$, denoted $\xi_k^*$, and then update $\eta_k$ along the direction $\mathcal{T}_{R_{\eta_k}}^{-1} \xi_k^*$; see Algorithm~\ref{RPG:a3}.%It can be shown that under Assumption~\ref{RPG:as15} that minimizing $\tilde{\ell_k}(\xi_k)$ gives a descent direction for $\ell_u(\eta_k + \mathcal{T}_{R_{\eta_k}}^{-1} \xi_k )$. In the light of this, the algorithm is given in Algorithm~\ref{RPG:a3}\kwcomm{is this trivial?}.

\begin{algorithm}[h]
\caption{Solving the Riemannian Proximal Mapping}
\label{RPG:a3}
\begin{algorithmic}[1]
\Require Initial iterate $\eta_0 \in \T_x \mathcal{M}$; a small positive constant $\sigma$ for line search;
\For {$k = 0, \ldots$}
\State $y_k = R_{x}(\eta_k)$;
\State \label{RPG:a3:st1} Compute $\xi_k^*$ by solving
\begin{equation} \label{RPG:subdir}
\xi_k^* = \arg\min_{\xi \in \T_{y_k} \mathcal{M}} \inner[y_k]{\mathcal{T}_{R_{\eta_k}}^{-\sharp} (\grad f(x) + \tilde{L} \eta_k)}{\xi} + \frac{\tilde{L}}{2} \|\xi\|_{\E}^2 + g(y_k + \xi);
\end{equation}
\State $\alpha = 1$;
\While {$\ell_x(\eta_k + \alpha \mathcal{T}_{R_{\eta_k}}^{-1} \xi_k^*) \geq \ell_x(\eta_k) - \sigma \alpha \|\xi_k^*\|_x^2 $} \label{RPG:a3:st2}
\State $\alpha = \frac{1}{2} \alpha$;
\EndWhile
\State $\eta_{k+1} = \eta_k + \alpha \mathcal{T}_{R_{\eta_k}}^{-1} \xi_k^*$;
\EndFor
\end{algorithmic}
\end{algorithm}

Let $y=R_x(\eta)$ and suppose $R_x^{-1}(y)$ is well defined. Then~\eqref{RPG:subprobRestate} can be rewritten as
\begin{equation} \label{RPG:e44}
\arg\min_{y \in \mathcal{M}} \inner[x]{\grad f(x)}{R_x^{-1}(y)} + \frac{\tilde{L}}{2} \|R_x^{-1}(y)\|_x^2 + g(y).
\end{equation}
Interestingly, it is not hard to see that Algorithm~\ref{RPG:a3} can be interpreted as the application of the Riemannian proximal gradient method in~\cite{CMSZ2019} to the cost function in~\eqref{RPG:e44} under a proper choice of the retraction. \kwcomm{more details for non-experts?}\whcomm{[Done.]}{Specifically, the gradient of $\inner[x]{\grad f(x)}{R_x^{-1}(y)} + \frac{\tilde{L}}{2} \|R_x^{-1}(y)\|_x^2$ with respect to the Euclidean metric is $M_{x} \mathcal{T}_{R_{\eta_k}}^{-\sharp} (\grad f(x) + \tilde{L} \eta_k)$, where $M_x$ is the matrix expression of the Riemannian metric at $x$, i.e., $\inner[x]{\eta}{\xi} = \eta^T M_x \xi$. Thus, if we choose the retraction to be $R_y(\eta_y) = R_x(\xi_x + \mathcal{T}_{\xi_x}^{-1} \eta_y)$, where $\xi_x$ satisfies $R_x(\xi_x) = y$, one can easily see that Algorithm~\ref{RPG:a3} is indeed an application of~\cite[Algorithm~1]{CMSZ2019}.
}

Algorithm~\ref{RPG:a3} provides a general method for solving the Riemannian proximal mapping~\eqref{RPG:subdir} under Assumption~\ref{RPG:as15}. However, it is by no means the only method to do so. For example, another efficient algorithm can be developed when $\mathcal{M}$ is an oblique manifold (i.e., a Cartesian  product of unit spheres), see  Section~\ref{RPG:sect:Oblique} for more details.

\section{\whfirrev{}{Accelerating the Riemannian Proximal Gradient Method}} \label{sec:ARPG}

%\whfirrev{}{
%In the Euclidean setting, the convergence rate $O(1/k^2)$ is obtained under the assumption that the objective function is convex. 
%In Section~\ref{RPG:sect:RFISTA}, we discussed the difficulties of generalizing the FISTA method to the Riemannian setting. %Though no rigorous proof for the acceleration is given, 
%A practical Riemannian proximal gradient method is present with an acceleration in Section~\ref{RPG:sect:ARPG} and it is shown in the numerical section that acceleration techniques speed up the convergence. 
%}

%\subsection{A Riemannian FISTA Method} \label{RPG:sect:RFISTA} %$O(1/k^2)$ Convergence Analysis
\whfirrev{}{In this section, we attempt to develop an acceleration of Algorithm~\ref{RPG:a1} based on the idea of FISTA.}
The \whfirrev{}{vanilla} Riemannian generalization of the FISTA method~\eqref{RPG:EAPG} is presented in Algorithm~\ref{RPG:a2}, where
 the Riemannian proximal mapping and the update scheme are the same as those in Algorithm~\ref{RPG:a1}. Similarly to the FISTA method in a Euclidean space, an auxiliary sequence $\{y_k\}$ is generated, see~\eqref{RPG:e21} in Algorithm~\ref{RPG:a2}. \whcomm{}{In the Euclidean setting, the exponential mapping and its inverse are given by $R_x(\eta_x) = x + \eta_x$ and $R_x^{-1}(y) = y - x$. The definition of $y_{k+1}$ in~\eqref{RPG:e21} then becomes
$$
y_{k+1} = y_k + \frac{t_{k+1} + t_k - 1}{t_{k+1}} (x_{k+1} - y_k) - \frac{t_k-1}{t_{k+1}} (x_k - y_k) = x_{k+1} + \frac{t_k - 1}{t_{k+1}} (x_{k+1} - x_k),
$$
which coincides with the definition in~\eqref{RPG:EAPG}.
}
%y_{k+1} = x_{k+1} + \frac{t_k - 1}{t_{k+1}} (x_{k+1} - x_k).
%\kwcomm{point out the difference with the one in HW19?}\whcomm{[Done]}{}
%\kwcomm{point out that it reduces to FISTA in the Euclidean case} \whcomm{[Done]}{}
\whfirrev{}{As mentioned in the introduction, the accelerated $O(1/k^2)$ convergence rate can be established for FISTA in the Euclidean setting. %Though the Riemannian FISTA method presented in Algorithm~\ref{RPG:a2} is observed to have a faster empirical convergence rate,  
Establishing the theoretical $O(1/k^2)$ convergence rate for the V-APRG method is very challenging since we essentially deal with a nonconvex problem and a rigorous analysis of V-ARPG is beyond the scope of this paper. Indeed, a simple numerical experiment  shows that the V-APRG method may diverge when the RPG method converges. Here, we test an optimization problem for the sparse principal component analysis in~\cite{GHT2015}:
$$\min_{X \in \mathrm{OB}(p, n)} \|X^T A^T A X - D^2\|_{\E}^2 + \lambda \|X\|_1,$$ where $\mathrm{OB}(p, n)$ denotes the oblique manifold, i.e., the product manifold of $p$ number of $n-1$ dimensional spheres. The data matrix $A$ is generated randomly. See more details about the problem and experiment setup in Section~\ref{RPG:sect:NumExp}. Figure~\ref{RPG:RPGvsVARPG} reports the results of  two typical random instances. The left plot shows an instance when both RPG and V-APRG converge. In this case, it is evident that V-ARPG has the desirable acceleration behavior, as the Euclidean FISTA. On the other hand, RPG always converges in our tests while V-ARPG may not converge as shown in the right plot. %Due to the numerical observation, a thorough analysis of the V-APRG method is beyond 
\begin{figure}[ht!]
\centering
\includegraphics[width=0.8\textwidth]{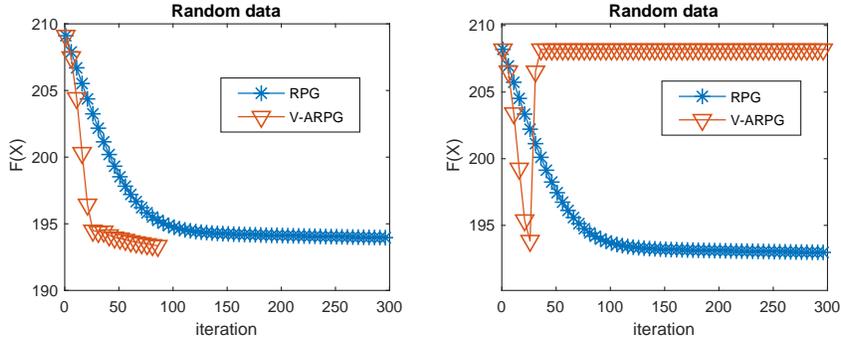}
\caption{
\whfirrev{}{
Comparisons of RPG and V-ARPG. The constant $\tilde{L} = 1.35\|A\|_F^2$. RPG terminates when the number of iterations reaches 1000. V-ARPG terminates when its function value is smaller than that obtained by RPG. Note that in the right plot, the iterates of V-ARPG actually diverge. The function value of V-ARPG  in the right plot remains constant as we plot the minimum of the function values with a prescribed upper bound. 
}
}
\label{RPG:RPGvsVARPG}
\end{figure}

%If, in addition to the retraction convexity and Assumption~\ref{RPG:as6}, more stringent assumptions on the sequence generated by Algorithm~\ref{RPG:a2} are imoposed, it is possible to give the $O(1/k^2)$ convergence rate analysis. However, since it is by no means a rigorous proof, we will present this as a discussion in  Appendix~\ref{RPG:appRFISTA}, but leave the rigorous analysis as future work.
}

\begin{algorithm}[h]
\caption{Vanilla Accelerated Riemannian Proximal Gradient Method (V-ARPG)}
\label{RPG:a2}
\begin{algorithmic}[1]
\Require A constant $\tilde{L} > L$; an initial iterate $x_0$; %$p \in [0, 1)$;, $\sigma \in (0, 1)$, $\nu \in (0, 1)$.
\State $t_0 = 1$, $y_0 = x_0$;
\For {$k = 0, \ldots$}
\State Find $\eta_{y_k}^*{\in \T_{y_k}\mathcal{M}}$ such that
\begin{align*}
\eta_{y_k}^* \hbox{ is a \whfirrev{}{stationary point} of } \ell_{y_k}(\eta) {\hbox{ on }\T_{y_k}\mathcal{M}}\hbox{ and }  \ell_{y_k}(0) \geq \ell_{y_k}(\eta_{y_k}^*);
\end{align*}
\State $x_{k+1} = R_{y_k}(\eta_{y_k}^*)$;
\State Let $t_{k+1} = \frac{1 + \sqrt{1 + 4 t_k^2}}{2}$; %%Find $t_{k + 1} > 0$ such that $t_{k + 1}^2 = t_k^2 + t_{k+1}$;
\State Compute $y_{k+1} \in \mathcal{M}$ by
\begin{equation} \label{RPG:e21}
y_{k+1} = R_{y_k}\left(\frac{t_{k+1} + t_k - 1}{t_{k+1}} \eta_{y_k}^* - \frac{t_k-1}{t_{k+1}} R_{y_k}^{-1}(x_k) \right);
\end{equation}
\EndFor
\end{algorithmic}
\end{algorithm}

%\subsection{A Practical Riemannian Proximal Gradient Method \kwcomm{title seems not appropriate, alternative like ``A Practical Riemannian Proximal Gradient Method''} \whcomm{[Done]}{}} \label{RPG:sect:ARPG}
\whfirrev{}{ The above observation motivates  us to develop a practical accelerated Riemannian proximal gradient method. To develop an empirically accelerated algorithm which enjoys the basic global convergence rate analysis, we adopt a restarting technique 
by combining Algorithms~\ref{RPG:a1} and~\ref{RPG:a2} together, which gives the practical accelerated Riemannian proximal gradient method, see Algorithm~\ref{RPG:alg:ARPG}.} Specifically, a safeguard is introduced in every $N$ iterations to check whether there is a sufficiently large decrease in the cost function, in contrast to the result given by one iteration of the proximal gradient method from the current reference point. If the function value decrement is sufficient the iteration continues, otherwise the algorithm will be restarted; see Step~\ref{RPG:ARPG:st1} to Step~\ref{RPG:ARPG:st2} of Algorithm~\ref{RPG:alg:ARPG}.%the function value decrement is sufficiently large, compared with the result
%obtained from one iteration of the proximal gradient method from the current reference point. When the function value decrement is not sufficient, the algorithm will be restarted, see Step~\ref{RPG:ARPG:st1} to Step~\ref{RPG:ARPG:st2} of Algorithm~\ref{RPG:alg:ARPG}.

%It has been shown that Algorithm~\ref{RPG:a1} converges to a stationary point under mild assumptions. If more assumptions are made, then it converges on the order of $O(1/k)$. Algorithm~\ref{RPG:a2} converges faster in the sense of order~$O(1/k^2)$ but requires even stronger assumptions. Algorithm~\ref{RPG:alg:ARPG} shares the positive properties of Algorithms~\ref{RPG:a1} and~\ref{RPG:a2} in the sense that it converges to stationary points under mild assumptions and on the order $O(1/k^2)$ if the other assumptions in Theorem~\ref{RPG:th2} hold.

%Specifically, Algorithm~\ref{RPG:alg:ARPG} uses a safeguard from Step~\ref{RPG:ARPG:st1} to Step~\ref{RPG:ARPG:st2}. This idea has been used in~\cite{HuaWei2019}. It can be shown that if the estimation of Lispchitz $\mathfrak{L}$ is greater than $L$, then any accumulation point of the sequence $\{z_{i N}\}, i = 0, 1, \ldots$ is a stationary point by using the same spirit of Theorem~\ref{RPG:globaltheo} and \cite[Theorem~3.1]{HuaWei2019}. Moreover, if the assumptions in~\ref{RPG:th2} hold and the safeguard only takes effect finite times, then the sequence $\{f(x_k)\}$ converge to $f(x_*)$ on the order $O(1/k^{2 - p})$ by Theorem~\ref{RPG:th2}.

\begin{algorithm}[ht!]
\caption{Practical Accelerated Riemannian Proximal Gradient Method (P-ARPG)}
\label{RPG:alg:ARPG}
\begin{algorithmic}[1]
\Require Initial iterate $x_0$; positive integers $N, N_{\min}$, and $N_{\max}$ for safeguard; an upper bound $\tilde{L}$ of the Lipschitz constant; a lower bound $\mathfrak{L}$ of the Lipschitz constant; Enlarging parameter $\tau \in (1, \infty)$ for updating $\mathfrak{L}$; line search parameter $\sigma \in (0, 1)$, shrinking parameter in line search $\nu \in (0, 1)$; maximum number of iterations in line search $N_{\mathrm{ls}} > 0$;
\State $t_0 = 1$, $y_{0} = x_0$, $z_0 = x_0$, $j_1 = 0$, and $j_2 = N$;
\For {$k = 0, \ldots$}
%\State Set $\tilde{\xi}_{k - 1} = \frac{t_{k - 1} - 1}{t_{k}} \xi_{k - 1}$ and $y_{k} = R_{x_k}\left( \tilde\xi_{k-1}\right)$
\If {$k == j_2$} \Comment{Invoke safeguard} \label{RPG:ARPG:st1}
\State  Invoke Algorithm~\ref{alg:Safeguard}:
$
[z_{j_2}, x_k, y_k, t_k, \mathfrak{L}, N] = Algo\ref{alg:Safeguard}(z_{j_1}, x_k, y_k, t_k, F(x_k), \mathfrak{L}, N);
$
\State Set $j_1 = j_2$ and $j_2 = j_2 + N$;
\EndIf \label{RPG:ARPG:st2}
\State Find $\eta_{y_k}^*{\in\T_{y_k}\mathcal{M}}$ such that
\begin{align*}
\eta_{y_k}^* \hbox{ is a \whfirrev{}{stationary point} of } \ell_{y_k}(\eta){\hbox{ on }\T_{y_k}\mathcal{M}} \hbox{ and }  \ell_{y_k}(0) \geq \ell_{y_k}(\eta_{y_k}^*);
\end{align*}
%\State Find any stationary point
%\begin{align*}
%\eta_{y_k}^* =& \arg\min_{\eta \in \T_{y_k} \mathcal{M}} \inner[y_k]{\grad f(y_k)}{\eta} + \frac{\mathfrak{L}}{2} \|\eta\|^2 + g(R_{y_k}(\eta)),  %\label{RPG:subproblem2}
%\end{align*}
%\hspace{0.6cm}such that $ \ell_{y_k}(\eta_{y_k}^*)\leq \ell_{y_k}(0)$, where $\ell$ is defined in Algorithm~\ref{RPG:a1};
%Compute
%\begin{align*}
%\eta_{y_k}^* =& \arg\min_{\eta \in \T_{y_k} \mathcal{M}} \inner[y_k]{\grad f(y_k)}{\eta} + \frac{\tilde{L}}{2} \|\eta\|^2 + g(R_{y_k}(\eta));
%\end{align*}

\State $x_{k+1} = R_{y_k}(\eta_{y_k}^*)$;
\State Let $t_{k+1} = \frac{1 + \sqrt{1 + 4 t_k^2}}{2}$; %%Find $t_{k + 1} > 0$ such that $t_{k + 1}^2 = t_k^2 + t_{k+1}$;
\State Compute $y_{k+1} \in \mathcal{M}$ by
%\begin{equation*}
%(t_{k+1} - 1) R_{y_{k + 1}}^{-1} (x_{k + 1}) + R_{y_{k + 1}}^{-1}(y_{k}) = \mathcal{T}_{y_{k} \rightarrow y_{k + 1}} \left( (t_k - 1) R_{y_k}^{-1}(x_k) - t_k \eta_{y_k} \right);
%\end{equation*}
%\begin{equation} \label{RPG:e21}
%(t_{k+1} - 1) R_{y_{k+1}}^{-1}(x_{k+1}) = \mathcal{T}_{y_{k} \rightarrow y_{k + 1}} \left( (t_k-1) R_{y_k}^{-1}(x_k) - t_k \eta_{y_k}^* + R_{y_k}^{-1}(y_{k+1}) \right)
%%t_{k+1} R_{y_{k + 1}}^{-1} (x_{k + 1}) = \mathcal{T}_{y_{k} \rightarrow y_{k + 1}} \left( t_k R_{y_k}^{-1}(x_k) - t_k \eta_{y_k} + R_{y_k}^{-1}(x_{k + 1}) - R_{y_k}^{-1}(x_k) \right);
%\end{equation}
\begin{equation*} %\label{RPG:e21}
y_{k+1} = R_{y_k}\left(\frac{t_{k+1} + t_k - 1}{t_{k+1}} \eta_{y_k}^* - \frac{t_k-1}{t_{k+1}} R_{y_k}^{-1}(x_k) \right);
\end{equation*}
\EndFor
\end{algorithmic}
\end{algorithm}

Since the constant $L$ for $f$ to be $L$-retraction-smooth is usually not known, an update strategy is also introduced to find an appropriate estimation of $L$.  The idea is to enlarge the estimation if the line search fails (Steps~\ref{alg:Safeguard:st6w} to~\ref{alg:Safeguard:st8w} in Algorithm~\ref{alg:Safeguard}) or the safeguard takes effect often (Steps~\ref{alg:Safeguard:st5} to~\ref{alg:Safeguard:st6} in Algorithm~\ref{alg:Safeguard}). In Algorithm~\ref{RPG:alg:ARPG}, solving the Riemannian proximal mapping dominates its computational cost, and invoking the safeguard requires at least one more Riemannian proximal mapping computation. Thus, in order to reduce the computational cost of the algorithm, an adaptive strategy is adopted to determine the frequency of invoking the safeguard. If the safeguard takes effect, then it will take effect more often. Otherwise, it will take effect less often, see Steps~\ref{alg:Safeguard:st7} and~\ref{alg:Safeguard:st8} in Algorithm~\ref{alg:Safeguard}.

%\whfirrev{}{As can be seen in the  numerical experiments, both Algorithms~\ref{RPG:a2} and \ref{RPG:a3} displays accelerated behavior over Algorithm~\ref{RPG:a1}. In practice, it is observed that the safeguard in Algorithm~\ref{RPG:a2} only takes effect finite times, and overall Algorithms~\ref{RPG:a2} and \ref{RPG:a3} have comparable convergence rates.}

\whfirrev{}{
Figure~\ref{RPG:RPGvsVARPGvsPARPG} compares P-ARPG with RPG and V-ARPG using the same problem as  Figure~\ref{RPG:RPGvsVARPG}. %Even though adding the safeguard slightly slows down the algorithm, we still prefer the P-ARPG since it has the basic global convergence.
It can be observed that the practical APRG method converges in both of the random instances and  at the same time it can still achieve significant acceleration over the RPG method. 
\begin{figure}[ht!]
\centering
\includegraphics[width=0.8\textwidth]{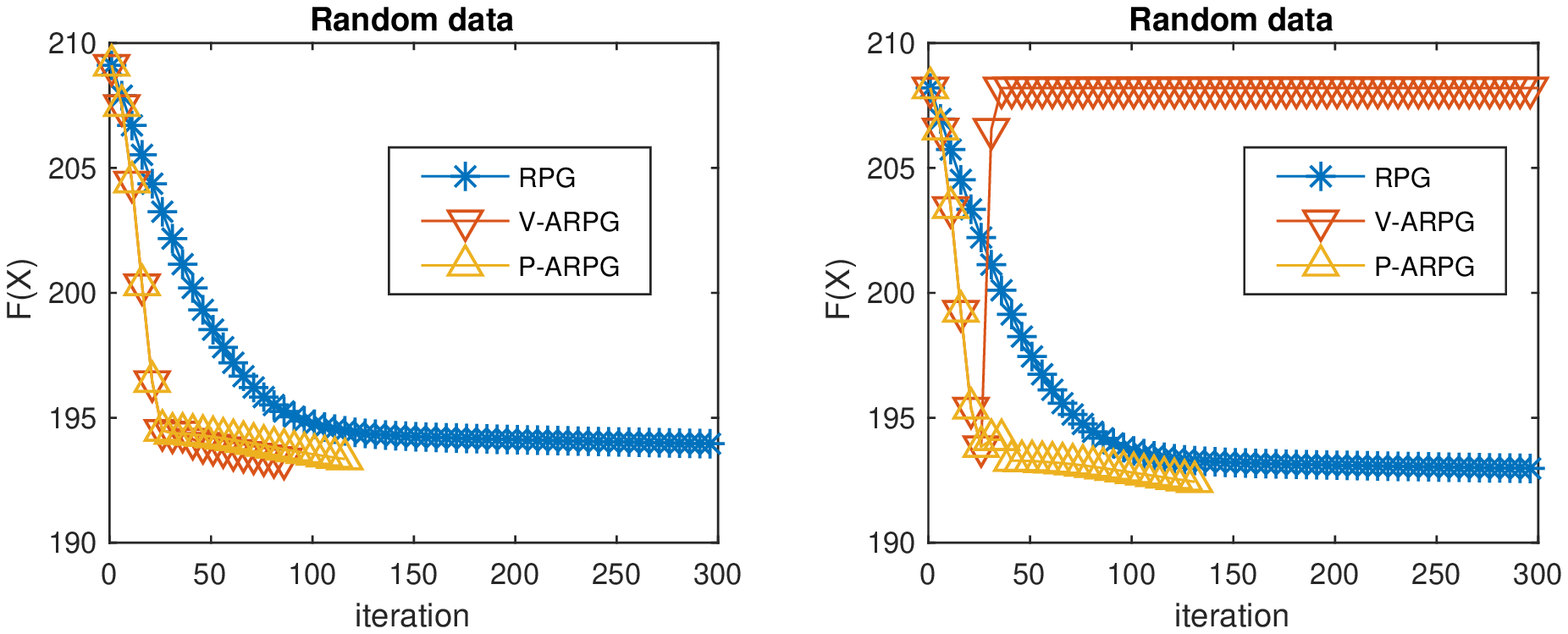}
\caption{
\whfirrev{}{
Comparisons of RPG and V-ARPG. The constant $\tilde{L} = \mathfrak{L} = 1.35\|A\|_F^2$. RPG terminates when the number of iterations reaches 1000. V-ARPG and P-ARPG terminate when their function values are smaller than that obtained by RPG. Left: the numbers of iterations of V-ARPG and P-ARPG are, respectively, 86 and 115. The number of restarts in P-ARPG is 1. The computations times of RPG, V-ARPG and P-ARPG are, respectively, 0.263, 0.042 and 0.056 second. Right: the number of iterations of P-ARPG is 133. The number of restarts in P-ARPG is 2. % time: 0.252, 0.507, 0.064
}
}
\label{RPG:RPGvsVARPGvsPARPG}
\end{figure}

}

\begin{algorithm}[ht!]
\caption{Safeguard for Algorithm~\ref{RPG:alg:ARPG}}
\label{alg:Safeguard}
\begin{algorithmic}[1]
\Require $(z_{j_1}, x_k, y_k, t_k, F(x_k), \mathfrak{L}, N)$;
\Ensure $[z_{j_2}, x_k, y_k, t_k, \mathfrak{L}, N]$;
%
%\State Find $\eta_{y_k}^*$ such that
%\begin{align*}
%\eta_{y_k}^* \hbox{ is a local minimizer of } \ell_{y_k}(\eta) \hbox{ and }  \ell_{y_k}(0) \geq \ell_{y_k}(\eta_{y_k}^*);
%\end{align*}
\State \label{alg:Safeguard:st1} Find $\eta_{z_{j_1}}^*$ such that
\begin{align*}
\eta_{z_{j_1}}^* \hbox{ is a stationary point of } \ell_{z_{j_1}}(\eta) \hbox{ and }  \ell_{z_{j_1}}(0) \geq \ell_{z_{j_1}}(\eta_{z_{j_1}}^*);  %\label{RPG:subproblem2}
\end{align*}
\State Set $\alpha = 1$ and $i_{\mathrm{ls}} = 0$;
\While {$F(R_{z_{j_1}}(\alpha\eta_{z_{j_1}})) > F(z_{j_1}) - \sigma \alpha \|\eta_{z_{j_1}}\|^2$ and $i_{\mathrm{ls}} < N_{\mathrm{ls}}$} \Comment{Line search}  \label{alg:Safeguard:st2}
\State $\alpha = \nu \alpha$, $i_{\mathrm{ls}} = i_{\mathrm{ls}} + 1$;\Comment{If \whcomm{}{$\mathfrak{L} > L + 2 \sigma$}, then no backtracking is performed by Lemma~\ref{RPG:le3}. \kwcomm{1) not quite true; 2) bound for $\sigma\alpha$ when terminating for stationary point analysis}}
\EndWhile
\If {$i_{\mathrm{ls}} == N_{\mathrm{ls}}$} \Comment{Line search fails}\label{alg:Safeguard:st6w}
\State $\mathfrak{L} = \tau \mathfrak{L}$ and goto Step~\ref{alg:Safeguard:st1}; \Comment{The estimation $\mathfrak{L}$ is too small; }
\EndIf\label{alg:Safeguard:st8w}
\If {$F(R_{z_k}(\alpha \eta_{z_k})) < F(x_k)$} \Comment{Safeguard takes effect} \label{alg:Safeguard:st3}
\If {$N \neq N_{\max}$} \label{alg:Safeguard:st5}
\State $\mathfrak{L} = \tau \mathfrak{L}$; \Comment{$\mathfrak{L}$ is not sufficiently large;\kwcomm{not clear to me; cannot shrink?} \whcomm{[right. As we discussed on wechat, I found this idea already works well.]}{}}
\EndIf \label{alg:Safeguard:st6}
\State $x_k = R_{z_k}(\eta_{z_k})$, $y_k = x_k$, and $t_k=1$;
\State $N = \max(N - 1, N_{\min})$; \label{alg:Safeguard:st7} \Comment{Check safeguard more often;}
\Else
\State $x_k$, $y_k$ and $t_k$ keep unchanged;
\State $N = \min(N + 1, N_{\max})$; \label{alg:Safeguard:st8} \Comment{Check safeguard less often;}
\EndIf \label{alg:Safeguard:st4}
\State $z_{j_2} = x_k$; \Comment{Update the compared iterate;}
\end{algorithmic}
\end{algorithm}

%%%%%%
%%%%%%
\section{Numerical Experiments}\label{RPG:sect:NumExp}

In this section we conduct numerical experiments on sparse principal component analysis (PCA) to demonstrate the performance of the proposed Riemannian proximal gradient methods. PCA is an important data processing technique which aims for linear combinations of  variables that can capture the maximal variance. In order to achieve the maximal variance, PCA tends to use a linear combination of all the variables which typically yields a dense solution. Alternatively, sparse PCA attempts to achieve a better trade-off between the data variance and  solution sparsity by incorporating the sparse structure into the mathematical models.

We consider two models for sparse PCA. The first one, aiming to find weakly correlated low dimensional representations~\cite{GHT2015}, considers the optimization problem on the oblique manifold
\begin{align} \label{RPG:spcaob}
&\min_{X \in \mathrm{OB}(p, n)} \|X^T A^T A X - D^2\|_{\E}^2 + \lambda \|X\|_1,
\end{align}
where  we recall that  $\mathrm{OB}(p, n) \whfirrev{}{= \{ X \in \mathbb{R}^{n \times p} \mid x_i^T x_i = 1, i = 1, \ldots, p, \hbox{ and $x_i$ is the $i$-th column of $X$} \}}$ denotes the oblique manifold,  $A \in \mathbb{R}^{m \times n}$ is the data matrix, $D$ is the diagonal matrix whose diagonal entries are the dominant singular values of $A$.
The second one is a penalized version of the ScoTLASS model introduced in~\cite{JoTrUd2003a} and it has been used in~\cite{CMSZ2019,HuaWei2019} to examine the performance of the proposed algorithms. The optimization problem is %analogous to \eqref{RPG:testassuprob}  but with the $\ell_1$-norm penalty term,
\begin{align} \label{RPG:spcast}
&\min_{X \in \mathrm{St}(p, n)} - \trace(X^T A^T A X) + \lambda \|X\|_1,
\end{align}
where $\St(p,n)$ denotes the Stiefel manifold, defined as
\begin{align*}
\St(p,n) = \{X\in\mathbb{R}^{n\times p}~|~X^TX=I_p\}.%\numberthis\label{eq:Stiefel}
\end{align*}
%and $\|X\|_1= \sum_{i,j}|X_{ij}|$ imposes sparsity on $X$.

%   and $\lambda>0$ is a tuning parameter.
%In this section, Problems~\eqref{RPG:spcaob} and~\eqref{RPG:spcast} are used to demonstrate the performance of the proposed Riemannain proximal gradient methods and the methods in~\cite{CMSZ2019,HuaWei2019}.
%%%%%
%%%%%

\subsection{\whfirrev{}{Convergence of RPG for Sparse PCA}} \label{sec:verify}

\whfirrev{}{
Here we verify that the above two objective functions satisfy the conditions for the global convergence (Theorem~\ref{RPG:globaltheo}), as well as the Riemannian KL property that guarantees the convergence to a single stationary point with a local convergence rate (Theorems~\ref{RPG:single} and \ref{RPG:LocalRateKL}).
Let $f_1 = \|X^T A^T A X - D^2\|_{\E}^2$, $f_2 = - \trace(X^T A^T A X)$, and $g = \lambda \|X\|_1$. Then the objective function~\eqref{RPG:spcaob} is $F_1(X) = f_1(X) + g(X)$ and the objective function~\eqref{RPG:spcast} is $F_2(X) = f_2(X) + g(X)$.

Since $F_1(X) \geq 0$ for all $X \in \mathrm{OB}(p, n)$, $F_2(X) \geq - p \sigma_{\max}^2$ for all $X \in \St(p, n)$ and the oblique manifold and the Stiefel manifold are compact, Assumption~\ref{RPG:as10} holds, where $\sigma_{\max}$ denotes the largest singular value of $A$. Assumption~\ref{RPG:as3} also holds by verifying the assumptions in~\cite[Lemma~2.7]{BAC2018}: (i) the Stiefel manifold and the oblique manifold are compact, (ii) the \whsecrev{}{retractions} %exponential mappings 
(see Sections~\ref{RPG:sect:Oblique} and~\ref{RPG:sect:Stiefel}) are globally defined, (iii) $F_1$ and $F_2$ are well-defined in $\mathbb{R}^{n \times p}$, and (iv) $F_1$ and $F_2$ are, respectively, $L$-smooth in the convex hull of the oblique manifold and the Stiefel manifold. Therefore, any accumulation point of the sequence generated by Algorithm~\ref{RPG:a1} is a stationary point by Theorem~\ref{RPG:globaltheo}.

It has been shown in~\cite[Proposition~7.4.5 and Corollary~7.4.6]{AMS2008} that any smooth function on a compact manifold is Lipschitz continuously differentiable. Therefore, the $f_1$ and $f_2$ are, respectively, Lipschitz continuously differentiable on the oblique manifold and the Stiefel manifold. The function $g$ is obviously continuous. Therefore, Assumption~\ref{RPG:as19} holds. Since $f_1$ and $f_2$ are polynomial functions, they are semialgebraic. The function $g$ is also semialgebraic by~\cite[Example~4]{BST2014}. Thus, $F_1$ and $F_2$ are semialgebraic. Since the oblique manifold can be viewed as a product manifold of $p$ number of $\St(1, n)$, it follows from Theorem~\ref{RPG:th5} that the function by restricting $F_1$ ($F_2$) to the oblique (Stiefel) manifold satisfies the Riemannian KL property at any point with the desingularising function in the form of $\varsigma(t) = \frac{C}{\theta} t^{\theta}$ for $\theta \in (0, 1]$ and $C > 0$. 
%\kwnote{Since Theorem~\ref{RPG:th5} only states for Stiefel manifold, we should note briefly about the oblique manifold.}{} 
%\whfirrev{A paragraph right after Theorem~\ref{RPG:th5} is added to discuss about product of Stiefel manifolds.}{}
Therefore, the convergence rate analysis given in Theorem~\ref{RPG:LocalRateKL} holds. Note that the local convergence rate is faster than $O(1/k)$ for any legitimate $\theta$.

%The assumptions in Section~\ref{RPG:sect:ConRatConv} are not verified. In particular, we found that empirically the function $g$ is not retraction convex using the exponential mappings in~\eqref{RPG:obexp} and~\eqref{RPG:stieexp}. It is pointed out here that this does not exclude the potential of retraction-convexity for the function $g$. For example, numerically the function $g:\mathbb{PB} \rightarrow \mathbb{R}: x \mapsto \|x\|_1$ is retraction convex, where $\mathbb{PB} = \{x \in \mathbb{R}^n \mid \|x\|_2 < 1\}$ is the Poincar\'e ball and the retraction is the exponential mapping, see detailed definitions in~\cite{Ganea2018}.
}

\subsection{Computations Related to Oblique Manifold} \label{RPG:sect:Oblique}
%Both the oblique manifold and the Stiefel manifold are the submanifolds of
%the Euclidean space $\mathbb{R}^{n\times p}$.
Let $\mathcal{M}$ be a submanifold of a Euclidean space and $f$ be a smooth function defined on $\mathcal{M}$. Then the Riemannian gradient of $f$ at $X$ is simply the projection of $\nabla f(X)$ onto the tangent space $\T_X\mathcal{M}$. Note that $\mathrm{OB}(p, n)$ is a submanifold of $\mathbb{R}^{n\times p}$ and the tangent space of $\mathrm{OB}(p, n)$ at a matrix $X\in \mathrm{OB}(p, n)$ is given by
\begin{align*}
\T_X\mathrm{OB}(p, n) = \{\eta_X~|~\diag(X^T\eta_X) = 0\},
\end{align*}
Thus, under
%The Riemannian metric of the oblique manifold is set to be
the Euclidean metric, i.e., $\inner[X]{\eta_X}{\xi_X} = \trace(\eta_X^T \xi_X)$, the Riemannian gradient of the smooth term $f$ in~\eqref{RPG:spcaob} is
 $$\grad f(X) = \nabla f(X) - X \diag(X^T \nabla f(X)),$$ where $\nabla f(X) = 4 A^T A X (X^T A^T A X - D^2)$ is the Euclidean gradient of $f$.

In this section we choose the exponential  mapping as the retraction. Since $\mathrm{OB}(p, n)$ is a product manifold of unit spheres, the exponential mapping from $\T_X\mathrm{OB}(p, n)$ to $\mathrm{OB}(p, n)$ is given by applying the exponential mapping on the unit sphere $\mathbb{S}^{n - 1}$, \whfirrev{}{see e.g.,~\cite{AMS2008}}
%~\cite{AMS2008}:
\begin{align} \label{eq:ke001}
\Exp_x(\eta_x) = x \cos(\|\eta_x\|_2) + \eta_x \sin(\|\eta_x\|_2) / \|\eta_x\|_2, \quad x \in \mathbb{S}^{n - 1}, \quad \eta_x \in \T_x \mathbb{S}^{n - 1},
\end{align}
to each column of a tangent vector separately. That is, with a slight abuse of notation, we have
\begin{align} \label{RPG:obexp}
\Exp_X(\eta_X) =[\Exp_{X_1}(\eta_X)_1,\cdots,\Exp_{X_p}(\eta_X)_p],
\end{align}
where $(M)_i$ denotes the $i$-th column of  $M$. Likewise, the inverse exponential mapping can also be computed by applying the inverse exponential mapping on the unit sphere $\mathbb{S}^{n - 1}$, \whfirrev{}{see e.g.,~\cite{SriKla2016}}
\begin{align}
\mathrm{Log}_x(y) = \Exp_x^{-1}(y) = \frac{\cos^{-1}(x^T y)}{\sqrt{1 - (x^T y)^2}} (I - x x^T) y, \quad x, y\in \mathbb{S}^{n - 1} \label{eq:inv_exp_ob}
\end{align}
in a column-wise manner, i.e.,
\begin{align*}
\Exp^{-1}_X(Y) = [\mathrm{Log}_{X_1}Y_1,\cdots, \mathrm{Log}_{X_p}Y_p].
\end{align*}

When using the Riemannian proximal gradient method to solve \eqref{RPG:spcaob}, the Riemannian proximal mapping has the form
\begin{equation} \label{RPG:e60}
	\min_{\eta_X \in \T_X\mathrm{OB}(p, n)} \frac{\tilde{L}}{2} \left\|\eta_X + \frac{1}{\tilde{L}} \grad f(X)\right\|_{\E}^2 + \lambda \|\Exp_X(\eta_X)\|_1.
\end{equation}
Due to the separability of $\Exp_X(\eta_X)$, one can easily see that the solution to \eqref{RPG:e60} can be computed with respect to each column of $\eta_X$ separately.
Therefore, without loss of generality, we consider~\eqref{RPG:e60} with $p = 1$. After making the following substitutions $\tilde{\lambda} = \lambda / \tilde{L}$, $y = \Exp_x(\eta)$, and $\xi_x = \frac{1}{\tilde{L}} \grad f(x)$, ~\eqref{RPG:e60} can be rewritten as
\begin{equation} \label{RPG:subprobSphere}
\min_{y\in\mathbb{S}^{n-1}} u(y),\hbox{ where } u(y)= \underbrace{\frac{1}{2 \tilde\lambda} \left\|\mathrm{Log}_x(y) + \xi_x \right\|_2^2}_{h(y)} + \|y\|_1.
\end{equation}

We will present a conditional gradient method to compute the solution of \eqref{RPG:subprobSphere}. %Let $h(y)= \frac{1}{2 \tilde\lambda} \left\|\mathrm{Log}_x(y) + \xi_x \right\|_2^2 $.
Letting $y_k$ be the current estimate of the minimizer of $u(y)$ over the unit sphere, a  new estimate $y_{k+1}$ is then computed by solving the following optimization problem
\begin{align}\label{eq:ke002}
\min_{y\in\mathbb{S}^{n-1}} h(y_k)+\nabla h(y_k)^T(y-y_k)+\|y\|_1\Leftrightarrow \min_{y\in\mathbb{S}^{n-1}} \nabla h(y_k)^Ty+\|y\|_1.
\end{align}
In other words,  we approximate $h(y)$ by its first order Taylor expansion around $y_k$ in each iteration.

It  remains to see how to solve \eqref{eq:ke002}. Actually, it has a closed-form solution. To see this, note that \eqref{eq:ke002} is  further equivalent to
\begin{align}\label{eq:ke003}
\min_{y\in\mathbb{S}^{n-1}}  \frac{1}{2}\|y+\nabla h(y_k)\|_2^2+\|y\|_1
\end{align}
since $\|y\|_2=1$ for all $y\in\mathbb{S}^{n-1}$.
By  Lemma~\ref{RPG:le14} in the appendix we know that the solution to  \eqref{eq:ke003} is given by
 \begin{equation} \label{RPG:e61}
{y}_* =
\left\{
\begin{array}{ll}
\frac{z}{\|z\|_{2}},  & \hbox{ if $\|z\|_{2} \neq 0$; } \\
\sign(\tilde{x}_{i_{\max}}) e_{i_{\max}} & \hbox{ otherwise, }
\end{array}
\right.
\end{equation}
where $i_{\max}$ is the index of the largest magnitude entry of $\nabla h(y_k)$, $e_i$ denotes the $i$-th column in the canonical basis of $\mathbb{R}^n$, and $z$ is defined by
$$
z_i =
\left\{
\begin{array}{ll}
0 & \hbox{ if $|(\nabla h(y_k))_i| \leq 1$; } \\
-(\nabla h(y_k))_i - 1 & \hbox{ if $-(\nabla h(y_k))_i > 1$; } \\
-(\nabla h(y_k))_i + 1 & \hbox{ if $-(\nabla h(y_k))_i < - 1$.} \\
\end{array}
\right.
$$
Note that the gradient of $h(y)$ is
\begin{equation*} %\label{RPG:e36}
\nabla h(y) = \frac{1}{\tilde\lambda} \underbrace{\left( - \frac{\cos^{-1}(x^T y)}{\sqrt{1 - (x^T y)^2}} - \frac{\xi_x^T y}{1 - (x^T y)^2} + \frac{\cos^{-1}(x^T y) \xi_x^T y x^T y}{(1 - (x^T y)^2)^{\frac{3}{2}}} \right)}_{s(y)} x + \frac{1}{\tilde\lambda} \underbrace{\frac{\cos^{-1}(x^T y)}{\sqrt{1 - (x^T y)^2}}}_{t(y)} \xi_x.
\end{equation*}
Putting it all together, we obtain the algorithm for solving~\eqref{RPG:subprobSphere}, see Algorithm~\ref{alg:RPMOB}. Suppose the sequence $\{y_k\}$ generated by Algorithm~\ref{alg:RPMOB} converges to a point $y_*$. Then by the first order optimality condition of \eqref{eq:ke003}, it is easy to see that there exists a constant $c$ such that $cy_*\in\partial u(y_*)$, where $\partial u$ denotes the subdifferential of $u$. Hence, $y_*$ is a critical point of~\eqref{RPG:subprobSphere}. In our experiments, two iterations are usually sufficient for the algorithm to achieve high accuracy.

\begin{algorithm}[ht!]
\caption{Solving the Riemannian Proximal Mapping for Oblique Manifold}
\label{alg:RPMOB}
\begin{algorithmic}[1]
\Require initial iterate $y_0$; $k = 0$;
\For {$k = 0, \ldots$}
\State Compute $s(y_k)$ and $t(y_k)$;
\State Compute $y_{k+1}$ via~\eqref{eq:ke003} with $\nabla h(y_k) = \left[s(y_k) x + t(y_k) \xi_x\right]/\tilde{\lambda}$;
\EndFor
\end{algorithmic}
\end{algorithm}

%Therefore, the Riemannian proximal mapping~\eqref{RPG:subproblem2} on the oblique manifold can be solved by Algorithm~\ref{alg:RPMOB} with~\eqref{RPG:probsmall} replacing by~\eqref{RPG:e45}.
%%%%%
%%%%%
\subsection{Computations Related to Stiefel Manifold} \label{RPG:sect:Stiefel}
The Stiefel manifold $\St(p,n)$ is also a submanifold of $\mathbb{R}^{n\times p}$, and the tangent space of $\St(p,n)$ at a matrix $X\in\St(p,n)$ is given by
\begin{align*}
\T_X\St(p,n) =\{\eta\in\mathbb{R}^{n\times p}~|~X^T\eta+\eta^TX=0\}.%\numberthis\label{eq:tangent}
\end{align*}
Here we use the \whsecrev{}{Euclidean} metric
\begin{equation} \label{RPG:stiemetric}
\inner[X]{\eta_X}{\xi_X} = \trace\left(\eta_X^T \xi_X \right)
\end{equation}
as the Riemannian metric. The Riemannian gradient of the smooth term $f$ in~\eqref{RPG:spcast}  under the canonical metric is
\begin{equation*}
\grad f(X) = \nabla f(X) - \frac{1}{2} X (X^T \nabla f(X) + \nabla f(X)^T X),
\end{equation*}
where $\nabla f(X) = -2 A^T A X$ is the Euclidean gradient of $f$. \whsecrev{}{The retraction by polar decomposition
\begin{equation} \label{RPG:stieretr}
R_X(\eta_X) = (X + \eta_X) (I_p + \eta_X^T \eta_X)^{-1/2}
\end{equation}
is used. 
The vector transport by differentiated the retraction~\eqref{RPG:stieretr} is given in~\cite[Lemma~10.2.1]{HUANG2013} by
\begin{equation} \label{RPG:e108}
\mathcal{T}_{\eta_X} \xi_X = Y \Omega + (I_n - Y Y^T) \xi_X (Y^T (X + \eta_X))^{-1},
\end{equation}
where $Y = R_X(\eta_X)$ and $\Omega$ is the solution of the Sylvester equation $(Y^T(X + \eta_X)) \Omega + \Omega (Y^T(X + \eta_X)) = Y^T \xi_X - \xi_X^T Y$. The inverse vector transport by differentiated retraction and the adjoint of the inverse vector transport by differentiated retraction are derived in the Appendix~\ref{app1}. Specifically, we have
\[
\mathcal{T}_{{\eta_X}}^{-1} \zeta_Y = YA + P,
\]
where $Y = R_X(\eta_X)$, $P = (I_n - Y Y^T) \zeta_Y (Y^T(X + \eta_X))$ and $A$ is the solution of the Sylvester equation $X^T Y A + A Y^T X = [(Y^T \zeta_Y) (Y^T (X + \eta_X)) + (Y^T (X + \eta_X)) (Y^T \zeta_Y)] Y^T X - X^T P - P^T X$; and
\[
\mathcal{T}_{\eta_X}^{- \sharp} \xi_X = Y [B (X^T Y) (Y^T(X + \eta_X)) + (Y^T (X + \eta_X)) B (X^T Y) ] - (I_n - Y Y^T) (X B + X B^T - \xi_X) (Y^T (X + \eta_X)),
\]
where $B$ is the solution of the Sylvester equation $ Y^T X B + B X^T Y = Y^T \xi_X$.

}

% In addition, the exponential mapping with respect to the canonical metric is~\cite{EAS98}
%\begin{equation} \label{RPG:stieexp}
%\Exp_X(\eta_X) =
%\begin{bmatrix}
%X & Q	
%\end{bmatrix}
%\exp\left(
%\begin{bmatrix}
%\Omega & -R^T \\
%R & 0	
%\end{bmatrix}
%\right)
%\begin{bmatrix}
%I_p \\
%0	
%\end{bmatrix},
%\end{equation}
%where $\Omega = X^T \eta_X$, $Q$ and $R$ are from the compact QR factorization of $(I - X X^T) \eta_X$. The inverse of the exponential mapping can be computed by the algorithm proposed in~\cite{Zim2017}.
%
%
%
%
In the case of the Stiefel manifold, Algorithm~\ref{RPG:a3}  in Section~\ref{RPG:sect:subproblem} will be  used to solve the Riemannian proximal mapping~\eqref{RPG:subproblem2}. Note that the subproblem~\eqref{RPG:subdir} can be solved by semismooth Newton method, which has been discussed in~\cite{CMSZ2019} in details.
%To apply Algorithm~\ref{RPG:a3} it still requires computing  the inverse vector transport by differentiated retraction $\mathcal{T}_R^{-1}$ and the adjoint operator of the inverse vector transport by differentiated retraction $\mathcal{T}_R^{-\sharp}$.  The computational details are presented in Appendix~\ref{app1}. \kwcomm{note: seems there are repeated statements of lemmas, so I removed ones here.}

%In our implementations the linear systems~\eqref{RPG:e46} and~\eqref{RPG:e47} are solved by the Matlab built-in  function \verb+gmres+. %\kwcomm{maybe mention the size of the linear systems}
%%%%%%
\subsection{Experimental Setup} \label{RPG:sect:ExpSetup}
We will compare  RPG (Algorithm~\ref{RPG:a1}) and ARPG  (\whfirrev{}{Algorithms~\ref{RPG:a2} and \ref{RPG:alg:ARPG}})  with
the Riemannian proximal gradient methods  from~\cite{CMSZ2019}   and~\cite{HuaWei2019}. As stated previously, the Riemannian proximal gradient method introduced in \cite{CMSZ2019} (denoted ManPG) is based on a different Riemannian proximal mapping, namely the one in \eqref{RPG:subprobChen}.
Furthermore, a more practical variant called ManPG-Ada is also presented in \cite{CMSZ2019}, which can achieve faster empirical convergence  by adaptively adjusting the weight of the quadratic term in the cost function of the Riemannian proximal mapping. In contrast, similar to Algorithm~\ref{RPG:alg:ARPG}, the method proposed in~\cite{HuaWei2019} (denoted AManPG) attempts to accelerate ManPG using the Nesterov momentum technique. In our experiments, unless otherwise stated, RPG and ManPG terminate when the search direction $\eta_{x_k}^*$ satisfies $\|\eta_{x_k}^* \tilde{L}\|^2 < 10^{-8} n p$. The other algorithms terminate when their function values are smaller than the minimum of the function values obtained from RPG and ManPG. All experiments are performed in Matlab R2018b on a 64 bit Ubuntu platform with 3.5GHz CPU (Intel Core i7-7800X).%We remind the reader that Algorithms~\ref{RPG:a1} and~\ref{RPG:alg:ARPG} in this paper are abbreviated as RPG and ARPG, respectively.
\kwcomm{note: RPG-Ada maybe not needed so related statements are deleted} \whcomm{[Done]}{}
%as well as a more practical variant (called ManPG-Ada) from the same paper. In a nutshell, ManPG-Ada updates the weight of the second order term in the cost function of the Riemannian proximal mapping adaptively so that faster empirical convergence rate can be achieved. This idea also can be applied to Algorithm~\ref{RPG:a1}. We denote Algorithm~\ref{RPG:a1} and its corresponding variant by RPG and RPG-Ada, respectively. The method proposed in~\cite{HuaWei2019} merges the Riemannian proximal gradient method in~\cite{CMSZ2019} and the acceleration approach of the Nesterov momentum technique, and has empirically shown significantly improvement. We denote the algorithm by AManPG. Lastly, Algorithm~\ref{RPG:alg:ARPG} is denoted by ARPG.

The parameters in ManPG, ManPG-Ada and AManPG are set to their default values, as in the corresponding papers. It is worth noting that, since each column of a matrix on the oblique manifold is a point on the unit sphere, the Riemannian proximal mapping~\eqref{RPG:subprobChen}  for the optimization problem~\eqref{RPG:spcaob} can be solved by the semismooth Newton method column by column.

%\kwnote{It seems this part do not need revision even vanilla ARPG will be tested}{} \whfirrev{agree}{}
The parameters in RPG and ARPG are chosen as follows. For the problem on the oblique manifold, the constants $\tilde{L}$ and $\mathfrak{L}$ are set to be $4 \|A\|_F^2$ and $\frac{1}{2} \|A\|_F^2$, respectively. The parameters $N$, $N_{\min}$ and $N_{\max}$ for the safeguard algorithm are set to be $5$, $2$, and $10$, respectively. The enlarging parameter $\tau$, line search parameter $\sigma$, shrinking parameter $\nu$ for step size, and the maximum number of iterations $N_{\mathrm{ls}}$ in the line search are set to be $1.1$, $0.0001$, $0.5$, and $3$, respectively. Algorithm~\ref{alg:RPMOB} terminates when the maximum value of $|x^T y_k - x^T y_{k+1}|$ and $|\xi_x^T y_k - \xi_x^T y_{k+1}|$ is smaller than $10^{-10}$.  % and Algorithms~\ref{RPG:alg:ARPG} stop when $\|\eta_{y_k}^* \tilde{L}\|^2 < 10^{-8} n p$.
For the problem on the Stiefel manifold, the constants $\tilde{L}$ and $\mathfrak{L}$ are set to be $2 \|A\|_2$ and $1.6 \|A\|_F^2$, respectively. The parameters $N$, $N_{\min}$ and $N_{\max}$ for the safeguard algorithm are set to be $5$, $3$, and $5$, respectively.
Algorithm~\ref{RPG:a3} terminates whenever one of the following three conditions is reached:  $\|\xi_k^*\| < {\whsecrev{}{ 0.003 }}$,  $\|\alpha \mathcal{T}_{R_{\eta_k}}^{-1} \xi_k^* \| < {\whsecrev{}{ 0.003 }}$, or the number of iterations exceeds 50. %, where \whsecrev{}{$\tau_{\mathrm{sub}} = 1/16$ unless otherwise indicated}. 
The remaining settings are the same as those for the problem on the oblique manifold.

%\whcomm{[Currently, all the experiments are run on my laptop. Rerun every thing on my desktop later TODO:]
%}{

%}

\whfirrev{}{
Two different types of data matrices $A$ are tested:
\begin{enumerate}
\item \textbf{Random data:} Generate $A$ such that its entries are drawn from the standard normal distribution $\mathcal{N}(0, 1)$. Then the matrix $A$ is shifted and normalized such that their columns have mean zero and standard deviation one.
\item \textbf{Synthetic data:} Five principal components shown in Figure~\ref{figPCs} are used. We repeat each of them $m/5$ times to obtain an $m$-by-$n$ noise-free matrix. The matrix $A$ is computed by further adding a random noise matrix, where each entry of the noise matrix is drawn from~$\mathcal{N}(0, 0.25)$. Finally the matrix $A$ is shifted and normalized such that their columns have mean zero and standard deviation one. Such idea has been used in~\cite{SCLEE2018} for constructing data that are close to real data.
\end{enumerate}
\begin{figure}[ht!]
\centering
\includegraphics[width=0.8\textwidth]{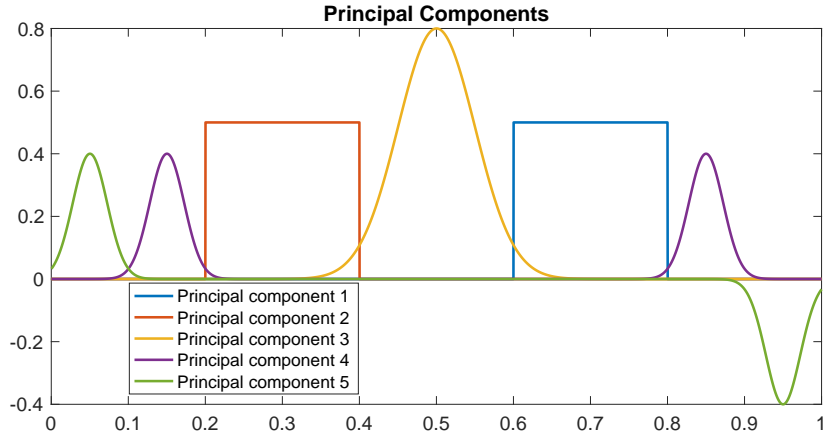}
\caption{
The five principal components used in the synthetic data.
}
\label{figPCs}
\end{figure}
The initial iterate is the leading $r$ right singular vectors of the matrix~$A$.
}

%\subsection{Comparing RPG, V-ARPG and P-ARPG}\label{num:V-ARPG}

\subsection{Comparing RPG and P-ARPG with ManPG(-Ada) and AManPG }
\whfirrev{}{This sections compares the algorithms developed in this paper with those in \cite{CMSZ2019}. and~\cite{HuaWei2019}. Even though \whsecrev{}{V-ARPG} is empirically slightly faster than \whsecrev{}{P-ARPG} as can be seen from Section~\ref{sec:ARPG}, we choose to compare other algorithms with \whsecrev{}{P-ARPG} here since it has the basic global convergence. }
\whfirrev{}{Figures~\ref{RPG:SPCAOB} and~\ref{RPG:SPCAOB_artificial} show} the performance of the aforementioned algorithms under multiple values of $n$, $p$ and $\lambda$ for the sparse PCA model on the oblique manifold \whfirrev{}{with random data and synthetic data}. As illustrated from the left and middle plots \whfirrev{}{of the two figures},  AManPG and \whfirrev{}{P-ARPG} take fewer number of iterations and less computational time to converge than the other algorithms. Moreover, since it is more efficient to solve~\eqref{RPG:subproblem2} than to solve~\eqref{RPG:subprobChen}, \whfirrev{}{P-ARPG} is slightly faster than AManPG.
Note that the solutions found by all the test algorithms have \whfirrev{}{similar percentage of non-zero entries}, as shown in the right plots of \whfirrev{}{Figures~\ref{RPG:SPCAOB} and~\ref{RPG:SPCAOB_artificial}}, \whfirrev{}{where the sparsity level is the portion of entries that are less than $10^{-5}$ in magnitude}. The figure also suggests that compared to ManPG and RPG the adaptive scheme used in ManPG-Ada is able to reduce the number of iterations upon convergence. However, the Nestrerov momentum acceleration technique used in AManPG and \whfirrev{}{P-ARPG} can further reduce the number of iterations without noticeably increasing the per iteration cost when solving the sparse PCA problem \eqref{RPG:spcaob}.

In addition, \whfirrev{}{Figures~\ref{RPG:SPCAOB_f} and~\ref{RPG:SPCAOB_f_artificial}, respectively,} display two {\em function values versus iterations} plots from two typical random instances, \whfirrev{}{of random data and synthetic data}.  Together with the middle plots in \whfirrev{}{Figures~\ref{RPG:SPCAOB} and~\ref{RPG:SPCAOB_artificial}}, it suggests that in the case of the oblique manifold the Riemannian proximal mappings~\eqref{RPG:subproblem2} and~\eqref{RPG:subprobChen} (the one used in~\cite{CMSZ2019}) perform similarly in the sense that it takes ManPG and RPG (respectively, AManPG and APRG)  approximately the same number of iterations to converge.
%the Riemannian proximal mappings~\eqref{RPG:subproblem2} and~\eqref{RPG:subprobChen} (the one used in~\cite{CMSZ2019}) perform similarly in the sense that it takes them approximately the same number of iterations to converge. However, solving~\eqref{RPG:subproblem2} takes less computational time than solving~\eqref{RPG:subprobChen}. Therefore, the Riemannian proximal gradient methods using~\eqref{RPG:subproblem2} is more efficient than those using~\eqref{RPG:subprobChen}, as shown from the right plots of Figure~\ref{RPG:SPCAOB}. In addition, the solutions found by the algorithms have similar sparsity, as shown from the right plots of Figure~\ref{RPG:SPCAOB}.

%It is pointed out that for the optimization problem~\eqref{RPG:spcaob}, the Riemannian proximal gradient methods without acceleration but with an adaptive coefficient, including ManPG-Ada and RPG-Ada, significantly reduce the number of iterations when compared to the ones with fixed coefficient, ManPG and RPG. Additionally, the Nestrerov momentum acceleration technique further reduces the number of iterations without noticeably increasing cost per iteration.

\begin{figure}
\scalebox{1.0}{
\hspace{-0.12\textwidth}\includegraphics[width=1.2\textwidth]{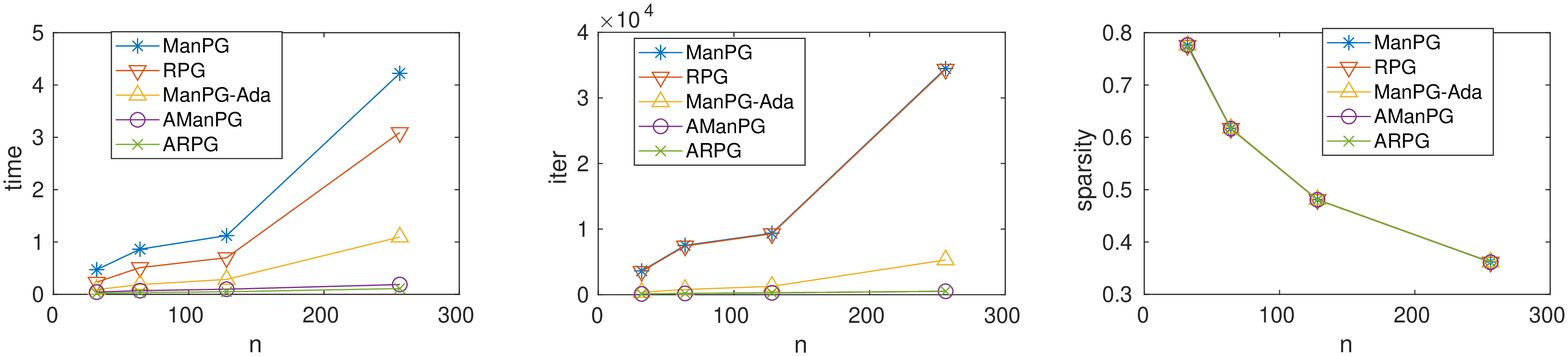} \\
}
\scalebox{1.0}{
\hspace{-0.12\textwidth}\includegraphics[width=1.2\textwidth]{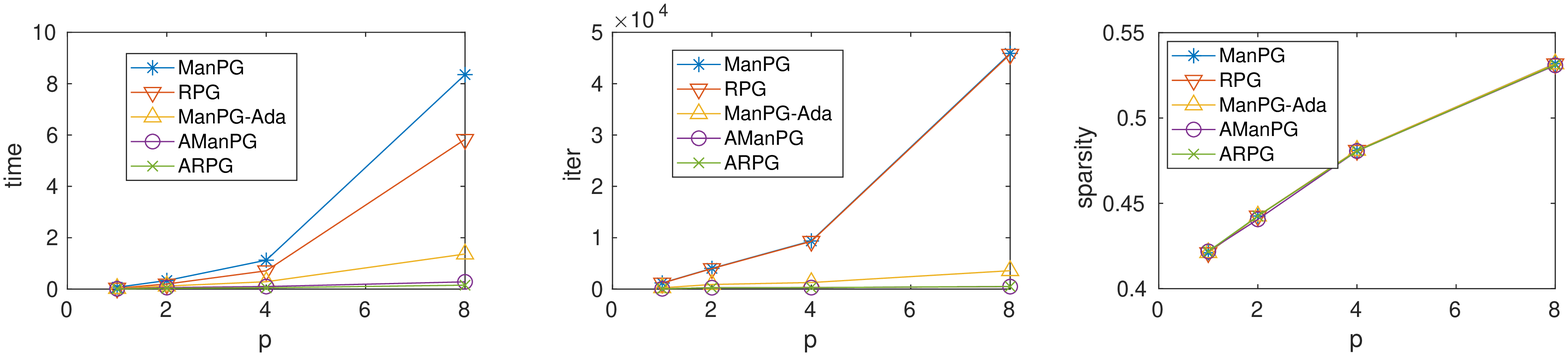} \\
}
\scalebox{1.0}{
\hspace{-0.12\textwidth}\includegraphics[width=1.2\textwidth]{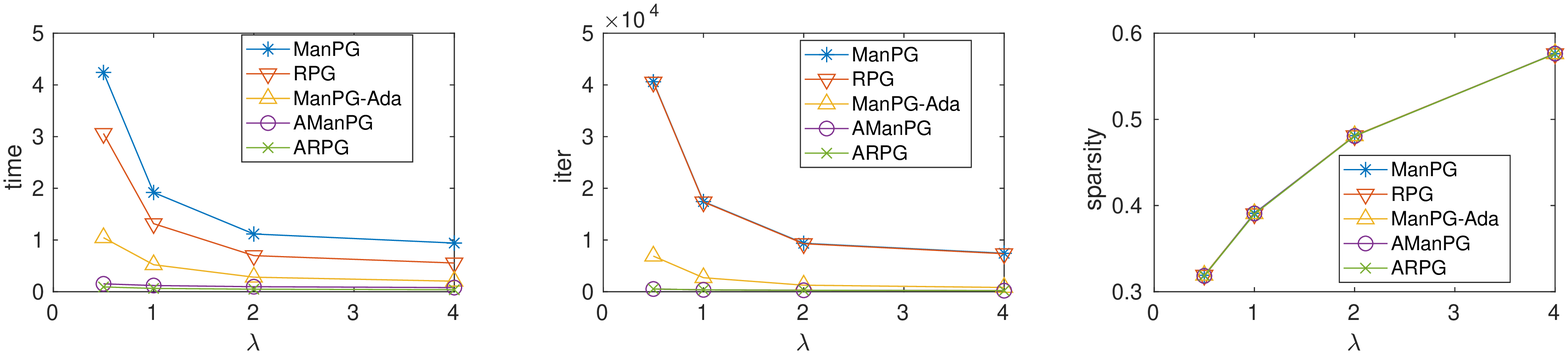}
}
\caption{
\whfirrev{}{Random data:} Average results of 10 random runs for the sparse PCA on the oblique manifold~\eqref{RPG:spcaob}. Top: multiple values $n = \{32, 64, 128, 256\}$ with $p = 4$, $m = 20$, and $\lambda = 2$; Middle: multiple values $p = \{1, 2, 4, 8\}$ with $n = 128$, $m = 20$, and $\lambda = 2$; Bottom: Multiple values $\lambda = \{0.5, 1, 2, 4\}$ with $n = 128$, $p = 4$, and $m = 20$.
}
\label{RPG:SPCAOB}
\end{figure}

\begin{figure}
\centering
\includegraphics[width=0.8\textwidth]{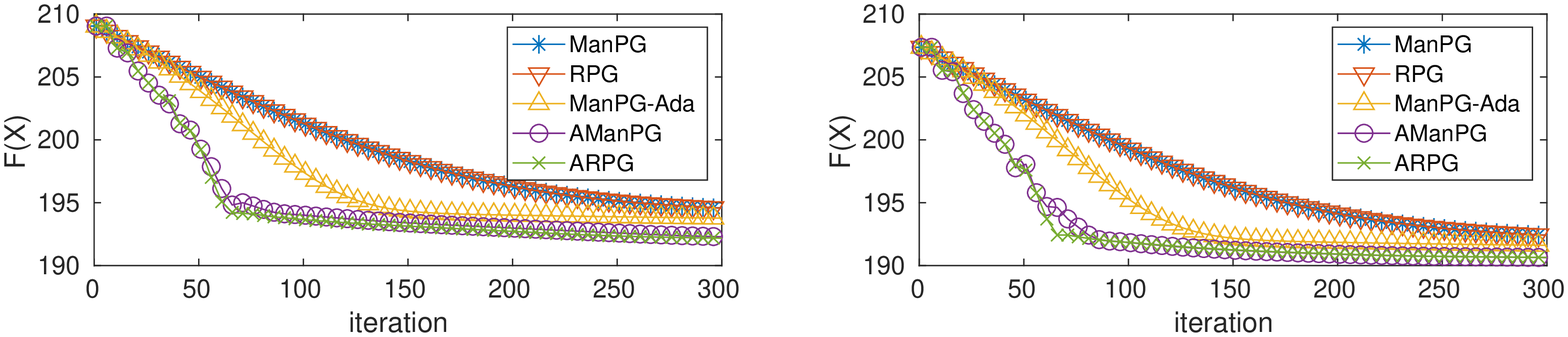}
\caption{
\whfirrev{}{Random data:} Comparison of the tested methods using two typical instances for the sparse PCA on the oblique manifold~\eqref{RPG:spcaob}. $n = 1024$, $p = 4$, $\lambda = 2$, $m = 20$.
}
\label{RPG:SPCAOB_f}
\end{figure}

\begin{figure}
\scalebox{1.0}{
\hspace{-0.12\textwidth}\includegraphics[width=1.2\textwidth]{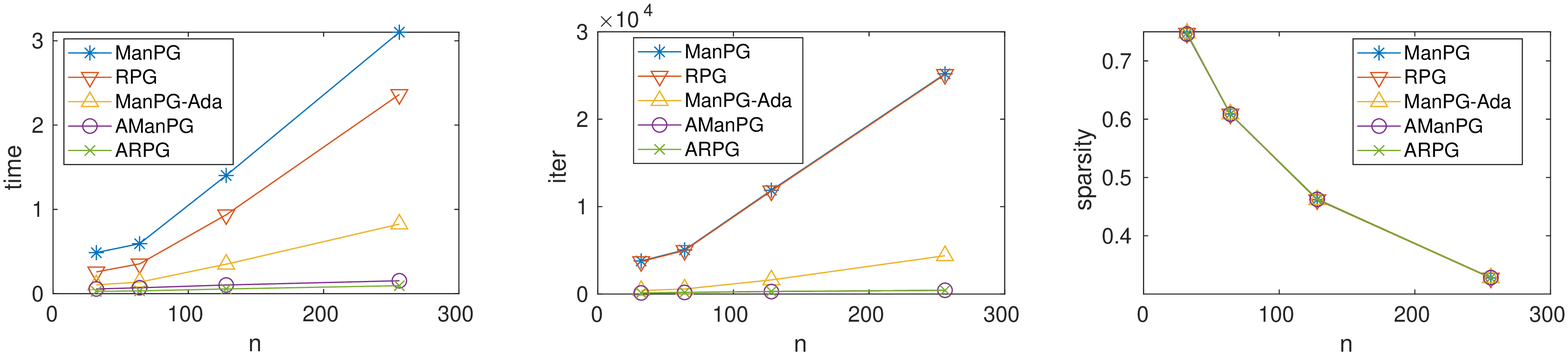} \\
}
\scalebox{1.0}{
\hspace{-0.12\textwidth}\includegraphics[width=1.2\textwidth]{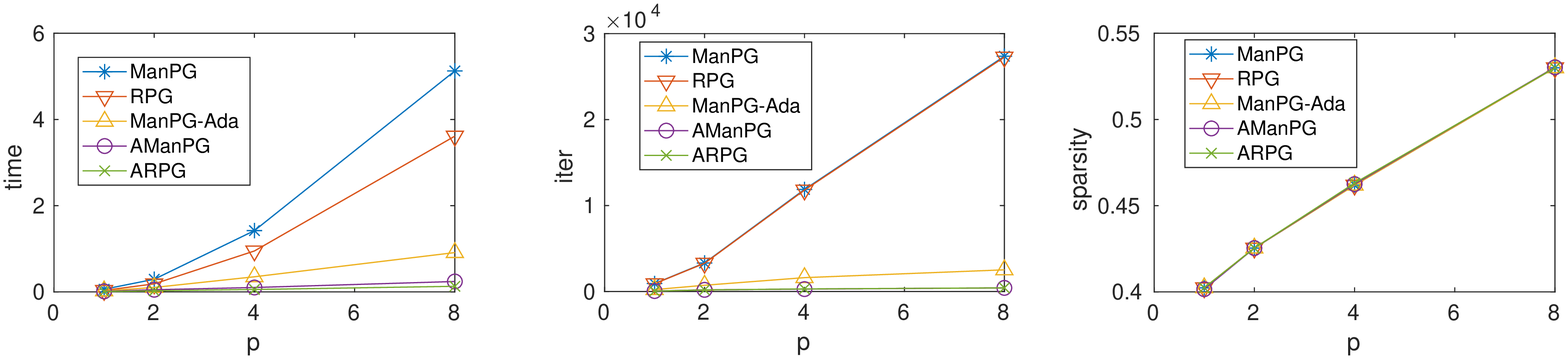} \\
}
\scalebox{1.0}{
\hspace{-0.12\textwidth}\includegraphics[width=1.2\textwidth]{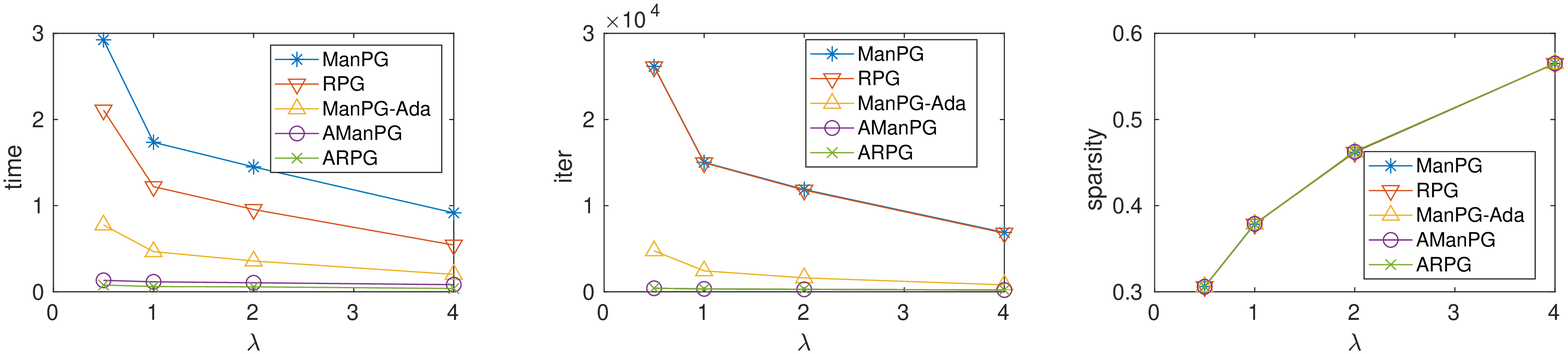}
}
\caption{
\whfirrev{}{Synthetic data: Average results of 10 random runs for the sparse PCA on the oblique manifold~\eqref{RPG:spcaob}. Top: multiple values $n = \{32, 64, 128, 256\}$ with $p = 4$, $m = 20$, and $\lambda = 2$; Middle: multiple values $p = \{1, 2, 4, 8\}$ with $n = 128$, $m = 20$, and $\lambda = 2$; Bottom: Multiple values $\lambda = \{0.5, 1, 2, 4\}$ with $n = 128$, $p = 4$, and $m = 20$.}
}
\label{RPG:SPCAOB_artificial}
\end{figure}

\begin{figure}
\centering
\includegraphics[width=0.8\textwidth]{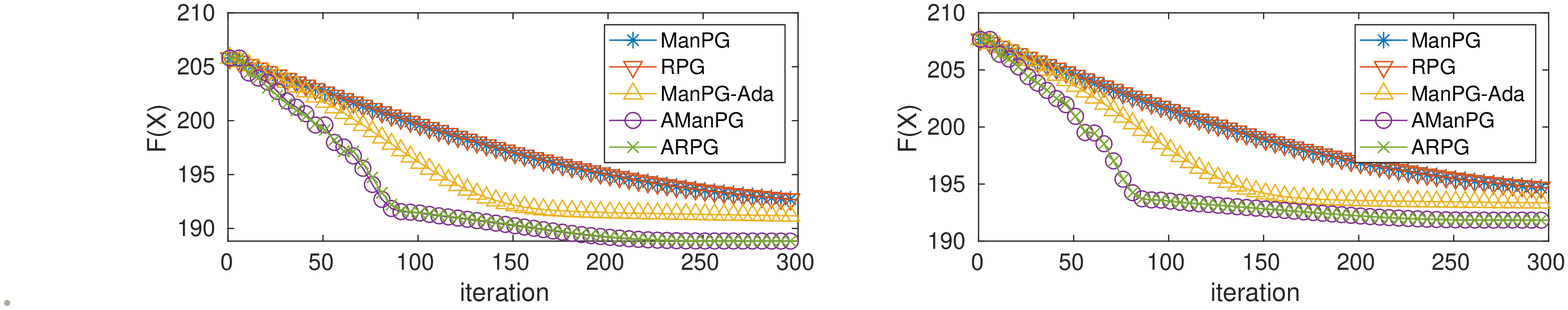}
\caption{
\whfirrev{}{Synthetic data: Comparison of the tested methods using two typical instances for the sparse PCA on the oblique manifold~\eqref{RPG:spcaob}. $n = 1024$, $p = 4$, $\lambda = 2$, $m = 20$.}
}
\label{RPG:SPCAOB_f_artificial}
\end{figure}

The comparisons are then repeated for the sparse PCA model on the Stiefel manifold, see \whfirrev{}{Figures~\ref{RPG:SPCAST}, ~\ref{RPG:SPCAST_f}, \ref{RPG:SPCAST_artificial}, and~\ref{RPG:SPCAST_f_artificial}} for the computational results.  %We note that the average number of inner iterations of Algorithm~\ref{RPG:a3} for solving the Riemannian proximal mapping~\eqref{RPG:subproblem2} is approximately three.
In this case, it is readily observed that the Riemannian proximal mappings~\eqref{RPG:subproblem2} and~\eqref{RPG:subprobChen} have different effects on the convergence of the algorithms. Figures~\ref{RPG:SPCAST_f} and~\ref{RPG:SPCAST_f_artificial} show that the Riemannian proximal gradient methods with~\eqref{RPG:subproblem2}  need fewer number of iterations to converge than those  with~\eqref{RPG:subprobChen}.
%used in~\cite{CMSZ2019} perform noticeably differently in the sense that the Riemannian gradient methods with~\eqref{RPG:subpro blem2} only need smaller number of iterations than those with~\eqref{RPG:subprobChen} to get similarly accuracy.
However, the Riemannian proximal gradient methods with~\eqref{RPG:subproblem2} are more costly \whsecrev{}{since the subproblem is solved by Algorithm~\ref{RPG:a3} that may involve multiple runnings of semi-smooth Newton algorithms.} %the inverse of the exponential mapping, the inverse of the differentiated exponential mapping and the inverse of the adjoint operator of the differentiated exponential mapping do not have closed-form solutions. 
%Instead, we must resort to iterative methods to solve them, which dominates the computational time of the algorithms. 
Therefore, for problem~\eqref{RPG:spcast}, using the new Riemannian proximal mapping~\eqref{RPG:subproblem2} can  reduce the number of iterations required for the algorithms to converge, but will increase the overall computational time due to the excessive cost for solving the new Riemannian proximal mapping.

\begin{figure}
\scalebox{1.0}{
\hspace{-0.12\textwidth}\includegraphics[width=1.2\textwidth]{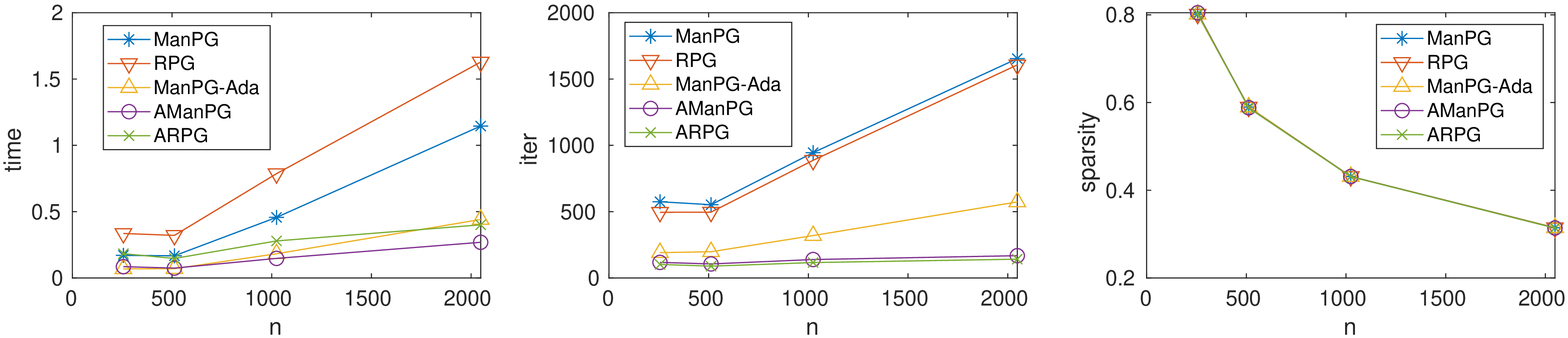} \\
}
\scalebox{1.0}{
\hspace{-0.12\textwidth}\includegraphics[width=1.2\textwidth]{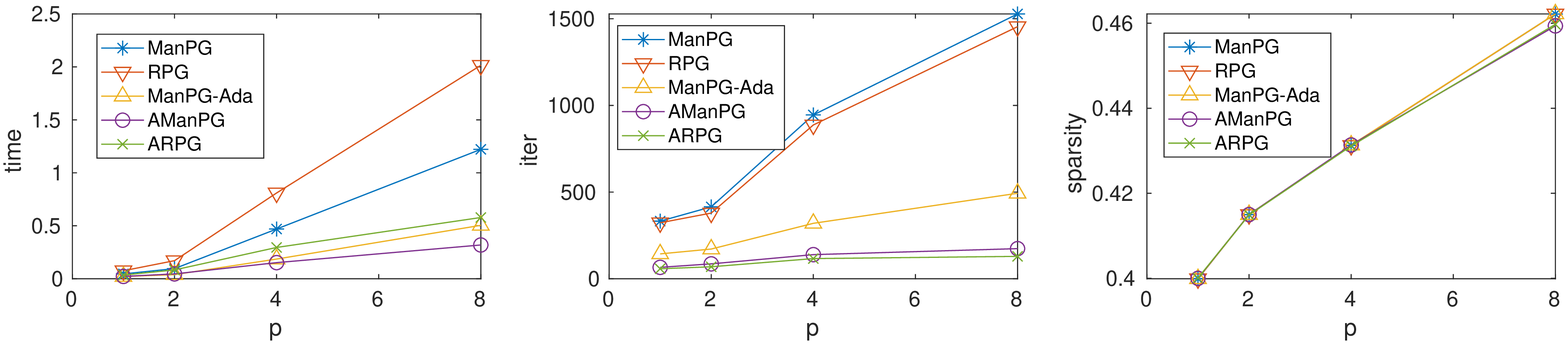} \\
}
\scalebox{1.0}{
\hspace{-0.12\textwidth}\includegraphics[width=1.2\textwidth]{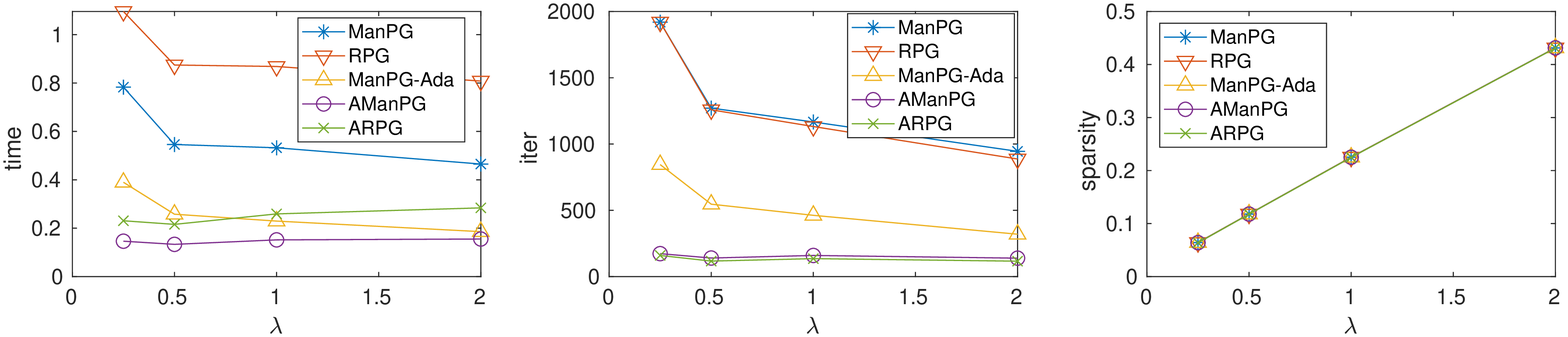}
}
\caption{
\whfirrev{}{Random data:} Average results of 10 random runs for the sparse PCA on the Stiefel manifold~\eqref{RPG:spcast}. Top: multiple values $n = \{256, 512, 1024, 2048\}$ with $p = 4$, $m = 20$, and $\lambda = 2$; Middle: multiple values $p = \{1, 2, 4, 8\}$ with $n = 1024$, $m = 20$, and $\lambda = 2$; Bottom: Multiple values $\lambda = \{0.25, 0.5, 1, 2\}$ with $n = 1024$, $p = 4$, and $m = 20$.
}
\label{RPG:SPCAST}
\end{figure}

\begin{figure}
\centering
\includegraphics[width=0.99\textwidth]{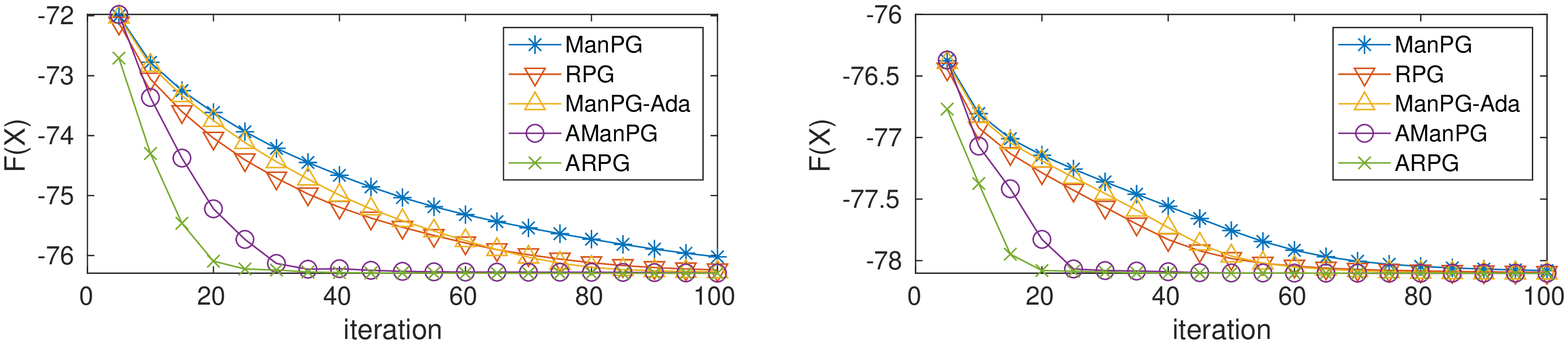}
\caption{
\whfirrev{}{Random data:} Comparison of the tested methods using two typical instances for the sparse PCA on the Stiefel manifold~\eqref{RPG:spcast}. $n = 1024$, $p = 4$, $\lambda = 2$, $m = 20$.
}
\label{RPG:SPCAST_f}
\end{figure}
\begin{figure}
\scalebox{1.0}{
\hspace{-0.12\textwidth}\includegraphics[width=1.2\textwidth]{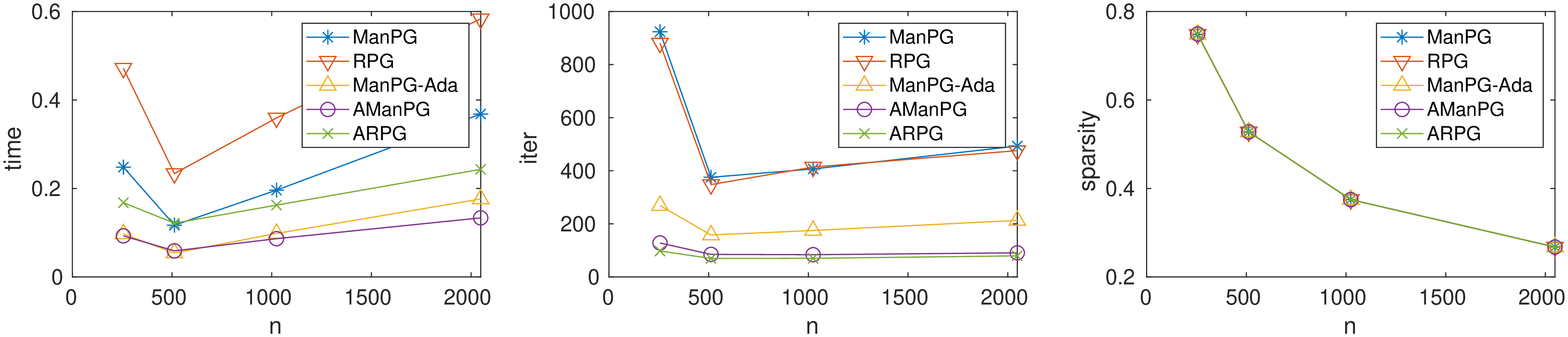} \\
}
\scalebox{1.0}{
\hspace{-0.12\textwidth}\includegraphics[width=1.2\textwidth]{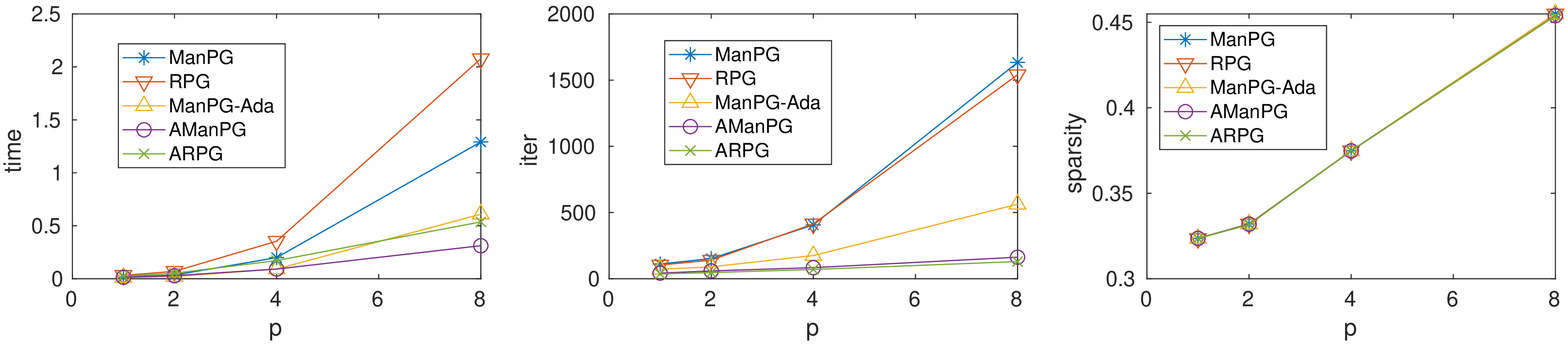} \\
}
\scalebox{1.0}{
\hspace{-0.12\textwidth}\includegraphics[width=1.2\textwidth]{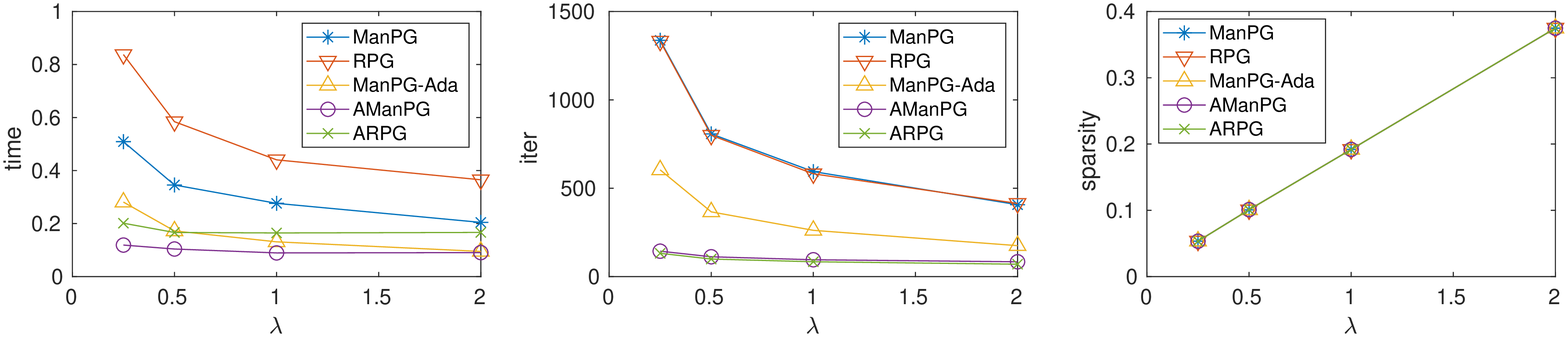}
}
\caption{
\whfirrev{}{Synthetic data: Average results of 10 random runs for the sparse PCA on the Stiefel manifold~\eqref{RPG:spcast}. Top: multiple values $n = \{256, 512, 1024, 2048\}$ with $p = 4$, $m = 20$, and $\lambda = 2$; Middle: multiple values $p = \{1, 2, 4, 8\}$ with $n = 1024$, $m = 20$, and $\lambda = 2$; Bottom: Multiple values $\lambda = \{0.25, 0.5, 1, 2\}$ with $n = 1024$, $p = 4$, and $m = 20$.}
}
\label{RPG:SPCAST_artificial}
\end{figure}

\begin{figure}
\centering
\includegraphics[width=0.99\textwidth]{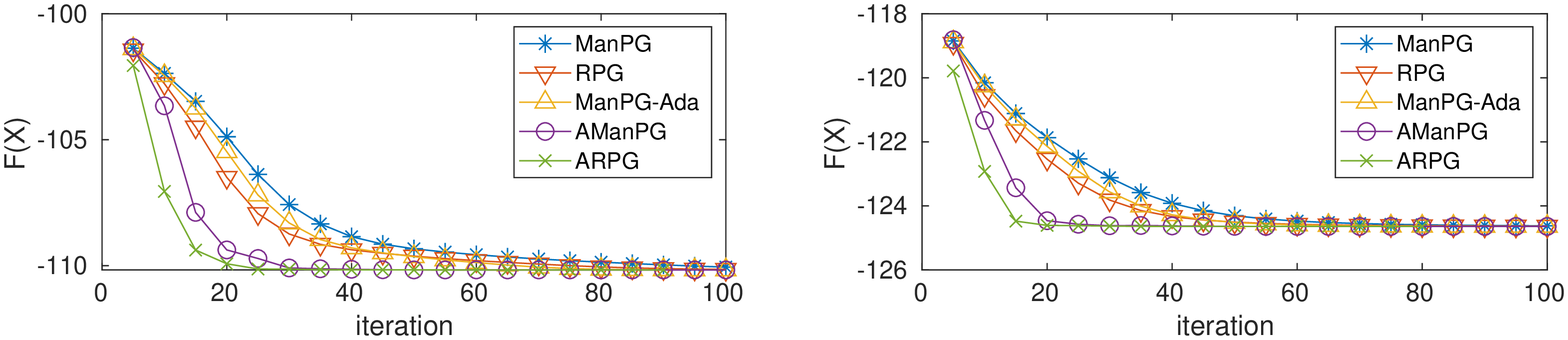}
\caption{
\whfirrev{}{Synthetic data: Comparison of the tested methods using two typical instances for the sparse PCA on the Stiefel manifold~\eqref{RPG:spcast}. $n = 1024$, $p = 4$, $\lambda = 2$, $m = 20$.}
}
\label{RPG:SPCAST_f_artificial}
\end{figure}

\section{Conclusion and Future Work}\label{RPG:sect:Con}

In this paper we propose a  Riemannian proximal gradient method %stated in Algorithm~\ref{RPG:a1}
as well as its accelerated %stated in Algorithm~\ref{RPG:a2}
 for solving nonsmooth optimization problems on a Riemannian manifold. %These methods are based on the  Riemannian proximal mapping that allows us to establish the convergence rate  of the proposed methods ($O(1/k)$ for Riemannian proximal gradient and $O(1/k^2)$ for Riemannian FISTA) under  certain reasonable assumptions.
 %Convergence rate analyses for Riemannian proximal gradient method $O(1/k)$ and Riemannian FISTA method $O(1/k^2)$ are given, respectively.
\whfirrev{}{Convergence analysis has been established for the Riemannian proximal gradient method. In particular,  the convergence analysis based on the Riemannian KL property is provided, which applies to the sparse PCA problem. }
A practical Riemannian proximal gradient method is also constructed  which guarantees the global convergence under the minimum requirements, and at the same time can achieve an empirical accleration.  We compare our methods with the Riemannian proximal gradient methods in~\cite{CMSZ2019} and~\cite{HuaWei2019} using two  optimization problems from sparse PCA. Numerical results show that our methods are superior in terms of the number of iterations for both the optimization problems, and they are also superior in terms of the runtime for the optimization problem  on the oblique manifold. However, for the optimization problem  on the Stiefel manifold the Riemannian proximal methods in~\cite{CMSZ2019} and~\cite{HuaWei2019} have the advantage of solving the Riemannian proximal mapping more efficiently, hence are faster. %Two optimization problems from sparse PCA are used to demonstrate the performance of the Riemannian proximal gradient method and the pratical Riemannian proximal gradient method with acceleration. We compare those methods with the Riemannian proximal gradient methods in~\cite{CMSZ2019} and~\cite{HuaWei2019}.

%It is observed that the performance of the Riemannian proximal gradient methods largely relies on the efficiency of solving the Riemannian proximal mapping~\eqref{RPG:subproblem2}. In particular, solving the Riemannian proximal mapping on the Stiefel manifold with exponential mapping is costly due to the excessive cost on computing the inverse exponential mapping, inverse of the differentiated exponential mapping and inverse of the adjoint operator of the differentiated exponential mapping. One future research direction is to develop a Riemannian proximal gradient framework with convergence rate analysis which has less restrictions on retraction, such that other retractions, for instance, polar and qf retractions, can be used.

As suggested by the numerical experiments, the efficacy of the proposed methods hinges  on the  efficient solution to the Riemannian proximal mapping. For future work we will look for new algorithms for solving the Riemannian proximal mapping, possibly those based on different retractions and vector transports. In this paper numerical tests focus primarily on optimization problems based on the embedded submanifolds. It is also interesting to see how the algorithms work for other manifolds, for example the Grassman manifolds.
\whfirrev{}{On the theoretical side, we would like to study  the  convergence behavior of the accelerated Riemannian proximal gradient methods. In addition, Theorem~\ref{RPG:th5} shows that the restriction of a semialgebraic function onto the Stiefel manifold satisfies the Riemannian KL property and there exists a $\theta\in(0,1]$ such that the corresponding 
desingularising function has the form $\frac{C}{\theta}t^{\theta}$. For future work it will be interesting to   calculate the exact value of $\theta$ for typical applications problems.  }
\section*{Acknowledgments}
The authors would like to thank Zirui Zhou for fruitful discussions on the KL property, and thank Shiqian Ma for kindly sharing their codes with us.

\bibliographystyle{plain}
\kw{reference WH19 needed to be edited later}
\bibliography{WHlibrary}

\appendix

\section{Proofs of Lemmas~\ref{RPG:le16} and \ref{RPG:le19}}\label{app:proofs}
\subsection{Proof of Lemma~\ref{RPG:le16}}
\begin{proof}
{
%For any $x \in \bar{\Omega}$, there exists a positive constant $\varrho_x$ and a neighborhood $\mathcal{U}_x$ of $x$ such that $\mathcal{U}_x$ is a totally restrictive set with respect to $\varrho_x$. Since $\bar{\Omega}$ is compact, there exists finite number of $x_i$ such that their totally restrictive sets covering $\bar{\Omega}$, i.e., $\cup_{i = 1}^t \mathcal{U}_{x_i} \supset \bar{\Omega}$. Let $\delta_T =  \min(\varrho_{x_i}, i = 1, \ldots, t)$. We have that for any $x \in \bar{\Omega}$, the retraction $R$ is a diffeomorphism on $\mathbb{B}(x, \delta_T)$. 

Since $R$ is smooth and therefore $C^2$, the mapping $m: \T \mathcal{M} \times \mathbb{R} \rightarrow \T \mathcal{M}: (\eta, t) \mapsto \frac{D}{d t} \frac{d}{d t} R\left(t \eta \right) $ is continuous where $\frac{D}{d t}$ denotes the covariant derivative along the curve $t \mapsto R(t \eta)$, see definition of covariant derivative in e.g.,~\cite[Proposition~2.2]{dC92}. 
%Since $\mathcal{V}$ is compact, it holds that the subset of $\T \mathcal{M}$, defined by $\mathcal{S} = \{R_x^{-1}(y) \mid x \in \mathcal{V}, y \in \mathcal{V}\}$, is compact. 
%It follows that $T_{\mathcal{S}} = \inf_t \{t \mid t \geq \|\eta\|, \forall \eta \in \mathcal{S}\} < \infty$. 
%Therefore, 
In addition, since the set $\mathcal{D} = \{ (\eta_x, t) \mid x \in \bar{\Omega}, \|\eta_x\|_x = 1, 0 \leq t \leq \delta_T \}$ is compact, there exists a positive constant $b_2$ such that 
\begin{equation} \label{RPG:e87}
\|m(\eta, t)\| \leq b_2
\end{equation}
for all $(\eta, t) \in \mathcal{D}$.

If $\eta_x = 0_x$, then the conclusion holds. Otherwise, let $\tilde{\eta}_x = \eta_x / \|\eta_x\|_x$. Since $\dist(x, y)$ is the shortest distance of a curve connecting $x$ and $y$, we have
\begin{align} \label{RPG:e88}
\dist(x, y) \leq& \int_0^{\|\eta_x\|_x} \left\|\frac{d}{d t} R_{x} \left(t \tilde{\eta}_x\right)\right\|_{R_x(t \tilde{\eta}_x)} d t,
\end{align}
where the right side is the length of the curve $R_x(t \eta_x)$. Using the Cauchy-Schwarz inequality and the invariance of the metric by the Riemannian affine connection, we have
\begin{align*}
\left| \frac{d}{d t} \left\|\frac{d}{d t} R_{x} (t \tilde{\eta}_x)\right\| \right| =& \left| \frac{d}{d t} \sqrt{ \inner[]{\frac{d}{d t} R_{x} (t \tilde{\eta}_x)}{\frac{d}{d t} R_{x} (t \tilde{\eta}_x)} } \right|
= \left| \frac{ \inner[]{ \frac{D}{d t} \frac{d}{d t} R_{x} (t \tilde{\eta}_x) }{ \frac{d}{d t} R_{x} (t \tilde{\eta}_x) } }{ \left\| \frac{d}{d t} R_{x} (t \tilde{\eta}_x) \right\| } \right| \\
\leq& \left\| \frac{D}{d t} \frac{d}{d t} R_{x} (t \tilde{\eta}_x) \right\| \leq b_2. \qquad \hbox{ (by~\eqref{RPG:e87})}
\end{align*}
It follows that
\begin{equation} \label{RPG:e89}
\int_0^{\|\eta_x\|_x} \left\|\frac{d}{d t} R_{x} (t \tilde{\eta}_x)\right\|_{R_x(t \eta_x)} d t \leq \int_0^{\|\eta_x\|_x} (1 + b_2 t) d t = \|\eta_x\|_x + \frac{b_2}{2} \|\eta_x\|_x^2 \leq b_3 \|\eta_x\|_x,
\end{equation}
where $b_3 = 1 + b_2 \delta_T / 2$. Combining~\eqref{RPG:e88} and~\eqref{RPG:e89} yields the result.
}
\end{proof}
%%%
\subsection{Proof of Lemma~\ref{RPG:le19}}
\begin{proof}

{
For any $x \in \bar{\Omega}$, there exists a positive constant $\varrho_x$ and a neighborhood $\mathcal{U}_x$ of $x$ such that $\mathcal{U}_x$ is a totally restrictive set with respect to $\varrho_x$. Since $\bar{\Omega}$ is compact, there exists finite number of $x_i$ such that their totally restrictive sets covering $\bar{\Omega}$, i.e., $\cup_{i = 1}^t \mathcal{U}_{x_i} \supset \bar{\Omega}$. Let $\delta = \frac{1}{2} \min(\varrho_{x_i}, i = 1, \ldots, t)$. We have that for any $x \in \bar{\Omega}$, the retraction $R$ is a diffeomorphism on $\mathbb{B}(x, 2 \delta)$. Therefore, $\mathcal{T}_{R_{\eta_x}}^{\sharp}$ is invertible for any $\eta_x$ satisfying $\|\eta_x\|_x < 2 \delta$.

Since $\mathcal{T}_{R_{\eta_x}}^{-\sharp}$ is smooth with respect to $\eta_x$ and the set $\{\eta_x \mid x \in \bar\Omega, \|\eta_x\| \leq \delta\}$ is compact, there exists a constant $L_t > 0$ such that
\begin{equation} \label{RPG:e95}
\|\mathcal{T}_{R_{\eta_x}}^{-\sharp}\| \leq L_t, \forall \eta_x \in \{\eta_x \mid x \in \bar\Omega, \|\eta_x\| \leq \delta\}.
\end{equation}
By Lemma~\ref{RPG:le16}, there exists a positive constant $\kappa$ such that 
\begin{equation} \label{RPG:e94}
\dist(x, R_x(\eta_x)) \leq \kappa \|\eta_x\|_x
\end{equation}
for all $x \in \bar{\Omega}$ and for all $\eta_x \in \mathcal{B}(0_x, \delta)$. Let $\tilde{\delta} = \min(\delta, i(\bar{\Omega}) / \kappa)$. For all $\eta_x \in \mathcal{B}(0_x, \tilde{\delta})$ it holds that
\begin{equation} \label{RPG:e93}
\dist(x, R_x(\eta_x)) \leq \kappa \|\eta_x\|_x \leq i(\bar{\Omega}).
\end{equation}
By the definition of locally Lipschitz continuity of a vector field, we have 
$\|\mathcal{P}_{\gamma}^{0 \leftarrow 1} \xi_y - \xi_x\|_x \leq L_v \dist(x, y)$
for any $x, y \in \bar\Omega$ and $\dist(x, y) < i(\bar{\Omega})$. Since the parallel translation is isometric, it holds that $\|\xi_y - \mathcal{P}_{\gamma}^{1 \leftarrow 0} \xi_x\|_y \leq L_v \dist(x, y)$. Using~\eqref{RPG:e94} and~\eqref{RPG:e93} yields
\begin{equation} \label{RPG:e90}
\|\xi_y - \mathcal{P}_{\gamma}^{1 \leftarrow 0} \xi_x\|_x \leq L_v \dist(x, y) \leq L_v \kappa \|\eta_x\|_x,
\end{equation}
for all $\eta_x \in \mathcal{B}(0_x, \tilde{\delta})$, where $y = R_x(\eta_x)$.
%If $\xi$ is locally Lipschitz continuous, then by the definition we have that
%$\|\mathcal{P}_{\gamma}^{0 \leftarrow 1} \xi_y - \xi_x\|_x \leq L_v \dist(x, y)$
%for any $x, y \in \bar\Omega$ and $\dist(x, y) < i(\bar{\Omega})$. Since the parallel translation is isometric, it holds that $\|\xi_y - \mathcal{P}_{\gamma}^{1 \leftarrow 0} \xi_x\|_y \leq L_v \dist(x, y)$. 
%
%By Lemma~\ref{RPG:le16}, it holds that
%\|\xi_y - \mathcal{P}_{\gamma}^{1 \leftarrow 0} \xi_x\|_x \leq L_v \dist(x, y) \leq L_v \kappa \|\eta_x\|_x.
%

By~\cite[Lemma~3.5]{HGA2014}, for any $\bar{x} \in \mathcal{M}$, there exists a neighborhood $\mathcal{U}_{\bar{x}}$ of $\bar{x}$ and a positive number $L_{\bar{x}}$ such that for all $x, y \in \mathcal{U}_{\bar{x}}$ it holds that
\[
\|\mathcal{P}_{\gamma}^{1 \rightarrow 0} \xi_x - \mathcal{T}_{\eta_x}^{-\sharp} \xi_x\|_y \leq L_x \|\xi_x\|_x \|\eta_x\|_x.
\]
Since $\bar\Omega$ is compact, there exist finite number of $\bar{x}$, denoted by $\bar{x}_1, \ldots, \bar{x}_t$, such that $\cup_{i=1}^t \mathcal{U}_{\bar{x}_i} \supset \Omega$. Let $L_{cc}$ denote $\max(L_{\bar{x}_i}, i = 1, \ldots, t)$, and $\sigma = \sup_{r}\left\{ r \in \mathbb{R} \mid \exists i, \hbox{ such that } \mathbb{B}(z, r) \subseteq \mathcal{U}_{\bar{x}_i} \forall z \in \bar\Omega \right\}$. Since the number of $\bar{x}_i$ is finite, we have $L_{cc} < \infty$ and $\sigma > 0$. Therefore, for any $x, y \in \bar\Omega$ satisfying $\dist(x, y) < \sigma$, it holds that
\begin{equation} \label{RPG:e106}
\|\mathcal{P}_{\gamma}^{1 \rightarrow 0} \xi_x - \mathcal{T}_{\eta_x}^{-\sharp} \xi_x\|_y \leq L_{cc} \|\xi_x\|_x \|\eta_x\|_x.
\end{equation}
Note that $\|\eta_x\|_x < \sigma / \kappa$ implies $\dist(x, y) < \sigma$ by~\eqref{RPG:e94}.
It follows from~\eqref{RPG:e95}, ~\eqref{RPG:e90} and~\eqref{RPG:e106} that for any $x, y \in \bar\Omega$ satisfying $\|\eta_x\|_x < \min(\sigma / \kappa, \tilde{\delta})$,
\[
\|\xi_y - \mathcal{T}_{\eta_x}^{-\sharp} (\xi_x + a \eta_x)\|_y \leq \|\xi_y - \mathcal{P}_{\gamma}^{1 \leftarrow 0} \xi_x\|_y + \|\mathcal{P}_{\gamma}^{1 \leftarrow 0} \xi_x - \mathcal{T}_{\eta_x}^{-\sharp} \xi_x\|_y + \|\mathcal{T}_{\eta_x}^{-\sharp} a \eta_x\|_y \leq L_c \|\eta_x\|_x,
\]
where $L_{c} = L_v \kappa + L_{cc} \sup_{x \in \bar\Omega} \|\xi_x\|_x + a L_t$.
}

\end{proof}
%%%%%%%%%%%%
\section{Proof of Theorem~\ref{RPG:LocalRateKL}}\label{appen:localKL}
\begin{proof}
Applying $\varsigma(t) = \frac{C}{\theta} t^{\theta}$ to~\eqref{RPG:e99} yields
\begin{equation} \label{RPG:e100}
\|\eta_{x_{k}}^*\|_{x_{k}}^2 \leq \|\eta_{x_{k-1}}^*\|_{x_{k-1}} \frac{C L_c}{\theta \beta} \left( (F(x_k) - F(x_*))^{\theta} - (F(x_{k+1} - F(x_*)))^{\theta} \right), \forall k > \hat{l}.
\end{equation}
Taking square root to the both sides of~\eqref{RPG:e100} and noting $2 \sqrt{a b} \leq a + b$ for all $a, b\geq 0$, we have
\begin{equation} \label{RPG:e101}
2 \|\eta_{x_{k}}^*\|_{x_{k}} \leq \|\eta_{x_{k-1}}^*\|_{x_{k-1}} + \frac{C L_c}{\theta \beta} \left( (F(x_k) - F(x_*))^{\theta} - (F(x_{k+1} - F(x_*)))^{\theta} \right), \forall k > \hat{l}.
\end{equation}
Summing the both sides from $p > \hat{l}$ to $\infty$ yields
\begin{equation} \label{RPG:e103}
\sum_{k = p}^{\infty} \|\eta_{x_k}^*\|_{x_k} \leq \|\eta_{x_{p-1}}^*\|_{x_{p-1}} + \frac{C L_c}{\theta \beta} (F(x_p) - F(x_*))^{\theta}.
\end{equation}
By~\eqref{eq:kwrev01}, we have $ \frac{1}{C} (F(x_k) - F(x_*))^{1 - \theta} \leq \dist(0, \partial F(x_k)) $. Combining this inequality with~\eqref{RPG:e66} yields 
\begin{equation} \label{RPG:e102}
\frac{1}{C} (F(x_k) - F(x_*))^{1 - \theta} \leq L_c \|\eta_{x_{k-1}}^*\|_{x_{k-1}}.
\end{equation}
It follows from~\eqref{RPG:e103} and~\eqref{RPG:e102} that
\begin{equation} \label{RPG:e104}
\sum_{k = p}^{\infty} \|\eta_{x_k}^*\|_{x_k} \leq \|\eta_{x_{k-1}}^*\|_{x_{k-1}} + \frac{C L_c}{\theta \beta} \left( C L_c \|\eta_{x_{k-1}}^*\|_{x_{k-1}} \right)^{\frac{\theta}{1 - \theta}}, \quad \forall p > \hat{l}.
\end{equation}
Define $\Delta_k = \sum_{i = k}^{\infty} \|\eta_{x_i}^*\|_{x_i}$. Therefore, inequality~\eqref{RPG:e104} becomes
\begin{equation} \label{RPG:e105}
\Delta_k \leq (\Delta_{k-1} - \Delta_k) + b_1 (\Delta_{k-1} - \Delta_k)^{\frac{\theta}{1 - \theta}}, \forall k > \hat{l},
\end{equation}
where $b_1 = \frac{C L_c}{\theta \beta} (C L_c)^{\frac{\theta}{1 - \theta}}$. Noting that~\eqref{RPG:e105} has the same form as~\cite[(12)]{AB2009}, we can follow the same derivations in~\cite[Theorem~2]{AB2009} and show that (i) if $\theta = 1$, then Algorithm~\ref{RPG:a1} terminates in finite steps; (ii) if $\theta \in [\frac{1}{2}, 1)$, then $\Delta_k < C_r d^k$ for $C_r > 0$ and $d \in (0, 1)$; and (iii) if $\theta \in (0, \frac{1}{2})$, then $\Delta_k < \tilde{C}_r k^{\frac{1}{1 - 2 \theta}}$ for $\tilde{C}_r > 0$. It only remains to show that $\dist(x_k, x_*) < C_p \Delta_k$ for a positive constant $C_p$. This can be obtained by
\[
\dist(x_k, x_*) \leq \sum_{i = k}^{\infty} \dist(x_k, x_{k+1}) \leq \kappa \sum_{i = k}^{\infty} \|\eta_{x_k}^*\|_{x_k} = \kappa \Delta_k,
\]
where the first inequality is by triangle inequality and the second inequality is from Lemma~\ref{RPG:le16}. This completes the proof.
\end{proof}

%%%%%%%%%%%%
\section{Operations on Stiefel Manifold}\label{app1}

%\subsection{The exponential mapping under the canonical metric}
\subsection{Differentiated Retraction of the retraction by polar decomposition}

The retraction by polar decomposition is given by
\begin{equation} \label{RPG:e107}
R_X(\eta_X) = (X + \eta_X) (I_p + \eta_X^T \eta_X)^{-1/2}.
\end{equation}
The vector transport by differentiated the retraction~\eqref{RPG:e107} is given in~\cite[Lemma~10.2.1]{HUANG2013} by
\begin{equation} \label{RPG:e108}
\mathcal{T}_{\eta_X} \xi_X = Y \Omega + (I_n - Y Y^T) \xi_X (Y^T (X + \eta_X))^{-1},
\end{equation}
where $Y = R_X(\eta_X)$ and $\Omega$ is the solution of the Sylvester equation $(Y^T(X + \eta_X)) \Omega + \Omega (Y^T(X + \eta_X)) = Y^T \xi_X - \xi_X^T Y$.

\subsection{Inverse Differentiated Retraction of the retraction by polar decomposition}

\begin{lemma} \label{app:le2}
The inverse differentiated retraction of~\eqref{RPG:e107} is
\begin{equation}
\mathcal{T}_{{\eta_X}}^{-1} \zeta_Y = YA + P,
\end{equation}
where $Y = R_X(\eta_X)$, $P = (I_n - Y Y^T) \zeta_Y (Y^T(X + \eta_X))$ and $A$ is the solution of the Sylvester equation $X^T Y A + A Y^T X = [(Y^T \zeta_Y) (Y^T (X + \eta_X)) + (Y^T (X + \eta_X)) (Y^T \zeta_Y)] Y^T X - X^T P - P^T X$.
\end{lemma}
\begin{proof}
Let $\xi_X$ denote $\mathcal{T}_{{\eta_X}}^{-1} \zeta_Y$ and $Y_{\perp}$ denote a $n$-by-$(n-p)$ orthonormal matrix such that $[Y \, Y_{\perp}]^T [Y \, Y_{\perp}] = I_n$. Therefore, there exist matrices $A \in \mathbb{R}^{p \times p}$ and $P \in \mathbb{R}^{n \times (n - p)}$ such that $\xi_X = YA + P$, where $\mathrm{span}(P) \subseteq \mathrm{span}(Y_{\perp})$.

Substituting $\xi_X = YA + P$ into~\eqref{RPG:e108} and using $\mathcal{T}_{\eta_X} \xi_X = \zeta_Y$, we have
\[
\zeta_Y = Y \Omega + (I_n - Y Y^T) P (Y^T (X + \eta_X))^{-1}.
\]
It follows that $P = (I_n - Y Y^T) \zeta_Y (Y^T(X + \eta_X))$ and 
\begin{equation} \label{RPG:e109}
	A - A^T = (Y^T \zeta_Y) (Y^T (X + \eta_X)) + (Y^T (X + \eta_X)) (Y^T \zeta_Y).
\end{equation}
Since $\xi_X \in \T_X \St(p, n)$, we have $X^T \xi_X + \xi_X^T X = 0$. Combining this equation with~\eqref{RPG:e109} yields that
\[
X^T Y A + A Y^T X = [(Y^T \zeta_Y) (Y^T (X + \eta_X)) + (Y^T (X + \eta_X)) (Y^T \zeta_Y)] Y^T X - X^T P - P^T X.
\]

\end{proof}

\subsection{Adjoint of the Inverse Differentiated Retraction of the retraction by polar decomposition}

%We consider the adjoint operator with respect to the canonical metric. Note that the adjoint operator of the inverse differentiated retraction is the same as the inverse of the adjoint operator of differentiated retraction. Therefore, we use~\eqref{app:e8} to derive the desired operator.
\begin{lemma} \label{app:le5}
The adjoint operator of the inverse differentiated retraction is
\[
\mathcal{T}_{\eta_X}^{- \sharp} \xi_X = Y [B (X^T Y) (Y^T(X + \eta_X)) + (Y^T (X + \eta_X)) B (X^T Y) ] - (I_n - Y Y^T) (X B + X B^T - \xi_X) (Y^T (X + \eta_X)),
\]
where $B$ is the solution of the Sylvester equation $ Y^T X B + B X^T Y = Y^T \xi_X$.
\end{lemma}
\begin{proof}
For any $\xi_X \in \T_X \St(p, n)$, we have
\begin{align*}
\inner[]{\xi_X}{\mathcal{T}_{{\eta_X}}^{-1} \zeta_Y}	 =& \inner[]{\xi_X}{Y A + P} = \inner[]{\xi_X}{Y A + (I_n - Y Y^T) \zeta_Y (Y^T(X + \eta_X))} \\
=& \inner[]{Y^T \xi_X}{A} + \inner[]{(I_n - Y Y^T) \xi_X (Y^T(X + \eta_X))^T }{ \zeta_Y } \\
=& \inner[]{B}{ [(Y^T \zeta_Y) (Y^T (X + \eta_X)) + (Y^T (X + \eta_X)) (Y^T \zeta_Y)] Y^T X - X^T P - P^T X } \\
&+ \inner[]{(I_n - Y Y^T) \xi_X (Y^T(X + \eta_X))^T }{ \zeta_Y } \\
=& \inner[]{Y B (Y^T X)^T (Y^T (X + \eta_X)))^T}{\zeta_Y} + \inner[]{Y (Y^T (X + \eta_X)))^T B (Y^T X)^T }{\zeta_Y} \\
&- \inner[]{ (I_n - Y Y^T) X (B + B^T) (Y^T (X + \eta_X))^T }{ \zeta_Y } + \inner[]{(I_n - Y Y^T) \xi_X (Y^T(X + \eta_X))^T }{ \zeta_Y } 
\end{align*}
where $B$ is the solution of the Sylvester equation $ Y^T X B + B X^T Y = Y^T \xi_X$. The conclusion follows from the above equation.
\end{proof}

\section{Solution of a Proximal Subproblem on Unit Sphere}

\begin{lemma} \label{RPG:le14}
For any $x \in \mathbb{R}^n$ and $\lambda > 0$, the minimizer of the  optimization problem
\begin{equation} \label{RPG:e62}
\min_{y\in\mathbb{S}^{n-1}} \frac{1}{2\lambda} \|y- x\|_2^2 + \|y\|_1
\end{equation}
is given by
\begin{equation}
y_* =
\left\{
\begin{array}{ll}
\frac{z}{\|z\|_2},  & \hbox{ if $\|z\|_2 \neq 0$; } \\
\sign(x_{i_{\max}}) e_{i_{\max}} & \hbox{ otherwise, }
\end{array}
\right.
\end{equation}
where $i_{\max}$ is the index of the largest magnitude entry of $x$ (break ties arbitrarily), $e_i$ denotes the $i$-th column in the canonical basis of $\mathbb{R}^n$, and $z$ is defined by
$$
z_i =
\left\{
\begin{array}{ll}
0 & \hbox{ if $|x_i| \leq \lambda$; } \\
x_i - \lambda & \hbox{ if $x_i > \lambda$; } \\
x_i + \lambda & \hbox{ if $x_i < - \lambda$.} \\
\end{array}
\right.
$$
\end{lemma}
\begin{proof} Since $\|y\|_2=1$ for any $y\in\mathbb{S}^{n-1}$, the optimization problem \eqref{RPG:e62} is equivalent to
\begin{align}\label{eq:ke004}
\min_{y\in\mathbb{S}^{n-1}} u(y), \hbox{ where }u(y)= -\frac{1}{\lambda}y^Tx+\|y\|_1.
\end{align}
The subdifferential of the cost function in~\eqref{eq:ke004}, denoted by $\partial u(y)$, is
\begin{equation}
\partial u(y) = - \frac{1}{\lambda}x + \partial \|y\|_1,
\end{equation}
where $\partial \|y\|_1$ is given by
\begin{equation*}
\left(\partial \|y\|_1\right)_i =
\left\{
\begin{array}{ll}
 1 & \hbox{ if ${y}_i > 0$}; \\
- 1 & \hbox{ if ${y}_i < 0$}; \\
\hbox{$[-1, 1]$} & \hbox{ if ${y}_i = 0$}. \\
\end{array}
\right.
\end{equation*}
Assume $y$ is a critical point for \eqref{eq:ke004}. By the first order optimality condition on the unit sphere, there exists a subgradient at $y$, denoted $\nabla u(y)$, such that $\nabla u(y)$ is a multiple of $y$. In other words, there exists a constant $c$ such that
\begin{equation} \label{RPG:e63}
c y = x - \lambda \nabla \|y\|_1,
\end{equation}
where $\nabla \|y\|_1$ denotes a subgradient of $\|\cdot\|_1$ at $y$.

\paragraph{Case 1:} $\|x\|_\infty > \lambda$, where $\|x\|_\infty = \max_i(|x_i|)$.

If $c = 0$, then equation~\eqref{RPG:e63} can not hold due to the assumption that $\|x\|_\infty > \lambda$. If $c > 0$, then the corresponding critical point $y_*$ is unique and can be expressed as
\begin{equation} \label{RPG:e64}
y_* = z / \|z\|_2,
\end{equation}
where
$$
z_i =
\left\{
\begin{array}{ll}
0 & \hbox{ if $|x_i| \leq \lambda$; } \\
x_i - \lambda & \hbox{ if $x_i > \lambda$; } \\
x_i + \lambda & \hbox{ if $x_i < - \lambda$.} \\
\end{array}
\right.
$$
If $c < 0$, then there exist multiple critical points, which can be expressed as $v_* = w / \|w\|_2$, where
$$
w_i =
\left\{
\begin{array}{ll}
0 \hbox{ or } (- x_i + \lambda) \hbox{ or } (- x_i - \lambda) & \hbox{ if $|x_i| \leq \lambda$; } \\
- x_i - \lambda & \hbox{ if $x_i > \lambda$; } \\
- x_i + \lambda & \hbox{ if $x_i < - \lambda$.} \\
\end{array}
\right.
$$
One can easily verify that the global minimizer $y_*$ of~\eqref{RPG:e62} must have the same signs as $x$ in the sense that $(y_*)_i x_i \geq 0$ for all $i$. Otherwise, let us define
$$
\tilde{y}_* =
\begin{bmatrix}
(y_*)_1 & (y_*)_2 & \ldots & (y_*)_{j-1} & -(y_*)_j & (y_*)_{j+1} & \ldots, (y_*)_n
\end{bmatrix},
$$
where $(y_*)_j x_j < 0$. It follows that $u(\tilde{y}_*) < u(y_*)$, which conflicts with the global minimizer assumption of $y_*$. The only critical point that has the same sign as $x$ is $y_*$ in~\eqref{RPG:e64}. Therefore, $y_*$ is the global minimizer.

\paragraph{Case 2:}$\|x\|_\infty \leq \lambda$.

In this case we have
\begin{align*}
u(y) =& - \frac{1}{\lambda} y^T x + \|y\|_1 \geq - \frac{1}{\lambda} \|x\|_{\infty} \|y\|_1 + \|y\|_1 = \frac{1}{\lambda} \left( \lambda - \|x\|_\infty \right) \|y\|_1 \\
\geq& \frac{1}{\lambda} \left( \lambda - \|x\|_\infty \right) \|y\|_2 = \frac{1}{\lambda} \left( \lambda - \|x\|_\infty \right),
\end{align*}
where the equality  holds if $y = \sign(x_{i_{\max}}) e_{i_{\max}}$. Therefore, $\sign(x_{i_{\max}}) e_{i_{\max}}$ is a global minimizer.
\end{proof}

\end{document}